\documentclass{amsart}
\usepackage{preamble}

\title[The narrow escape problem in arbitrary dimension]{The narrow escape problem in arbitrary dimension}
\author[Carillo, Lelièvre, Normand, Stoltz, Vaes]{Louis Carillo\textsuperscript{1,2,3}, Tony Lelièvre\textsuperscript{1,2}, Thomas Normand\textsuperscript{2,1}, Gabriel Stoltz\textsuperscript{1,2}, and Urbain Vaes\textsuperscript{2,1}}
\thanks{\textsuperscript{1}CERMICS, CNRS, ENPC, Institut Polytechnique de Paris, Marne-la-Vallée, France. 
\textsuperscript{2}MATHERIALS project-team, Inria Paris, France. \textsuperscript{3}E-mail address: \href{mailto:louis.carillo@enpc.fr}{louis.carillo@enpc.fr}}

\begin{document}

\begin{abstract}
   The narrow escape problem is a prototypical example for studying entropic metastability, 
   motivated by the analysis of biological and chemical systems.
   The problem concerns the determination of the exit time and position of a Brownian particle trapped in a domain 
   with a reflecting boundary pierced by narrow holes. 
   Our goal is to investigate this problem in a general domain in any dimension (greater than or equal to two), using the quasi-stationary distribution approach to metastability.
   In particular, 
   we derive first order asymptotic expansions with error bounds of the mean exit time and the law of the exit position in the limit where the hole sizes tend to zero.
   Our analytical predictions are illustrated by numerical simulations, using dedicated Monte Carlo techniques.
\end{abstract}
\maketitle

\section{Introduction}
The \textit{narrow escape problem} consists in studying a Brownian particle diffusing in a domain with reflecting boundary, 
except for a few small absorbing holes. 
The goal is to characterize the exit time and exit position of the particle through one of the holes, 
in the limit where the hole sizes tend to zero.
This problem arises in the context of biological and chemical systems, 
where it is used to model the escape of an ion from a cell, or the binding of a ligand to a 
receptor~\cite{HolcmanSchuss04}, 
and has sparked a lot of interest in the physics community \cite{Singer2006, Benichou2008, GomezCheviakov2015, Grebenkov2024}.
A recent approach has been proposed in~\cite{Tony2024},
where the authors used the quasi-stationary distribution to study the exit time and exit position of the particle, 
which is a classic technique for metastable systems.
In this work, 
we propose a generalization of the results of~\cite{Tony2024} to arbitrary dimension and general domains.
This requires substantially new mathematical techniques compared to~\cite{Tony2024}.
The introduction is organized as follows:
we begin with the mathematical setting of the problem in~\cref{sec:setting}
and review the relevant literature in~\cref{sec:related_works};
we next introduce the quasi-stationary distribution and
explain why it is natural for studying the narrow escape problem in~\cref{sec:quasi_stationary}. 
We then present the quasi-mode used to approximate the quasi-stationary distribution in~\cref{sec:introquasi-mode}.
Our main results are stated in~\cref{sec:contribution}.
\label{sec:intro}
\subsection{Setting} 
\label{sec:setting}
We consider a smooth bounded domain~$\Omega \subset \mathbb{R}^d$ ($d \geq 2$) 
with $N$ points~$\xn{1}, \dotsc, \xn{N}$ on its boundary.
The goal of this work is to study the narrow escape problem on the domain $\Omega_\varepsilon$ defined as
\begin{equation} \label{eq:omega varepsilon}
    \Omega_\varepsilon = \Omega \setminus \bigcup_{k=1}^{\ndoors} \overline{B\Bigl(x^{(k)},\re{(k)}\Bigr)},
\end{equation}
where $\varepsilon > 0$ is a small parameter.
For $x \in \real^d$ and $r \in (0, \, + \infty)$, 
the notation $B(x, r)$ refers to the open ball of radius $r$ centered at $x$,
and $\re{(k)} \in (0, \, + \infty)$ is the radius of the $k$-th spherical hole in terms of~$\varepsilon$,
which is assumed to satisfy $\re{(k)} \to 0$ as $\varepsilon \to 0$.
See \cref{fig:domain} for an illustration of the domain~$\Omega_\varepsilon$.
The aim is to understand the event whereby a Brownian particle evolving within the domain~$\Omega_{\varepsilon}$ leaves it through one of the absorbing holes,
located at the intersection between~$\Omega$ and~$B\bigl(\xn{k}, \re{(k)}\bigr)$ for $k \in \{1, \ldots, N\}$.
The rest of the boundary of~$\Omega_{\varepsilon}$ is assumed to be reflecting.
For convenience, we 
introduce the notation
\[
    \Gamma_\mathcal{N}^\varepsilon = \partial \Omega_\varepsilon \setminus \bigcup_{k=1}^N \overline{B\left(x^{(k)}, r_\varepsilon^{(k)}\right)}, \qquad
    \Gamma_\mathcal{D}^\varepsilon = \bigcup_{k=1}^N \Gamma_k^\varepsilon \qquad \text{ where } \qquad \Gamma_k^\varepsilon = \partial B {\left(x^{(k)}, \re{(k)}\right)} \cap \Omega
\]
to refer to the reflecting Neumann region and the absorbing Dirichlet region of the boundary $\partial \Omega_\varepsilon$, respectively.
Note that $\partial \Omega_\varepsilon = \overline{\Gamma_\mathcal{N}^\varepsilon} \cup \overline{\Gamma_\mathcal{D}^\varepsilon}$.

\begin{figure}
  \centering
  \scalebox{0.7}{\begin{tikzpicture}

  \def\DtwoX{0.8}   \def\DtwoY{7.0}   \def\DtwoR{0.55}   
\def\DthreeX{1.4}  \def\DthreeY{3.0}  \def\DthreeR{1.1}  
\def\DoneX{5.8}   \def\DoneY{0.8}   \def\DoneR{0.45}    

% === Blue boundary (Neumann, reflecting) ===

\draw[blue, line width=1.8pt]
  ({\DtwoX+\DtwoR*cos(90)}, {\DtwoY+\DtwoR*sin(90)})
  .. controls (0.8, 8.5) and (2.0, 9.2) ..
  (3.5, 9.2)
  .. controls (5.5, 9.3) and (7.5, 8.5) ..
  (8.8, 7.2)
  .. controls (10.0, 5.5) and (10.2, 3.2) ..
  (9.0, 1.8)
  .. controls (8.0, 0.8) and (7.0, 0.6) ..
  ({\DoneX+\DoneR*cos(10)}, {\DoneY+\DoneR*sin(10)});

\draw[blue, line width=1.8pt]
  ({\DoneX+\DoneR*cos(170)}, {\DoneY+\DoneR*sin(170)})
  .. controls (3.8, 0.4) and (2.8, 0.3) ..
  ({\DthreeX+\DthreeR*cos(-60)}, {\DthreeY+\DthreeR*sin(-60)});

\draw[blue, line width=1.8pt]
  ({\DthreeX+\DthreeR*cos(110)}, {\DthreeY+\DthreeR*sin(110)})
  .. controls (0.4, 5.0) and (0.2, 6.0) ..
  ({\DtwoX+\DtwoR*cos(-90)}, {\DtwoY+\DtwoR*sin(-90)});

% === Red arcs (Dirichlet, absorbing) — exact circular arcs ===

\draw[red, line width=1.8pt]
  ({\DoneX+\DoneR*cos(10)}, {\DoneY+\DoneR*sin(10)})
  arc (10:170:\DoneR);

]\draw[red, line width=1.8pt]
  ({\DthreeX+\DthreeR*cos(-60)}, {\DthreeY+\DthreeR*sin(-60)})
  arc (-60:110:\DthreeR);

\draw[red, line width=1.8pt]
  ({\DtwoX+\DtwoR*cos(-90)}, {\DtwoY+\DtwoR*sin(-90)})
  arc (-90:90:\DtwoR);

% === Hole centers x^{(k)} on \partial\Omega ===
\fill (\DoneX,\DoneY) circle (1.6pt);
\node[below]      at (\DoneX,\DoneY)   {\Large $x^{(1)}$};
\fill (\DtwoX,\DtwoY) circle (1.6pt);
\node[left]       at (\DtwoX,\DtwoY)   {\Large $x^{(2)}$};
\fill (\DthreeX,\DthreeY) circle (1.6pt);
\node[above right] at (\DthreeX,\DthreeY) {\Large $x^{(3)}$};

% === Labels ===
\node at (5.0, 5.0) {\Large $\Omega_\varepsilon$};
\node at (7.8, 7.6) {\Large $\Gamma_\mathcal{N}^\varepsilon$};
\node[left] at (0.0, 7.0) {\Large $\Gamma_{2}^\varepsilon$};
\node[left] at (1.0, 2.6) {\Large $\Gamma_{3}^\varepsilon$};
\node[below] at (5.8, 0.0) {\Large $\Gamma_{1}^\varepsilon$};

\end{tikzpicture}}
  \caption{Example of a domain $\Omega_\varepsilon$ with $N=3$ holes.
    The blue parts of the boundary are reflecting,
    while the red parts are absorbing.}
  \label{fig:domain}
\end{figure}
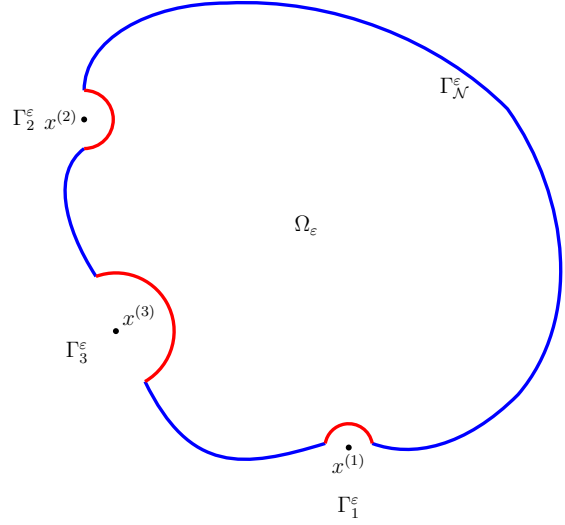

Let $(X_t)_{t \geq 0}$ be the reflected Brownian motion in $\Omega_\varepsilon$ initialized according to a probability distribution~$\mu_0$ \cite{Tanaka, RelefctedBrownian}:
\[
        \d X_t = \sqrt{2} \, \d W_t - \mathbbm{1}_ {\partial \Omega_\varepsilon} (X_t) n(X_t) \, \d L_t,
        \quad \text{with} \quad X_0 \sim \mu_0 \, ,
\]
where $(W_t)_{t \geq 0}$ is a standard Brownian motion in $\real^d$,
for any $x \in \partial \Omega_\varepsilon$
the notation~$n(x)$ refers to the outward unit normal vector to $\Omega_\varepsilon$,
and $(L_t)_{t \geq 0}$ is the local time of the process on the boundary $\partial \Omega_\varepsilon$.
Let~$\tau_\varepsilon$ denote the exit time of the process from the domain:
\begin{equation}
    \label{eq:exit_time}
    \tau_\varepsilon = \inf {\left\{ t \geq 0 \, |\, X_t \in \overline{\Gamma_\mathcal{D}^\varepsilon} \right\}} \, .
\end{equation}
Having introduced this notation,
we are now ready to state the central question of this work more precisely.
Our aim is to understand the statistics of the pair of random variables~$(\tau_\varepsilon, \kappa_\varepsilon)$
in the limit~$\varepsilon \to 0$,
where~$\tau_\varepsilon \geq 0$ is the exit time of the Brownian particle defined in~\eqref{eq:exit_time} 
and~$\kappa_\varepsilon$ is the index of the exit hole, defined as
\begin{equation}
    \label{eq:exit_hole}
    \forall k \in \{1, \ldots, N\}, \quad \kappa_\varepsilon = k \quad \Leftrightarrow \quad X_{\tau_\varepsilon} \in \overline{\Gamma_k^\varepsilon}.
\end{equation}

\subsection{Related works}
\label{sec:related_works}
We now give an overview of the literature on the narrow escape problem.
This problem was first introduced to describe biological phenomena in dimension~$3$, 
such as the escape of an ion from a cell, or the binding of a ligand to a receptor \cite{HolcmanSchuss04}.
The authors of the latter work already established the connection between the narrow escape problem and
the problem of finding $f_k$ satisfying $-\Delta f_k = 1$ in $\Omega$ with $\partial_n f_k = 0$ on $\partial\Omega \setminus \{x^{(k)}\}$, 
which is also central in the approach we develop here (see \cref{sec:introquasi-mode}).
This Poisson problem is related to electrostatics 
and potential theory problems 
(see~\cite{Pop90, Silbergleit03} for an analysis in dimension $2$ and $3$;
in higher dimensions, 
expansions around the singularity of $f_k$ have been studied only in particular cases, 
such as hyper-spheres~\cite{Sadybekov16}).
The narrow escape problem in dimension greater than $3$ and for general domains has received less attention.
Yet, such settings arise naturally in the study of phase transitions and metastable systems,                                                                                                                                                                                 
where the free energy landscape takes the form of a broad confining potential
with steep walls and narrow saddle-point passages, 
which can be modeled by the narrow escape problem
(see for instance \cite{BicoutSzabo2000, FRENKEL199926}).
In such a high-dimensional context, direct simulations are prohibitively expensive. 
Thus, 
toy models like the one described in \cref{fig:domain} can be used as a surrogate to understand these numerical phenomena. 
More broadly, the narrow escape problem belongs to a family of questions                                                                                                                                                                                       
concerning diffusion in geometrically constrained domains.
Related works have for example addressed the convergence of reflected Brownian motion in a thin neighborhood 
of a graph as the neighborhood shrinks to the graph itself \cite{Wentzell, Hsu2025a, Hsu2025b}.

We next give an overview of the  techniques that have been used to study the narrow escape problem.
In dimensions $2$ and $3$, the first methods developed were based on
the expansion of the Green's functions~\cite{Singer2006, Benichou2008, ChenFriedman2011}, 
matched asymptotics \cite{Xiaofei2014, GomezCheviakov2015}
or specific geometrical properties of the domain \cite{Grebenkov2024}. 
Another approach relies on the reformulation of the problem using capacities \cite{Ward, Felli2021}, 
which is more general but yields less explicit asymptotic formulas.   
Besides, 
thanks to layer potential techniques, 
precise asymptotics of $\mathbb{E}_{x}[\tau_\varepsilon]$, 
where~$\mathbb{E}_{x}$ denotes the expectation when starting from a fixed deterministic point $x \in \Omega_\varepsilon$, 
were obtained in \cite{AMMARI201266, Tzou2021}.
Finally, \cite{Tony2024} used semi-classical methods and the quasi-stationary distribution, 
initially developed to study the metastability of stochastic dynamics of energetic origins (see for instance \cite{Di_Ges__2016}), 
to address the narrow escape problem. 
The main advantage of this approach is that it provides tools to obtain the law of the first exit point, 
which was previously not the main focus of the narrow escape literature
(see however~\cite{Chevalier201} for the expansion of the exit point distribution
for the sphere and the disk using Green functions).
Moreover, 
it provides a rigorous framework to justify the use of jump Markov models (Markov State Models or kinetic Monte Carlo \cite{LeBris2012}) to simulate
metastable dynamics over very long times. 
The aim of this work is to extend the contribution of \cite{Tony2024} to domains of arbitrary shapes in general dimension. 
In contrast to \cite{Tony2024}, 
in this work the holes are defined explicitly in \eqref{eq:omega varepsilon}                                                                                                                                                                
rather than as level sets of a function not known analytically.
This prevents us from building a quasi-mode in the operator domain, 
unlike in \cite{Tony2024} (see in particular Lemma 3.3), 
and requires different techniques to derive results on the exit time and the exit hole
using this quasi-mode.
The companion paper~\cite{narrow2d} studies the narrow escape problem in dimension two without assuming that the holes are spherical, 
as in~\eqref{eq:omega varepsilon}. 
It also relies on the quasi-stationary distribution and uses a similar quasimode-based approach to approximate it. 
However, it derives asymptotic expansions of the exit time and exit position to arbitrary order, 
whereas the present work provides only the leading-order terms, together with error bounds.

\subsection{The quasi-stationary approach to metastability}
\label{sec:quasi_stationary}
We now introduce the concept of quasi-stationary distribution, 
and establish its relation with the narrow escape problem, 
in the spirit of \cite{Tony2024}. 
The statistics of the pair of random variables~$(\tau_\varepsilon, \kappa_\varepsilon)$ defined in (\ref{eq:exit_time}) and (\ref{eq:exit_hole}) generally depend on the initial position of the particle.
If the holes are sufficiently small, 
the particle reaches a local equilibrium known as the quasi-stationary distribution in the domain before exiting~\cite{ColletQSD, DiGesu2016, Di_Ges__2016}.
Then the random variables~$\tau_\varepsilon$ and $\kappa_\varepsilon$ are independent and, furthermore,
the exit time is exponentially distributed:~${\tau_\varepsilon \sim {\rm Exp}(\lambda_{\varepsilon}^0)}$
(see \cref{cor:QSD} for a rigorous statement).
The rate parameter $\lambda_{\varepsilon}^0$ is the smallest eigenvalue of $\mathcal{L}_\varepsilon $, 
the differential operator $-\Delta$ with domain 
\begin{equation}
    \label{eq:domain L}
    \mathcal{D}(\mathcal{L}_\varepsilon) = \{ u \in H^1(\Omega_\varepsilon), \Delta u \in L^2(\Omega_\varepsilon) 
    \, |\, u = 0 \text{ on } \Gamma_{\mathcal{D}}^\varepsilon, \, \partial_n u = 0 \text{ on } \Gamma_{\mathcal{N}}^\varepsilon\}.
\end{equation}
The normal derivative of an element of $\mathrm{D}(\mathcal{L}_\varepsilon)$ 
is defined in a weak sense by duality between $H^{-\frac{1}{2}}(\partial \Omega_\varepsilon)$ 
and $H^{\frac{1}{2}}(\partial \Omega_\varepsilon)$: 
for a function $u \in H^1(\Omega_\varepsilon)$ such that $\Delta u \in L^2(\Omega_\varepsilon)$, 
\begin{equation}\label{eq:normal_derivative_weak_h12}
    \forall v \in H^\frac{1}{2}(\partial \Omega_\varepsilon), \quad \quad \langle \partial_n u, v \rangle_{H^{-\frac{1}{2}}(\partial \Omega_\varepsilon), H^{\frac{1}{2}}(\partial \Omega_\varepsilon)} 
    = \int_{\Omega_\varepsilon} \nabla u \cdot \nabla w + \int_{\Omega_\varepsilon} \left(\Delta u \right) w,
\end{equation}
where $w$ is any extension of $v$ to $H^1(\Omega_\varepsilon)$.
Then $\partial_n u = 0 \text{ on } \Gamma_{\mathcal{N}}^\varepsilon$ must be understood as 
\begin{equation}\label{eq:meaning dn nu zero}
    \forall v \in H^{\frac{1}{2}}_{00}(\Gamma_\mathcal{N}^\varepsilon), \quad \quad \langle \partial_n u, \tilde{v} \rangle_{H^{-\frac{1}{2}}(\partial \Omega_\varepsilon), H^{\frac{1}{2}}(\partial \Omega_\varepsilon)} = 0,
\end{equation}
where $\tilde{v}$ is the extension of $v$ by $0$ on $\partial \Omega_\varepsilon \setminus \Gamma_\mathcal{N}^\varepsilon$ 
and $H^{\frac{1}{2}}_{00}(\Gamma_\mathcal{N}^\varepsilon)$
is by definition the set of functions that belong to $H^{\frac{1}{2}}(\partial \Omega_\varepsilon)$ when extended by $0$ 
on the complement of $\Gamma_\mathcal{N}^\varepsilon$.
Let $u_\varepsilon^0$ denote an eigenvector associated to $\lambda_\varepsilon^0$: 
\begin{equation}
    \label{eq:eigen}
    \left \{
        \begin{aligned}
            -\Delta u_\varepsilon^0 &= \lambda_{\varepsilon}^0 u_\varepsilon^0
                                    && \text{ in }  \Omega_{\varepsilon} \,,\\
            \partial_n u_\varepsilon^0 & = 0
                         && \text{ on } \Gamma_{\mathcal{N}}^\varepsilon \,, \\
            u_\varepsilon^0 & = 0
                         && \text{ on } \Gamma_{\mathcal{D}}^\varepsilon \,.
        \end{aligned}
    \right .
\end{equation}
The rigorous definition of the quasi-stationary distribution is as follows: if the initial position of the particle is distributed according to $\nu_\varepsilon$,
then the distribution of the particle at any time $t \geq 0$ conditioned on the event 
$\{\tau_\varepsilon > t\}$ is still $\nu_\varepsilon$.
It can be shown that the quasi-stationary distribution $\nu_\varepsilon$ of the reflected Brownian particle
has a density with respect to the Lebesgue measure, given by $u_\varepsilon^0 \, / \int_{\Omega_\varepsilon} u_\varepsilon^0$.
With a slight abuse of notation, this density will also be denoted by $\nu_\varepsilon$.
This is our first result. 
\begin{thm}\label{thm:QSD}
    Consider the reflected Brownian motion on $\Omega_\varepsilon$ 
    that is killed when reaching $\Gamma_{\mathcal{D}}^\varepsilon$. 
    There exists a unique quasi-stationary distribution $\nu_\varepsilon$ on $\Omega_\varepsilon$,
    whose density with respect to the Lebesgue measure (also denoted by $\nu_\varepsilon$)
    is the non-negative $L^1$-normalized solution of the eigenvalue problem~\eqref{eq:eigen} associated 
    with the smallest eigenvalue $\lambda_{\varepsilon}^0$.
    Furthermore, $\nu_\varepsilon \in \mathrm{D}(\mathcal{L}_\varepsilon) \cap C^0(\overline{\Omega_\varepsilon})$.
\end{thm}

\begin{rema}\label{rema:Yaglom}
    To prove that $\nu_\varepsilon$ is unique, 
    we show that $\nu_\varepsilon$ is the Yaglom limit
    of the process: for any initial distribution~$\mu_0$ on~$\Omega_\varepsilon$,
    \[
        \nu_\varepsilon = \lim_{t \to +\infty} \mathbb{P}_{\mu_0} \left( X_t \in \, \cdot \, \middle|\, t < \tau_\varepsilon \right).
    \]
    This motivates the choice of $\nu_\varepsilon$ as the initial distribution in the results stated in this paper.
    Indeed, in the small hole limit $\varepsilon \to 0$, starting under $\nu_\varepsilon$ is a natural assumption:
    as the hole sizes shrink to $0$, the particle spends a long time in $\Omega_\varepsilon$ before exiting,
    allowing it to reach local equilibrium and forget its initial distribution~$\mu_0$.
\end{rema}

\begin{prop}\label{cor:QSD}
    Assume that $X_0 \sim \nu_\varepsilon$, 
    i.e. the initial position of the Brownian particle is distributed according to the quasi-stationary distribution $\nu_\varepsilon$.
    Then the following holds:
    \begin{itemize}
        \item The exit time $\tau_\varepsilon$ is exponentially distributed with rate~$\lambda_{\varepsilon}^0$, the smallest eigenvalue of $\mathcal{L}_\varepsilon$.
        In particular, the mean exit time is given by $\mathbb{E}_{\nu_\varepsilon}[\tau_\varepsilon] = 1/\lambda_{\varepsilon}^0$.
        \item The exit point $X_{\tau_\varepsilon}$ is independent of $\tau_\varepsilon$.
        \item The normal derivative $\partial_n \nu_\varepsilon$ 
        defined in \eqref{eq:normal_derivative_weak_h12} can be identified 
        with a Radon measure supported on~$\overline{\Gamma_\mathcal{D}^\varepsilon}$ 
        with $\partial_n \nu_\varepsilon(\Gamma_{\mathcal{D}}^\varepsilon) = - \lambda_\varepsilon^0$.
        Furthermore, 
        the law of $X_{\tau_\varepsilon}$ is given by $-\frac{1}{\lambda_\varepsilon^0} \partial_n \nu_\varepsilon$:
        \[
            \forall v \in C^0(\overline{\Gamma_\mathcal{D}^\varepsilon}), \qquad 
            \mathbb{E}\bigl[ v(X_{\tau_\varepsilon}) \bigr] 
            = - \frac{1}{\lambda_\varepsilon^0} \int_{\Gamma_\mathcal{D}^\varepsilon} v \, \mathrm{d} \left(\partial_n \nu_\varepsilon\right).
        \]
        In particular,
        the law of the exit hole $\kappa_\varepsilon$ is given by:
        \[
            \forall k \in \{1, \ldots, N\}, \qquad
            \mathbb{P}_{\nu_\varepsilon} \Big( \kappa_\varepsilon=k \Big) 
            = \mathbb{P}_{\nu_\varepsilon} \Big( X_{\tau_\varepsilon} \in \overline{\Gamma_k^\varepsilon} \Big) 
            = - \frac{1}{\lambda_\varepsilon^0} \partial_n \nu_\varepsilon \left( \Gamma_k^\varepsilon\right).
        \]
    \end{itemize}
\end{prop}

The proofs of \cref{thm:QSD} and of \cref{cor:QSD} are provided in \cref{sec:app QSD}.

\Cref{cor:QSD} makes explicit the relationship between the eigenvalue problem~\eqref{eq:eigen},
and the statistics of the exit time and exit hole.
The goal of this work is to compute approximations of $\lambda_\varepsilon^0$ and $\partial_n \nu_\varepsilon$ 
in the small hole limit $\varepsilon \to 0$,
to characterize the laws of $\tau_\varepsilon$ and $\kappa_\varepsilon$ with the help of \cref{cor:QSD}.

\subsection{A quasi-mode to approximate the quasi-stationary distribution}
\label{sec:introquasi-mode}
We now formally introduce the quasi-mode $\varphi_\varepsilon$ that will be used to obtain results on the quasi-stationary distribution $\nu_\varepsilon$.
For a general domain, there is no explicit expression for~$\nu_\varepsilon$ and its normal derivative 
for positive~$\varepsilon$.
Nevertheless, in the absence of holes, the eigenvalue problem~\eqref{eq:eigen} reduces to the Neumann problem in $\Omega$,
in which case the smallest eigenvalue is $0$ and associated eigenfunctions are constant functions.
Thus, it is natural to expect that $\lambda_{\varepsilon}^0 \to 0$ as $\varepsilon \to 0$
and that~$\nu_\varepsilon$ tends to the constant $1/|\Omega|$ in the same limit,
at least far away from the exit holes.
Therefore, to understand the behavior of the eigenvalue~$\lambda_{\varepsilon}^0$ in the limit~$\varepsilon \to 0$,
we consider the following ansatz $\varphi_\varepsilon$ for $\nu_\varepsilon$ up to a multiplicative constant:
\begin{equation}
    \label{eq:def phi}
    \varphi_\varepsilon(x) = 
    1 + \sum_{k=1}^{N} \Kek f_k(x) \, ,
\end{equation}
where, for any $k \in \{1, \ldots, N\}$, the real number $\Kek$ is a small parameter that will be determined later on as a function of $r_\varepsilon^{(k)}$,
and $f_k$ are functions having singularities at the corresponding holes.
The normalization of $\varphi_\varepsilon$ has been chosen so that $\int_{\Omega_\varepsilon}\nu_\varepsilon \varphi_\varepsilon \simeq 1$.
Substituting this ansatz into \eqref{eq:eigen}, we are naturally led to considering the following equation for $f_k$:
\begin{equation}
    \label{eq:fkepd}
    \left \{
    \begin{aligned}
        -\Delta f_k & = C_d
        &&\text{ in }  \Omega,\\
        \partial_n f_k & = 0 &&\text{ on } \partial \Omega \setminus \{x^{(k)}\},
    \end{aligned}
    \right .
\end{equation}
where the positive constant $C_d$ is given an explicit expression in \eqref{eq:def_Cd}.
This choice of $C_d$ is a consequence of the decomposition of $f_k$ performed in \cref{sec:construction fk}
and in particular the compatibility condition of the Neumann problem \eqref{eq:subsingular_terms}.
The Neumann problem \eqref{eq:fkepd} is not standard as one point is missing from the boundary condition.
We give this equation a rigorous sense in \cref{sec:construction fk}.

Formally, assuming enough regularity on $f_k$, we have
\begin{equation}
    \label{eq:quasi-mode edp}
    \left \{
        \begin{aligned}
            -\Delta \varphi_\varepsilon &= - \sum_{k=1}^{\ndoors} \Kek \Delta f_k
            = \left(C_d \sum_{k=1}^{\ndoors}  \Kek \right)\varphi_\varepsilon + R_1\!\left((\Kek, f_k)_{1 \le k \le N}\right)
                                    && \text{ in }  \Omega_{\varepsilon} \,,\\
            \partial_n \varphi_{\varepsilon} & = 0
                         && \text{ on } \Gamma_{\mathcal{N}}^\varepsilon \,, \\
            \varphi_{\varepsilon} & = R_2\!\left((\Kek, f_k)_{1 \le k \le N}\right)
                         && \text{ on } \Gamma_{\mathcal{D}}^\varepsilon \, ,
        \end{aligned}
    \right .
\end{equation}
with the remainders $R_1$ and $R_2$ defined as
\[
    R_1\!\left((\Kek, f_k)_{1 \le k \le N}\right) = - \left( C_d \sum_{\ell=1}^{\ndoors} K_\varepsilon^{(\ell)} \right) \sum_{k=1}^{\ndoors} \Kek  f_k \,\,
    \qquad \,\,R_2\!\left((\Kek, f_k)_{1 \le k \le N}\right) = 1 + \sum_{k=1}^{N} \Kek f_k.
\]
The system \eqref{eq:quasi-mode edp} is very similar to \eqref{eq:eigen}, 
with however the two remainders~$R_1$ and~$R_2$. 
It is in this sense that $\varphi_\varepsilon$ is a quasi-mode for the eigenvalue problem \eqref{eq:eigen} 
(see \cref{prop:properties quasi-mode} for a precise statement of the properties of $\varphi_\varepsilon$).
We emphasize that, 
because of the non-homogeneous Dirichlet condition in \eqref{eq:quasi-mode edp}, the quasi-mode $\varphi_\varepsilon$ does not belong to the domain of the operator $\mathcal{L}_\varepsilon$, 
and it is not even in the form domain of the operator $\mathcal{L}_\varepsilon$ (see \eqref{eq:form domain} for its definition).  
Before using the quasi-mode $\varphi_\varepsilon$ to determine the distribution of the two random numbers $\tau_\varepsilon$ and $\kappa_\varepsilon$, 
we need to construct and control the functions~$f_k$ to ensure that the remainders $R_1$ and $R_2$ are small.  
The parameters $(\Kek)_{k \in \{1, \ldots, N\}}$ are 
chosen to ensure that~$\Kek f_k \sim -1$ on the boundary of the hole $\Gamma_k^\varepsilon$ in 
the small $\varepsilon$ limit.
Explicitly, 
they are given by, for all~{${k \in \{1, \ldots, N\}}$},
\begin{equation}
        \label{eq:def K}
        \Kek  = \left \{ \begin{aligned}
            & -\left( \log \re{(k)} \right)^{-1}\, && \text{ for } d = 2,\\
            & \left( \re{(k)}\right)^{d-2}\, && \text{ for } d \geq 3.
        \end{aligned} \right.
    \end{equation}
The benefit of this particular choice becomes apparent in the proof of \cref{lem:quasi-mode}.

\subsection{Our contribution and outline of the paper}
\label{sec:contribution}
We now describe the main contributions of this work.
The following parameter governs the asymptotic behavior of the quantities of interest: 
\begin{equation}\label{eq:def Keo}
    \Keo = \sum_{k=1}^N \Kek.
\end{equation}
Our first result concerns the mean exit time. 
\begin{thm}
    \label{thm:lam0}

    Assume that~$X_0 \sim \nu_\varepsilon$.
    Then, as~$\varepsilon \to 0$, the first eigenvalue 
    $\lambda_{\varepsilon}^0$, which is the reciprocal of the mean exit time, satisfies
    \[
        \lambda_\varepsilon^0 = C_d \Keo \Bigl( 1 + \mathrm{O}{\left(\mathcal{E}_d{\left(\Keo\right)}\right)}\Bigr),
    \]
    with the error term $\mathcal{E}_d\bigl(\Keo\bigr)$ defined as
    \begin{equation}\label{eq:def_Ed}
        \mathcal{E}_d(\Keo) = \left\{ \begin{aligned}
            & \, \Keo\, && \text{ for } d = 2 \, ,\\ 
            & \, \Keo \left| \log(\Keo) \right|\, && \text{ for } d = 3 \, ,\\
            & \, \Keo^{\frac{1}{d-2}}\, && \text{ for } d \geq 4 \, ,
            \end{aligned}
        \right.
    \end{equation}
    and the constant
    \begin{equation}
        \label{eq:def_Cd}
        C_d = \frac{\max \{d-2, 1\}}{2|\Omega|} \omega_d>0,
    \end{equation}
    where the constant $\omega_d = 2\pi^{\frac{d}{2}} / \Gamma\left(\frac{d}{2}\right)$
     is the surface area of the unit sphere in $\real^d$.
\end{thm}

\begin{rema}\label{rema:mean exit time}
    The scaling of the eigenvalue~$\lambda_\varepsilon^0$ with respect to~$(r_\varepsilon^{(k)})_{k \in \{1, \ldots, N\}}$ provided by~\cref{thm:lam0} is consistent with the results in the literature on the narrow escape problem.
    For instance in~\cite[Theorem 4.1]{ChenFriedman2011}, 
    for a single hole of radius $r_\varepsilon$ in dimension $d = 3$, 
    the authors show that 
    \[
        \mathbb{E}_\mu[\tau_\varepsilon] = \frac{|\Omega|}{4} r_\varepsilon^{-1} + \mathrm{O}{\left(\left|\log\left(r_\varepsilon\right)\right|\right)},
    \]
    where $\mu$ is the uniform distribution on $\Omega$.
    From \cref{thm:lam0}, we obtain
    \[
        \mathbb{E}_{\nu_\varepsilon}[\tau_\varepsilon] = \frac{1}{\lambda_\varepsilon^0} = \frac{|\Omega|}{2 \pi  r_\varepsilon} \Bigl( 1 + \mathrm{O}{\left(r_\varepsilon |\log(r_\varepsilon)|\right)}\Bigr).
    \]
    Both quantities scale like $C/r_\varepsilon$ for some constant $C > 0$.
    The multiplicative constant differs since the authors considered slightly different hole shapes
    than considered here.

    Let us add a comment on the error $\mathcal{E}_d$ in dimension greater than or equal to~$4$. 
    For a single hole of radius $r_\varepsilon$ in dimension $d \geq 4$,
    the eigenvalue $\lambda_\varepsilon^0$ behaves as 
    \[
        \lambda_\varepsilon^0 = C_d r_\varepsilon^{d-2}\bigl(1 + \mathrm{O}(r_\varepsilon) \bigr),
    \]
    so the relative error on the mean exit time is actually of order $r_\varepsilon$.
\end{rema}

Our second contribution concerns the distribution of the exit hole $\kappa_\varepsilon$. 

\begin{thm}\label{thm:exit hole distribution}
    Assume that~$X_0 \sim \nu_\varepsilon$. 
    Then,
    the distribution of the exit hole is the following:
    for any~{${k \in \{1, \ldots, N\}}$},
    \[
    \mathbb{P}_{\nu_\varepsilon} \Big( \kappa_\varepsilon=k \Big) 
    = \mathbb{P}_{\nu_\varepsilon} \Big( X_{\tau_\varepsilon} \in \overline{\Gamma_{k}^\varepsilon} \Big)
    =\frac{\Kek}{\Keo}
    + \mathrm{O}\left(\mathcal{E}_d(\Keo)\right),
    \]
    with $\mathcal{E}_d(\Keo)$ defined in \eqref{eq:def_Ed}.
\end{thm}
\begin{rema}\label{rq:exit hole distribution}
    As observed in \cite{Tony2024}, 
    this result has surprising implications. 
    Consider a domain with two holes, 
    one of radius $r_\varepsilon^{(1)} = 2 \varepsilon$ and the other of radius $r_\varepsilon^{(2)} = \varepsilon$. 
    In dimension $2$,
    the exit probabilities satisfy~$\mathbb{P}_{\nu_\varepsilon} ( \kappa_\varepsilon=i ) \to 1/2$ when 
    $\varepsilon \to 0$, 
    both for~$i = 1$ and $i = 2$.
    This property is only valid in dimension~$2$; 
    for larger dimension,
    it holds that $\lim_{\varepsilon \to 0} \mathbb{P}_{\nu_\varepsilon} \left( \kappa_\varepsilon=1 \right) > \lim_{\varepsilon \to 0} \mathbb{P}_{\nu_\varepsilon} \left( \kappa_\varepsilon=2 \right)$,
    as might intuitively be expected. 
    We illustrate this phenomenon in \cref{fig:small hole} below.

    We also note that, as pointed out in \cite{Tony2024},
    the error term in \cref{thm:exit hole distribution} can be larger than the leading order term in some regimes of the parameters.
    For instance for a domain in dimension $3$ with two holes,
    one with radius~$r_\varepsilon^{(1)} = \varepsilon$ and the other with radius $r_\varepsilon^{(2)} = \varepsilon^2$,
    \cref{thm:exit hole distribution} states that
    \[
        \mathbb{P}_{\nu_\varepsilon} \Big( \kappa_\varepsilon=2 \Big) 
        = \frac{\varepsilon^2}{\varepsilon + \varepsilon^2} + \mathrm{O}\bigl(\varepsilon |\log(\varepsilon)|\bigr)
        = \varepsilon \bigl( 1 + \mathrm{O}\bigl(|\log(\varepsilon)|\bigr) \bigr).
    \]
    In this case the error term is larger than the leading order term
    and \cref{thm:exit hole distribution} only states that~${\mathbb{P}_{\nu_\varepsilon} ( \kappa_\varepsilon=2 ) \to 0}$ when $\varepsilon \to 0$, 
    at least as fast as~$\varepsilon |\log(\varepsilon)|$,
    and thus $\mathbb{P}_{\nu_\varepsilon} ( \kappa_\varepsilon=1 )$ tends to $1$ in this limit.
\end{rema}

This work is organized as follows. 
The proofs of \cref{thm:lam0} and \cref{thm:exit hole distribution} are presented in 
\cref{sec:quasi-mode} after having constructed the functions~$f_k$, solutions to \eqref{eq:fkepd}, in \cref{sec:construction fk}.
\Cref{sec:numerics} is dedicated to numerical illustrations of our results. 
Finally, 
\cref{sec:app eigenvalues} contains results on the spectrum of operators with mixed boundary conditions,
and \cref{sec:app QSD} is devoted to the proofs of \cref{thm:QSD} and~\cref{cor:QSD}.

\section{Construction and local expansion of the quasi-mode}
\label{sec:construction fk}
The goal of this section is to construct the solution of~\eqref{eq:fkepd}.
Fix $k \in \{ 1, \ldots, N \}$. 
We show that $f_k$ can be decomposed close to the hole centered around $\xk$ as
\[
    f_k = \widetilde{\Lambda}_k + S_k - \frac{1}{|\Omega|} \int_{\Omega} \left( \widetilde{\Lambda}_k + S_k \right),
\]
where $\widetilde{\Lambda}_k$ is the composition of the fundamental solution of the Laplacian on $\real^d$ (see \eqref{eq:fundamental solution everywhere})
with an appropriate smooth diffeomorphism to account for the geometry of the domain,
and~$S_k = \mathrm{o}(\widetilde{\Lambda}_k)$ is a term subsingular in the limit $x \to x^{(k)}$.

Before providing the details of the construction,
we give a rigorous sense to the Neumann problem~\eqref{eq:fkepd} in the next subsection.
Then, we compute an explicit expression for $f_k$ itself in three steps: 
we first establish the relevant geometric properties of $\Omega$
in~\cref{sec:boundary_flattening},                                                                                                                                                            
then construct a solution to~\eqref{eq:fkepd} term by term, handling first the most singular terms $\widetilde{\Lambda}_k$ in~\cref{sec:fundamental sol} 
before turning to the subsingular ones $S_k$ in~\cref{sec:subsingular_terms}.
Finally, we conclude by providing local asymptotics and other estimates on $f_k$ in~\cref{sec:actual construction fk}.

\subsection{\texorpdfstring{A rigorous sense to the Neumann problem \eqref{eq:fkepd}}{A rigorous sense to the Neumann problem}}
Note that there cannot exist a classical solution to~\eqref{eq:fkepd}.
Indeed, 
if there were,
an application of the divergence theorem would yield the following contradiction:
\begin{equation}\label{eq:compat}
    - |\Omega| C_d = \int_\Omega \Delta f_k = \int_{\partial \Omega} \partial_n f_k = 0.
\end{equation}
Hence, 
we need a distributional interpretation of \eqref{eq:fkepd}. 
For this, 
we use the~$W^{1, p}$ framework, 
considered in particular in \cite{aramaki2018existence}.
For $p>1$, 
given two distributions $g \in L^p(\Omega)$ and~$h \in W^{-\frac{1}{p}, p}(\partial \Omega)$,
we define a weak solution to 
\begin{equation}
    \label{eq:generic_neumann}
    \left \{
    \begin{aligned}
        -\Delta f & = g
        &&\text{ in }  \Omega,\\
        \partial_n f & = h &&\text{ on } \partial \Omega,
    \end{aligned}
    \right .
\end{equation}
as a function $f \in W^{1, p}(\Omega)$ such that, for any $v \in C^1(\overline{\Omega})$,
\begin{equation}\label{eq:weak formulation}
    \int_\Omega \nabla f \cdot \nabla v = \int_\Omega g v + \langle h, v|_{\partial \Omega} \rangle_{W^{-\frac{1}{p}, p}(\partial \Omega), W^{\frac{1}{p}, p'}(\partial \Omega)}
\end{equation}
where $p'$ is defined by $1/p + 1/p' = 1$.
The bracket $\langle \cdot, \cdot \rangle_{W^{-\frac{1}{p}, p}(\partial \Omega), W^{\frac{1}{p}, p'}(\partial \Omega)}$
denotes the duality pairing between $W^{-\frac{1}{p}, p}(\partial \Omega)$ and $W^{\frac{1}{p}, p'}(\partial \Omega)$.
In a similar fashion to \eqref{eq:normal_derivative_weak_h12}, 
the meaning of the normal derivative we will use in the rest of the paper is
the following: for any~${u \in W^{\frac{1}{p}, p'}(\partial \Omega)}$ and~$f \in W^{1, p}(\Omega)$ with~$\Delta f \in L^p(\Omega)$,
\begin{equation}\label{eq:normal_}
    \langle \partial_n f, u \rangle_{W^{-\frac{1}{p}, p}(\partial \Omega), W^{\frac{1}{p}, p'}(\partial \Omega)} 
    = \int_\Omega \nabla f \cdot \nabla v + \int_\Omega (\Delta f) v
\end{equation}
where $v \in W^{1, p'}(\Omega)$ denotes any extension of $u$ to $\Omega$,
the existence of which is granted by the surjectivity of the trace operator in these spaces (see \cite[Theorem 1.5.1.3]{Grisvard}).
The duality defined in \eqref{eq:normal_} is independent of the choice of the extension $v$.

Let us state an existence and uniqueness result from \cite[Proposition 2.3]{aramaki2018existence} that will be useful in the rest of the paper.
\begin{prop}
    \label{prop:sol_a_la_aramaki}
    There exists a unique, up to an additive constant, solution to \eqref{eq:generic_neumann} 
    in the sense of~\eqref{eq:weak formulation}
    if and only if the following compatibility condition holds:
    \begin{equation}\label{eq:compat_w1p}
        \int_\Omega g + \langle h, 1 \rangle_{W^{-\frac{1}{p}, p}(\partial \Omega), W^{\frac{1}{p}, p'}(\partial \Omega)} = 0\, .
    \end{equation}
\end{prop}

We adopt this $W^{1,p}$ functional framework throughout the rest of the paper.
With this in mind, 
we now give a rigorous sense of solution to \eqref{eq:fkepd} in the following definition.
\begin{defi}
    \label{def:fk}
    We denote by $f_k$ the zero-mean solution of \eqref{eq:generic_neumann} in the sense of \eqref{eq:weak formulation},
    with $g = C_d$ and~${h = -C_d |\Omega| \delta_{\xk}}$, 
    where $\delta_{\xk}$ is the Dirac mass at $\xk$ and $C_d$ is the positive constant defined in~\eqref{eq:def_Cd}:
    \begin{equation}\label{eq:fkepd distributional}
        \left \{
        \begin{aligned}
            -\Delta f_k & = C_d
            &&\text{ in }  \Omega,\\
            \partial_n f_k & = -C_d |\Omega| \delta_{\xk}&&\text{ on } \partial \Omega.
        \end{aligned}
        \right .
    \end{equation}
\end{defi}
Notice first that $\delta_{\xk} \in W^{-\frac{1}{p}, p}(\partial \Omega)$ for any $p \in (1, d/(d-1))$, 
as $\left( W^{-\frac{1}{p}, p}(\partial \Omega)\right)' \hookrightarrow C^0(\partial \Omega)$ for such $p$.
Secondly, since $\int_\Omega g = C_d |\Omega|$ and $\langle h, 1 \rangle_{W^{-\frac{1}{p}, p}(\partial \Omega), W^{\frac{1}{p}, p'}(\partial \Omega)}  = -C_d |\Omega|$,
the compatibility condition~\eqref{eq:compat_w1p} is satisfied
and~\cref{prop:sol_a_la_aramaki} ensures that $f_k$ is unique and well-defined in $W^{1, p}(\Omega)$ for any $p \in (1, d/(d-1))$. 

The goal of the rest of the section is to build explicitly another solution to \eqref{eq:fkepd distributional}
and use the uniqueness result from \cref{prop:sol_a_la_aramaki} to conclude that it is indeed $f_k$.

\subsection{Local flattening of the boundary}
\label{sec:boundary_flattening}
The goal of this subsection is to construct a smooth local change of variables near~$\xk$
that maps the boundary $\partial\Omega$ to a hyperplane (see \cref{fig:psi})
that is central in the following analysis.

\begin{lem}\label{lem:variable change}
    For $x\in\partial\Omega$, 
    let $n(x)$ denote the unit outward normal to $\partial\Omega$ at $x$, 
    and write $n_k := n(\xk)$.
    There exist open sets $U_k, V_k \subset \real^d$ with $\xk \in U_k$ and $0 \in V_k$,
    and a $\mathcal{C}^\infty$ diffeomorphism $\Psi_k \colon U_k \to V_k$ satisfying:
    \begin{enumerate}[label=\textbf{(\alph*)}]
        \item \label{item:variable_change_1}
            $\Psi_k\bigl(\partial\Omega \cap U_k\bigr) = n_k^\perp \cap V_k$,
        \item \label{item:variable_change_2}
            $\Psi_k\bigl(\Omega \cap U_k\bigr) = \bigl\{y \in V_k : y \cdot n_k < 0\bigr\}$,
        \item \label{item:variable_change_3}
            $\Psi_k(\xk) = 0$
            \quad and \quad
            $\nabla\Psi_k(\xk) = \I_d$,
        \item \label{item:variable_change_4}
            $\nabla\Psi_k(x)\, n(x) = n_k$
            \quad for all $x \in \partial\Omega \cap U_k$.
    \end{enumerate}
\end{lem}

\begin{figure}
    \centering
    \colorlet{axcol}{black!50}

\begin{tikzpicture}

%% ─── LEFT PLOT ───────────────────────────────────────────────
\begin{scope}[local bounding box=Lplot]

  \draw[->, axcol, semithick] (-2.8, 0) -- (3.5, 0);
  \draw[->, axcol, semithick] (0, -2.3) -- (0, 3.2);

  \draw[line width=1.8pt, color={rgb,255:red,31;green,119;blue,180}]
      (2.333, -2.255)
      .. controls (1.713, -1.363) and (1.067, -0.175) ..
      (0.471, 0.421)
      .. controls (-0.125, 1.017) and (-1.313, 1.563) ..
      (-2.385, 1.921);

  \node[black] at (-0.2, -0.6) {$\Omega$};
  \node[black] at (2.1, -1.3) {$\partial \Omega$};

  \coordinate (xk) at (0.471, 0.421);
  \fill[black] (xk) circle (2.5pt);
  \node[black, right=4pt] at (xk) {$x^{(k)}$};

    \node[black] at (0.0, 3.8) {$U_k$};

  \draw[->, black, thick] (xk) -- ++(0.707, 0.707)
      node[right=1pt] {$n_k$};

\end{scope}

%% ─── MAP ARROW Psi_k ─────────────────────────────────────────
\draw[->, black, thick]
    (3.8, 0.5) -- (5.2, 0.5) node[above, midway] {$\Psi_k$};

\draw[->, black, thick]
    (5.2, -0.5) -- (3.8, -0.5) node[above, midway] {$\Phi_k$};

%% ─── RIGHT PLOT ──────────────────────────────────────────────
\begin{scope}[shift={(8.3, 0)}]

  \draw[->, axcol, semithick] (-2.8, 0) -- (3.5, 0);
  \draw[->, axcol, semithick] (0, -2.3) -- (0, 3.2);

  \draw[line width=1.8pt, color={rgb,255:red,31;green,119;blue,180}]
      (-2.3, 2.3) -- (2.3, -2.3);
      \node[black] at (2.05, -1.35) {$n_k^\perp$};
    \node[black] at (0.0, 3.8) {$V_k$};

  \coordinate (ptR) at (0.0, 0.0);
  \draw[->, black, thick] (ptR) -- ++(0.707, 0.707)
      node[right=1pt] {$n_k$};
    \coordinate (zero) at (0.0, 0.0);
  \fill[black] (zero) circle (2.5pt);
  \node[black, right=4pt] at (zero) {$0$};
\end{scope}

\end{tikzpicture}
    \caption{Visualization of the boundary-flattening diffeomorphism $\Psi_k$ from~\cref{lem:variable change}.
    Its inverse $\Phi_k$ is given by~\eqref{eq:def Psi}.
    This figure represents the general case where $\xk$ is not at the origin and $n_k$ is not aligned with one of the coordinate axes.}
    \label{fig:psi}
\end{figure}
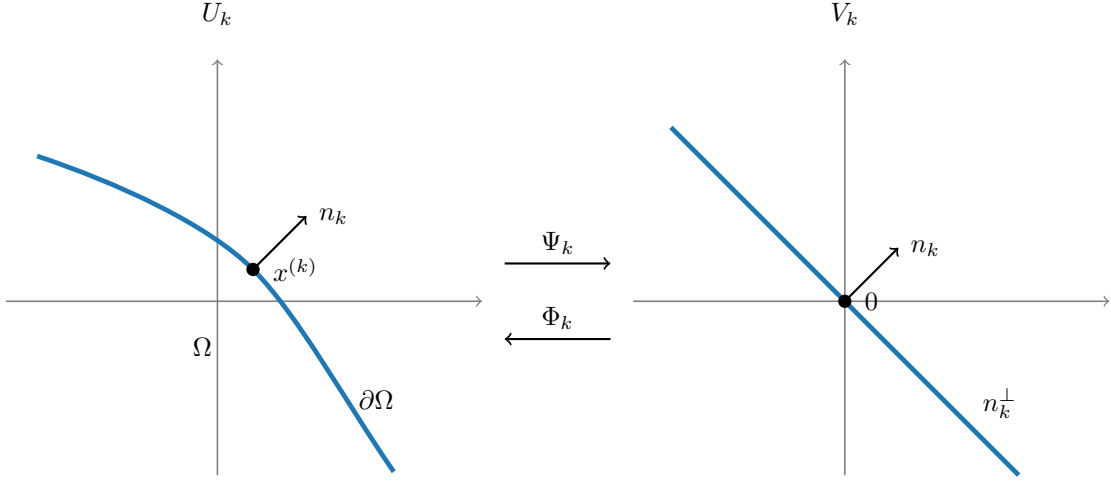

\begin{proof}
We first construct $\Psi_k$, and then prove the four properties in order.

\paragraph{Construction.}
Without loss of generality, upon translating and rotating coordinates, we assume
\[
    \xk = 0
    \qquad\text{and}\qquad
    n_k = e_d \coloneq (0, \dotsc, 0, 1)^\t.
\]
Since $\Omega$ is a $\mathcal{C}^\infty$ domain, there exist an open set $\widetilde{V}_k \subset n_k^\perp = \real^{d-1}$ containing $x^{(k)} = 0$,
a $\zeta_k > 0$
and a function $h_k \colon \widetilde{V}_k \to \real$ of class $\mathcal{C}^\infty$ such that
$\partial\Omega$ is locally the graph of $h_k$,
and $\Omega$ lies below it:
\begin{equation}\label{eq:local description omega}
    \Omega \cap \left(\widetilde{V}_k \times (-\zeta_k, \zeta_k) \right)
    = \bigl\{(\xi, x_d) \in \widetilde{V}_k \times (-\zeta_k, \zeta_k) \;\bigm|\; x_d < h_k(\xi)\bigr\}.
\end{equation}
Since $\xk = 0 \in \partial\Omega$ and $n_k = e_d$ is orthogonal to the tangent space
$T_0(\partial\Omega) = \bigl\{(\xi,\, \nabla h_k(0) \cdot \xi) : \xi \in \real^{d-1}\bigr\}$,
we have $h_k(0) = 0$ and $\nabla h_k(0) = 0$.
A Taylor expansion then gives, in the limit $\xi \to 0$,
\begin{equation}\label{eq:expansion h}
    h_k(\xi) = \displaystyle \frac{1}{2}\xi^\t \nabla^2 h_k(0)\,\xi + \mathrm{O}(|\xi|^3),
    \quad
    \nabla h_k(\xi) = \nabla^2 h_k(0)\,\xi + \mathrm{O}(|\xi|^2),
    \quad
    \nabla^2 h_k(\xi) = \nabla^2 h_k(0) + \mathrm{O}(|\xi|).
\end{equation}
We define the boundary parametrization $\psi_k \colon \widetilde{V}_k \to \partial\Omega$ by
\begin{equation}\label{eq:phi explicit}
    \psi_k(\xi) \coloneq \bigl(\xi,\, h_k(\xi)\bigr),
\end{equation}
and the map $\Phi_k \colon \widetilde{V}_k \times \real \to \real^d$ by
\begin{equation}\label{eq:def Psi}
    \Phi_k(\xi, t) \coloneq \psi_k(\xi) + t\, n\bigl(\psi_k(\xi)\bigr).
\end{equation}
See~\cref{fig:psi} for a visual representation.
Note that $\Phi_k(0, 0) = \psi_k(0) = 0 = \xk$.
By the tubular neighborhood theorem~\cite[Section~6]{IntroManifolds},
upon possibly restricting $\widetilde{V}_k$ and taking $\zeta_k > 0$ sufficiently small,
the restriction of $\Phi_k$ to $V_k \coloneq \widetilde{V}_k \times (-\zeta_k, \zeta_k)$
is a $\mathcal{C}^\infty$ diffeomorphism onto the open neighborhood $U_k \coloneq \Phi_k(V_k)$ of~$\xk$.
Let us show that 
\[
    \Psi_k = \Phi_k^{-1} \colon U_k \to V_k
\]
verifies all the required properties.
First notice that, by construction,
the function~$\Psi_k$ is a smooth diffeomorphism and $\Psi_k(\xk) = 0$.

\paragraph{Properties~\ref{item:variable_change_1} and~\ref{item:variable_change_2}.}
For any $x \in U_k$, we have $\Psi_k(x) = (\xi, t) \in V_k = \widetilde{V}_k \times (-\zeta_k, \zeta_k)$.
In particular, $\Psi_k(x) \cdot n_k = (\xi, t)^\top \cdot e_d = t$.
By~\eqref{eq:def Psi} and the local description~\eqref{eq:local description omega},
$x$ belongs to $\partial\Omega$ (respectively $\Omega$) if and only if $t = 0$ (resp.\ $t < 0$).
Hence $\Psi_k$ satisfies both properties~\ref{item:variable_change_1} and~\ref{item:variable_change_2}.

\paragraph{Jacobian of $\Phi_k$ and property~\ref{item:variable_change_3}.}
On $V_k$, $\nabla = (\partial_{\xi_1}, \dotsc, \partial_{\xi_{d-1}}, \partial_t)^\top$
and differentiating~\eqref{eq:def Psi}, we obtain
\begin{equation}\label{eq:jac Psi chain}
    \nabla\Phi_k(\xi, t)
    =
    \begin{pmatrix}
        \nabla\psi_k(\xi) & \0_{d \times 1}
    \end{pmatrix}
    +
    \begin{pmatrix}
        t\,\nabla(n \circ \psi_k)(\xi) & n\bigl(\psi_k(\xi)\bigr)
    \end{pmatrix}.
\end{equation}
From~\eqref{eq:phi explicit}, the Jacobian of $\psi_k$ is
\begin{equation}\label{eq:jac1}
    \nabla\psi_k(\xi) =
    \begin{pmatrix}
        \I_{d-1} \\
        \nabla h_k(\xi)^\t
    \end{pmatrix},
\end{equation}
and the definition of the outward unit normal yields
\begin{equation}\label{eq:jac2}
    (n \circ \psi_k)(\xi) =
    \frac{1}{\sqrt{1 + \lvert \nabla h_k(\xi) \rvert^2}}
    \begin{pmatrix}
        -\nabla h_k(\xi) \\
        1
    \end{pmatrix}.
\end{equation}
Differentiating~\eqref{eq:jac2} with respect to $\xi_j$ yields, for each $j \in \{1, \dotsc, d-1\}$,
\begin{equation}\label{eq:jac3}
    \partial_{\xi_j}(n \circ \psi_k)(\xi)
    =
    \frac{1}{\sqrt{1 + \lvert \nabla h_k(\xi) \rvert^2}}
    \begin{pmatrix}
        -\partial_{\xi_j} \nabla h_k(\xi) \\
        0
    \end{pmatrix}
    -
    \frac{\partial_{\xi_j} \nabla h_k(\xi) \cdot \nabla h_k(\xi)}{\bigl(1 + \lvert \nabla h_k(\xi) \rvert^2\bigr)^{3/2}}
    \begin{pmatrix}
        -\nabla h_k(\xi) \\
        1
    \end{pmatrix}.
\end{equation}
Substituting~\eqref{eq:jac1}--\eqref{eq:jac3} into~\eqref{eq:jac Psi chain}
and evaluating at $(\xi, t) = (0, 0)$,
using $h_k(0) = 0$ and $\nabla h_k(0) = 0$ from~\eqref{eq:expansion h},
we obtain:
\begin{equation}\label{eq:jac Psi at 0}
    \nabla\Phi_k(0, 0) = \I_d.
\end{equation}
By the inverse function theorem,
\begin{equation}\label{eq:connection_jacobian_inverse}
    \nabla\Psi_k(\xk)
    = \bigl(\nabla\Phi_k\bigl(\Psi_k(\xk)\bigr)\bigr)^{-1}
    = \bigl(\nabla\Phi_k\bigl(0\bigr)\bigr)^{-1}
    = \I_d,
\end{equation}
establishing property~\ref{item:variable_change_3}.

\paragraph{Property~\ref{item:variable_change_4}.} Let $x \in \partial\Omega \cap U_k$
and let $\xi \in \widetilde{V}_k$ be such that $\Psi_k(x) = (\xi, 0)$.
Evaluating~\eqref{eq:jac Psi chain} at $t = 0$,
the last column of $\nabla\Phi_k(\xi, 0)$ equals $n(\psi_k(\xi)) = n(x)$, i.e.,
\[
    \nabla\Phi_k(\xi, 0)\, e_d = n(x).
\]
Applying $\nabla\Psi_k(x) = \bigl(\nabla\Phi_k(\xi, 0)\bigr)^{-1}$ gives
\[
    \nabla\Psi_k(x)\, n(x) = e_d = n_k,
\]
which is property~\ref{item:variable_change_4}.
\end{proof}

The following corollary records the Taylor expansion of $\Psi_k$ near $\xk$,
which will be used throughout this section.

\begin{cor}\label{lem:inverse psi}
    As $x \to \xk$, the following estimates hold:
    \begin{align*}
        \Psi_k(x) &= x - \xk + \mathrm{O}\left(\left\lvert x - \xk \right\rvert^2\right),\\
        \nabla\Psi_k(x) &= \I_d + \mathrm{O}\left(\left\lvert x - \xk \right\rvert\right),\\
        \Delta\Psi_k(x) &= \mathrm{O}(1),
    \end{align*}
    where $\Delta$ denotes the componentwise Laplacian.
\end{cor}

\begin{proof}
    By~\cref{item:variable_change_3} of~\cref{lem:variable change},
    $\Psi_k(\xk) = 0$ and $\nabla\Psi_k(\xk) = \I_d$.
    A Taylor expansion of the smooth map $\Psi_k$ around $\xk$ gives
    \[
        \Psi_k(x)
        = \nabla\Psi_k(\xk)\,(x - \xk) + \mathrm{O}\left(\left\lvert x - \xk \right\rvert^2\right)
        = x - \xk + \mathrm{O}\left(\left\lvert x - \xk \right\rvert^2\right),
    \]
    proving the first estimate. The second point is obtained with the same arguments, 
    using the first order expansion of $\nabla\Psi_k$.
    The third point holds since $\Delta\Psi_k$ is smooth on $\overline{U_k}$, hence bounded.
\end{proof}
%%%%%%%%%%%%%%%%%%%%%%%%%%%%%%%%%%%%%%%%%%%%%%%%%%%%%%%
\subsection{\texorpdfstring{The leading term $\widetilde{\Lambda}_k$}{The leading term: Lambda tilde k}}
\label{sec:fundamental sol}
The main building block to construct an explicit expression of $f_k$, the solution of \eqref{eq:fkepd distributional}, is the fundamental solution of the Laplacian:
\begin{equation}\label{eq:fundamental_solution}
    \Lambda \colon \left|\, \, \begin{aligned}
    V_k\setminus\{0\}  & \to \real\\
    y & \mapsto \left\{ \begin{aligned}
        &\log|y| \, && \text{ for } d = 2,\\
        &-|y|^{2-d} \, && \text{ for } d \geq 3.
    \end{aligned}
    \right.
    \end{aligned}
    \right.
\end{equation}
The absence of the usual constant of normalization and the sign in this definition are intentional,
as it provides the correct multiplicative constant on $f_k$ to be a solution of \eqref{eq:fkepd distributional}
as seen later in \cref{sec:actual construction fk}.
To build~$f_k$ on~$\Omega$,
we first introduce the following intermediate function:
\begin{equation}\label{eq:fundamental solution on the tube}
    \widehat{\Lambda}_k \colon 
    \left\{ \begin{aligned} 
        U_k \cap \Omega  & \to \real\\
        x & \mapsto \Lambda \circ \Psi_k(x).
    \end{aligned}
    \right. 
\end{equation}
The change of variable $\Psi_k$ ensures that the normal derivative of $\widehat{\Lambda}_k$ is $0$ for any boundary point $x \in \partial \Omega$ around the hole,
as made precise in the following lemma.
\begin{lem}
    \label{lem:normal_derivative_fundamental}
    The normal derivative of $\widehat{\Lambda}_k$ defined in \eqref{eq:fundamental solution on the tube} is $0$
    on $\bigl( U_k \cap \partial \Omega \bigr) \setminus \{ x^{(k)}\}$.
\end{lem}
\begin{proof}
    We begin by calculating the gradient and Hessian matrix of~$\Lambda$.
    For any $y \in V_k \setminus\{0\}$, we have
    \begin{equation}
        \label{eq:grad gamma}
        \nabla \Lambda(y) = A_d |y|^{-d} y,
        \quad\quad\quad\quad \left[ \nabla^2\, \Lambda(y) \right]_{ij}=
        \delta_{ij} A_d |y|^{-d} - A_dd |y|^{-d-2} y_i y_j,
    \end{equation}
    where $A_d = \max{\{d-2, 1\}}$. It is clear that $\partial_n \widehat{\Lambda}_k$ is smooth on $U_k \cap \partial \Omega$,
    except at $\{x^{(k)}\}$ where its singularity is located.
    For  $x \in \bigl( U_k \cap \partial \Omega \bigr) \setminus \{ x^{(k)}\}$,
    let us evaluate $\partial_n \widehat{\Lambda}_k$:
    \begin{align}
        \notag
        \partial_n \widehat{\Lambda}_k(x) 
        & = n(x) \cdot \nabla \widehat{\Lambda}_k(x)\\
        \notag
        & = n(x) \cdot \Bigl( (\nabla \Psi_k)^{\t}\!(x) \nabla \Lambda\bigl(\Psi_k(x)\bigr) \Bigr)\\
        \label{eq:expresion_tilde_gamma_grad}
        & = n_k \cdot \nabla \Lambda(y),
    \end{align}
    where we used \cref{item:variable_change_4} of \cref{lem:variable change}
    and the existence of a unique $y \in V_k \cap n_k^\perp$ such that $\Psi_k(x) = y$, 
    thanks to \cref{item:variable_change_1} of \cref{lem:variable change}.
    Thus
    substituting \eqref{eq:grad gamma} into~\eqref{eq:expresion_tilde_gamma_grad},
    we obtain that
    \[
        \partial_n \widehat{\Lambda}_k(x) = n_k \cdot \nabla \Lambda(y) =  A_d |y|^{-d} n_k \cdot y = 0,
    \]
    which concludes the proof.
\end{proof}

To continue, 
we present a result concerning the scaling of~$\laplacian \widehat \Lambda_k$ around $x^{(k)}$ and its integrability. 
\begin{lem}
    \label{lem:behavior_near_0}
    There exist $R > 0$ and $C_{\Lambda} > 0$ such that
    \[
        \forall x \in \Omega \cap B(x^{(k)}, R), \qquad
        \bigl\lvert \Delta \widehat{\Lambda}_k(x) \bigr\rvert
        \leq C_{\Lambda} \left|x - x^{(k)}\right|^{1-d}.
    \]
    In particular, $\Delta \widehat{\Lambda}_k \in L^{p}(U_k \cap \Omega)$ for $1 < p < \frac{d}{d-1}$. 
\end{lem}
\begin{proof}
    Let us first compute the Laplacian of $\widehat{\Lambda}_k$ close to $x^{(k)}$ 
    using the rules of composition: 
    \begin{align*}
        \Delta \widehat{\Lambda}_k &= \nabla \cdot \left( (\nabla \Psi_k)^{\t} \nabla \Lambda \circ \Psi_k \right)\\
        &= \sum_{i, j=1}^d \partial_i^2 \left(\Psi_k\right)_j \partial_j \Lambda \circ \Psi_k
        + \sum_{i, j, \ell=1}^d \partial_i \left(\Psi_k\right)_j \partial_i \left(\Psi_k\right)_\ell \left[ \hess \Lambda \circ \Psi_k \right]_{j \ell }\\
        &= \Delta \Psi_k \cdot \nabla \Lambda \circ \Psi_k +  \left( \nabla \Psi_k \bigl(\nabla \Psi_k\bigr)^{\t}\right)  \colon \hess \Lambda \circ \Psi_k.
    \end{align*}
    The notation $\partial_i \left(\Psi_k\right)_j$ is here to indicate that 
    the $j$-th component of $\Psi_k$ is differentiated with respect to the $i$-th variable.  
    Since $\Delta \Lambda(x) = 0$ for $x \neq x^{(k)}$, we have
    \[
        \Delta \widehat{\Lambda}_k = 
        \Delta \Psi_k \cdot \nabla \Lambda \circ \Psi_k 
        +  \left( \nabla \Psi_k \bigl(\nabla \Psi_k\bigr)^{\t} - \I_d \right)  
        \colon \hess \Lambda \circ \Psi_k.
    \]
    To continue, we use the scalings of $\Psi_k$ and its derivatives around $x^{(k)}$ obtained in \cref{lem:inverse psi}.
    In particular, there exists a constant $C$ and a radius $R > 0$ such that
    \begin{alignat*}{2}
        \forall x \in B(x^{(k)}, R) \cap \Omega,
        \qquad
        \lvert \Psi_k(x) - (x - x^{(k)}) \rvert
        & \leq C \left|x - x^{(k)}\right|^2,\\
        \qquad
        \lvert \nabla \Psi_k (x) - \I_d \rvert
        & \leq C \left| x - x^{(k)} \right|,\\
        \qquad
        \lvert \laplacian  \Psi_k(x) \rvert
        & \leq C,
    \end{alignat*}
    where, for vectors and matrices,
    the notation $| \, \cdot \, |$ denotes the Euclidean norm and the associated operator norm.
    For any $x$ in $B(x^{(k)}, R) \cap \Omega$, 
    it follows from the triangle inequality that
    \begin{align*}
        \left\lvert \nabla \Psi_k (x) \, \bigl(\nabla \Psi_k\bigr)^{\t}\!(x) - \I_d  \right\rvert
        &=
        \left\lvert \Bigl( \I_d + (\nabla \Psi_k (x)\, - \I_d) \Bigr) \Bigl( \I_d + (\nabla \Psi_k (x)\, - \I_d) \Bigr)^\t - \I_d  \right\rvert
        \\
        &\leq
        2 C|x-x^{(k)}| + C^2|x-x^{(k)}|^2.
    \end{align*}
    Therefore, for $x \neq x^{(k)}$ close to $x^{(k)}$,
    it holds that
    \begin{equation}
        \label{eq:laplace tilde gamma}
        \begin{aligned}
            \bigl\lvert \Delta \widehat{\Lambda}_k(x) \bigr\rvert
            &\leq C \Bigl\lvert \nabla \Lambda\circ \Psi_k(x)  \Bigr\rvert
            + \Bigl(2C \left|x-x^{(k)}\right| + C^2 \left|x-x^{(k)}\right|^2\Bigr) \bigl\lvert \hess \Lambda \circ \Psi_k(x) \bigr\rvert.
        \end{aligned}
    \end{equation}
    Recall that
    \(
        \left\lvert \nabla \Lambda (y) \right\rvert = A_d |y|^{1 - d}
    \)
    for a constant factor $A_d$ depending on the dimension (see \eqref{eq:grad gamma}).
    Upon reducing~$R$ if necessary to ensure~$|x - x^{(k)}| - C|x - x^{(k)}|^2 \geq \frac{1}{2} |x - x^{(k)}|$ for all $x \in B(x^{(k)}, R)$,
    we deduce that
    \begin{align*}
        \lvert \nabla \Lambda\circ \Psi_k(x) \rvert
        = \frac{A_d}{\left\lvert \Psi_k(x)  \right\rvert^{d-1}}
        & \leq \frac{A_d}{\Bigl| \left\lvert x - x^{(k)} \right\rvert - \left\lvert \bigl( \Psi_k(x) - (x - x^{(k)}) \bigr)  \right\rvert \Bigr|^{d-1}}\\
        & \leq \frac{A_d}{\Bigl| \left\lvert x - x^{(k)} \right\rvert - C |x - x^{(k)}|^2 \Bigr|^{d-1}}\\
        & \leq \frac{2^{d-1} A_d}{|x - x^{(k)}|^{d-1}},
    \end{align*}
    for any $x$ in $B(x^{(k)}, R) \cap \Omega$.
    An analogous argument for the Hessian term $\hess \Lambda \circ \Psi_k$,
    shows that~$\left| \nabla^2 \Lambda \circ \Psi_k(x) \right|$ can be bounded by
    $\left| x - x^{(k)} \right|^{-d}$ up to a prefactor,
    which in view of \eqref{eq:laplace tilde gamma} concludes the proof.
\end{proof}

The function~$\widehat \Lambda_k =  \Lambda \circ \Psi_k$ is only defined on~$\Omega \cap U_k$.
From $\widehat{\Lambda}_k$, we construct a function $\widetilde{\Lambda}_k$ defined on the whole domain~$\Omega$.
To do so,
we introduce a smooth function~$\eta_L \colon \real \to \real$ satisfying
\[
    \eta_L(z) = 
    \begin{cases}
        0 & \quad \text{if $z \geq - L + 1 $}, \\
        z + L & \quad \text{if $z \leq -L - 1$}.
    \end{cases}
\]

Using this function, 
we construct the function $\widetilde\Lambda_k$ as
\begin{equation}
    \label{eq:fundamental solution everywhere}
        \widetilde{\Lambda}_k \colon 
        \left|\, \,  \begin{aligned}
        \Omega & \to \real\\
        x & \mapsto \left\{
            \begin{aligned}
                & \eta_L \circ \Lambda \circ \Psi_k(x) && \text{ if } x \in U_k \cap \Omega,\\
                & 0 && \text{ if } x \in \Omega \setminus U_{k}.
            \end{aligned}
        \right.
    \end{aligned}
        \right.
\end{equation}

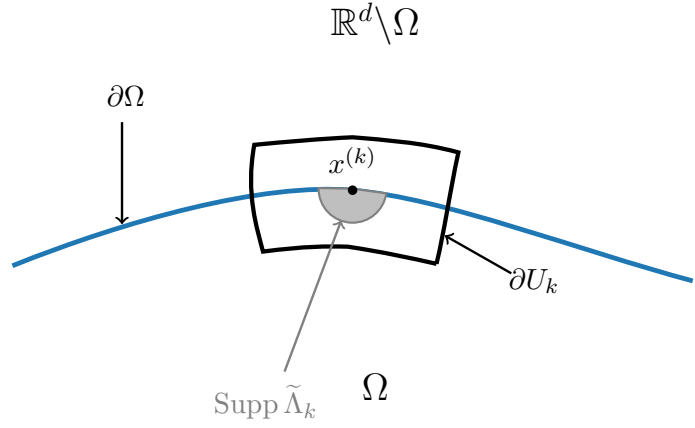
\begin{figure}
    \centering
    \begin{tikzpicture}
\begin{scope}[rotate=90]
\draw[line width=1.8pt, color={rgb,255:red,31;green,119;blue,180}]
    (-0.6, -4.5)
    .. controls (-0.2, -3.0) and (0.5, -1.2) .. (0.6, 0.0)
    .. controls (0.7, 1.2) and (0.2, 3.0) .. (-0.4, 4.5);

\draw[line width=1.6pt, black]
    (-0.3745, -1.1098)
    .. controls (-0.30, -0.78) and (-0.22, -0.42) .. (-0.1474, 0.0623)
    .. controls (-0.14, 0.35) and (-0.14, 0.62) .. (-0.2120, 1.1849)
    .. controls (0.08, 1.26) and (0.55, 1.42) .. (1.2000, 1.3000)
    .. controls (1.24, 0.88) and (1.28, 0.44) .. (1.300, 0.0000)
    .. controls (1.26, -0.44) and (1.24, -0.88) .. (1.100, -1.4000)
    .. controls (0.55, -1.28) and (0.08, -1.22) .. (-0.3745, -1.1098);

\fill[fill=black!25, draw=black!50, line width=0.8pt]
    (0.62, 0) -- ++(89:0.45) arc (89:262:0.45) -- cycle;

\fill[black] (0.6, 0) circle (1.8pt);
\end{scope}

\node[font=\large, above] at (0, 0.7) {$x^{(k)}$};

\draw[->, line width=0.9pt]
    (-3.05, 1.5) -- (-3.05, 0.15);
\node[font=\large] at (-3.0, 1.85) {$\partial\Omega$};

\draw[->, line width=0.9pt]
    (2.1, -0.5) -- (1.22, 0.0);
\node[font=\large] at (2.4, -0.6) {$\partial U_{k}$};

\draw[->, line width=0.9pt, black!50]
    (-0.9, -1.8) -- (-0.15, 0.2);
\node[font=\large, text=black!50] at (-1.15, -2.2)
    {$\mathrm{Supp}\,\widetilde{\Lambda}_k$};

\node[font=\LARGE] at (0.3, -2.0) {$\Omega$};
\node[font=\LARGE] at (0.3, 2.8) {$\mathbb{R}^d \backslash \Omega$};
\end{tikzpicture}
    \caption{This figure illustrates the domain~$U_k$ near the boundary at $x^{(k)}$, 
    introduced in \eqref{eq:fundamental solution on the tube}. 
    Notice that $U_k$ has been represented as a tube, 
    as it has been constructed using the tubular neighborhood theorem 
    (see the proof of \cref{lem:variable change}).
    The grey area represents the support of~$\widetilde \Lambda_k$.
    Since $\widetilde \Lambda_k = 0$ outside of it, 
    there are no continuity issues on the boundary~$\partial U_k \cap \Omega$.}
    \label{fig:gammatilde}
\end{figure}

Notice that $\widetilde \Lambda_k$ coincides with $\widehat \Lambda_k + L$ around~$x=x^{(k)}$,
in view of the divergence of~$\Lambda$ to~$-\infty$ at this point.
Fix any~$R > 0$ such that~$B(x^{(k)}, R) \cap \Omega \subset U_k$.
Clearly, there exists~$L > 0$ so that
\[
    \mathrm{Supp} \Bigl( \widetilde \Lambda_k \Bigr) \subset \Omega \cap U_k \cap B\bigl(x^{(k)}, \tfrac{R}{2}\bigr) .
\]
In particular $\widetilde \Lambda_k(x) = 0$ for all $x \in \Omega$ 
close to the boundary~$\partial U_k \cap \Omega$. 
From now on, since $L$ is fixed in the remainder of the analysis, 
we use the shorthand notation $\eta = \eta_L$.
Hence~$\widetilde \Lambda_k$ is a smooth function on~$\Omega$ (see \cref{fig:gammatilde}). 
We furthermore derive the following properties of~$\widetilde \Lambda_k$.

\begin{cor}\label{cor:tilde gamma}
    The function $\widetilde \Lambda_k$ defined in~\eqref{eq:fundamental solution everywhere} 
    is smooth on $\overline{\Omega}\backslash \{x^{(k)}\}$ and satisfies 
    \begin{equation}\label{eq:normal derivative tilde gamma}
        \partial_n \widetilde \Lambda_k(x) = 0 \quad \text{on } \partial \Omega \setminus \{ x^{(k)}\},
    \end{equation}
    \begin{equation}\label{eq:regularity laplacian tilde gamma}
        \widetilde \Lambda_k \in W^{1, p}(\Omega) \quad \text{and}\quad \Delta \widetilde \Lambda_k \in L^p(\Omega) \quad \text{for all } 1 < p < \frac{d}{d-1}.
    \end{equation}
    Finally, one has the following pointwise bound for a constant $\widetilde{C}_{\Lambda} > 0$:
    \begin{equation}\label{eq:scaling laplacian tilde gamma}
        \forall x \in \Omega, \qquad
        \bigl\lvert \Delta \widetilde{\Lambda}_k(x) \bigr\rvert
        \leq \widetilde{C}_{\Lambda} \left|x - x^{(k)}\right|^{1-d}.
    \end{equation}
\end{cor}
\begin{proof}
    The smoothness of $\widetilde \Lambda_k$ on $\overline{\Omega}\backslash \{x^{(k)}\}$ follows from its definition and the discussion before \cref{cor:tilde gamma}.
    Since $\widetilde \Lambda_k$ coincides with $\widehat \Lambda_k + L$ around $x^{(k)}$,
    \eqref{eq:normal derivative tilde gamma} follows from \cref{lem:normal_derivative_fundamental}.
    The integrability and the pointwise bound of~$\Delta \widetilde \Lambda_k$ follow from \cref{lem:behavior_near_0} as its proof only depends on the scaling of $\widehat{\Lambda}_k$ around $x^{(k)}$.
    Finally, 
    $\widetilde \Lambda_k \in W^{1, p}(\Omega)$ follows from the scaling of $\Lambda$ in \eqref{eq:grad gamma} and the smoothness of~$\Psi_k$ in~\cref{lem:variable change}.
\end{proof}

From the integrability of $\widetilde{\Lambda}_k$ and its derivatives established in~\eqref{eq:regularity laplacian tilde gamma},
we know from \eqref{eq:normal_} that~$\partial_n \widetilde{\Lambda}_k$ is well-defined as an element of $W^{-\frac{1}{p}, p}(\partial \Omega)$ for any $1 < p < \frac{d}{d-1}$.
In the following lemma, we compute the action of $\partial_n \widetilde{\Lambda}_k$ in the duality pairing against test functions in $C^1(\partial \Omega)$.
\begin{lem}
    \label{lem:compatibility constant}
    Let~$\widetilde \Lambda_k$ be defined as in~\eqref{eq:fundamental solution everywhere}.
    For any $v \in C^{1}(\partial \Omega)$,
    \[
        \left\langle \partial_n \widetilde \Lambda_k, v \right\rangle_{W^{-\frac{1}{p}, p}(\partial \Omega), W^{\frac{1}{p}, p'}(\partial \Omega)}
        = - |\Omega| C_d v\bigl(x^{(k)}\bigr)\, ,
    \]
    where $\left\langle \cdot, \cdot \right\rangle_{W^{-\frac{1}{p}, p}(\partial \Omega), W^{\frac{1}{p}, p'}(\partial \Omega)}$
    is the duality pairing introduced in \eqref{eq:normal_}, 
    and the constant $C_d$ is defined in \cref{thm:lam0}.
    In particular,
    \[
        \int_{\Omega} \laplacian \widetilde \Lambda_k
        = - |\Omega| C_d.
    \]
\end{lem}

\begin{proof}
    For $\lambda >0$, denote by~$B_{\lambda}^{(k)}$ the open ball centered around $x^{(k)}$ of radius $\lambda$.
    Consider a test function~$v \in C^1(\partial \Omega)$, 
    and, with a slight abuse of notation, let $v \in C^1(\overline{\Omega})$ also denote an arbitrary extension of $v$ to $\overline{\Omega}$.    
    From \eqref{eq:regularity laplacian tilde gamma}, notice that $(\Delta \widetilde{\Lambda}_k) v$ is integrable. 
    Thus, the following holds
    \[
        \int_{\Omega} (\Delta \widetilde \Lambda_k) v
        = \lim_{\lambda \to 0} \int_{\Omega \setminus B_{\lambda}^{(k)}} (\Delta \widetilde \Lambda_k) v.
    \]
    Notice that $\widetilde{\Lambda}_k$ is smooth over $\Omega \setminus B_{\lambda}^{(k)}$,
    and $v \in C^1(\overline{\Omega})$,
    thus by the Green formula on $\Omega \setminus B_{\lambda}^{(k)}$, 
    \begin{equation}
        \label{eq:two_integrals}
        \begin{aligned}
        \int_{\Omega \setminus B_{\lambda}^{(k)}} (\Delta \widetilde{\Lambda}_k) v
        & =  - \int_{\Omega \setminus B_{\lambda}^{(k)}}  \nabla \widetilde{\Lambda}_k \cdot \nabla v
        + \int_{\partial \left( \Omega \setminus B_{\lambda}^{(k)} \right)} (\partial_n \widetilde{\Lambda}_k) v \\
        & = - \int_{\Omega \setminus B_{\lambda}^{(k)}}  \nabla \widetilde{\Lambda}_k \cdot \nabla v
        + \int_{\partial \Omega \setminus B_{\lambda}^{(k)}} \partial_n \widetilde{\Lambda}_k v 
        + \int_{\partial B_{\lambda}^{(k)} \cap \, \Omega} (\partial_n \widetilde{\Lambda}_k) v \\
        & = - \int_{\Omega \setminus B_{\lambda}^{(k)}}  \nabla \widetilde{\Lambda}_k \cdot \nabla v + 
        0
        + \int_{\partial B_{\lambda}^{(k)} \cap \, \Omega} (\partial_n \widetilde{\Lambda}_k) v.
        \end{aligned}
    \end{equation}
    The first term in the last line of \eqref{eq:two_integrals} converges to $- \int_{\Omega}  \nabla \widetilde{\Lambda}_k \cdot \nabla v$ as $\lambda \to 0$, 
    by the integrability of the integrand.
    Hence, injecting \eqref{eq:two_integrals} in the definition of the duality pairing \eqref{eq:normal_} 
    and taking the limit~$\lambda \to 0$,
    yields
    \begin{equation}\label{eq:duality_pairing_in_the_proof}
        \left\langle \partial_n \widetilde \Lambda_k, v \right\rangle_{W^{-\frac{1}{p}, p}(\partial \Omega), W^{\frac{1}{p}, p'}(\partial \Omega)}
        = \int_{\Omega} (\Delta \widetilde \Lambda_k) v + \int_{\Omega} \nabla \widetilde \Lambda_k \cdot \nabla v
        = \lim_{\lambda \to 0} \int_{\partial B_{\lambda}^{(k)} \cap \, \Omega} (\partial_n \widetilde{\Lambda}_k) v.
    \end{equation}
    It remains to understand the limit as $\lambda \to 0$ of
    the right-hand side of~\eqref{eq:duality_pairing_in_the_proof}.
    Writing the normal derivative more explicitly,
    and recalling that $\widetilde \Lambda_k$ coincides with $\widehat \Lambda_k + L$ around $x^{(k)}$,
    we deduce that
    \begin{equation}\label{eq:integral laplacian gamma tilde}
        \int_{\partial B_{\lambda}^{(k)} \cap \, \Omega} \partial_n \widetilde{\Lambda}_k v
        = - \int_{\partial B_{\lambda}^{(k)} \cap \, \Omega} \frac{x - x^{(k)}}{|x - x^{(k)}|} \cdot \nabla \widehat \Lambda_k(x) v(x) \, \sigma(\d x).
    \end{equation}
    Similarly to the proof of \cref{lem:variable change}, 
    without loss of generality, we can consider that $\xk = 0$ and that the outward normal vector to $\partial \Omega$ at $x^{(k)}$ is $e_d = (0, \dotsc, 0, 1)$.
    Using~\cref{lem:inverse psi} 
    and reasoning as in~\cref{lem:behavior_near_0},
    we obtain that, in the limit~$x \to 0$,
    \begin{align*}
        \frac{x}{|x|} \cdot \nabla \widehat\Lambda_k(x)
        & =
        \frac{x}{|x|} \cdot \Bigl(\left(\nabla \Psi_k\right)^{\t}\!(x) \nabla \Lambda \circ \Psi_k(x)\Bigr) \\
        &= 
        \frac{x}{|x|} \cdot \biggl( 
            \left(\nabla \Psi_k\right)^{\t}\!(x) \Bigl(\nabla \Lambda \circ \Psi_k(x)  - \nabla \Lambda (x) \Bigr)
            + \bigl( \left(\nabla \Psi_k\right)^{\t}\!(x) - \I_d \bigr)  \nabla \Lambda (x) 
            + \nabla \Lambda(x) 
        \biggr) \, \\
        &= 
        \frac{x}{|x|} \cdot \biggl( 
            \bigl( \I_d + \mathrm{O}(|x|)\bigr) \Bigl(\nabla \Lambda \circ \Psi_k(x)  - \nabla \Lambda (x) \Bigr) + \mathrm O(|x|)  \frac{x}{|x|^{d}} + \max\{d-2, 1 \}\frac{x}{|x|^{d}}
        \biggr) \, ,
    \end{align*}
    where we used the explicit expression of $\nabla \Lambda$ in the last equation.
    To bound the first term, 
    we note that
    \begin{equation}
        \label{eq:bound gradient gamma tilde}
        \begin{aligned}
        \left\lvert \nabla \Lambda \circ \Psi_k(x)  - \nabla \Lambda (x)  \right\rvert
        &= \max\{d-2, 1 \}
        \left\lvert \frac{\Psi_k(x)}{\left\lvert \Psi_k(x) \right\rvert^d}  
        - \frac{x}{\left\lvert x \right\rvert^d}  \right\rvert \\
        &= \max\{d-2, 1 \}
        \left\lvert \Psi_k(x) \left( \frac{1}{\left\lvert \Psi_k(x) \right\rvert^d} - \frac{1}{|x|^d} \right)
        + \frac{\Psi_k(x) - x}{\left\lvert x \right\rvert^d}  \right\rvert
        = \mathrm O\bigl(|x|^{2-d}\bigr).
    \end{aligned}\end{equation}
    Thus, we finally conclude that
    \begin{equation}\label{eq:limit}\begin{aligned}
        \int_{\partial B_{\lambda}^{(k)} \cap \, \Omega} \frac{x}{|x|} \cdot \nabla \widehat \Lambda_k (x)v(x)\, \sigma( \d x)
        &= \max\{d-2, 1 \} \int_{\partial B_{\lambda}^{(k)} \cap \, \Omega} v(x) \frac{1 + \mathrm O(\lambda) }{\lambda^{d-1}}\, \sigma( \d x)
        \\
        &\xrightarrow[\lambda \to 0]{}
        \frac{\max\{d-2, 1 \}}{2} \omega_d \, v(0).
    \end{aligned}\end{equation}
    The limit is obtained by noticing that
    \begin{align} 
        \notag
        \int_{\partial B_{\lambda}^{(k)} \cap \, \Omega} v(x) \frac{1}{|x|^{d-1}}\, \sigma( \d x)
        &= \frac{1}{\lambda^{d-1}} \int_{\partial B_{\lambda}^{(k)}} v(x) \1_{\{x_d \leq h_k(x_1, \dotsc, x_{d-1})\}}\, \sigma( \d x)\\
        \notag
        &= \int_{\partial B_1} v(\lambda x) \1_{\{\lambda x_d \leq h_k(\lambda x_1, \dotsc, \lambda x_{d-1})\}}\, \sigma( \d x)
    \end{align}
    where we used \eqref{eq:local description omega} which follows from the local description of $\Omega$ around $0$
    fixed by the choice of coordinates. 
    In view of~\eqref{eq:expansion h}
    it holds that
    \[
        \left\lvert \frac{h_k(\lambda x_1, \dotsc, \lambda x_{d-1})}{\lambda} \right\rvert \xrightarrow[\lambda \to 0]{}   0.
    \]
    And thus, by the dominated convergence theorem, we obtain
    \begin{equation*}
        \int_{\partial B_{\lambda}^{(k)} \cap \, \Omega} v(x) \frac{1}{|x|^{d-1}}\, \sigma( \d x) \xrightarrow[\lambda \to 0]{} v(0) \int_{\partial B_1} \1_{\{x_d \leq 0\}}\, \sigma( \d x),
    \end{equation*}
    
    The result follows by combining~\eqref{eq:duality_pairing_in_the_proof}, \eqref{eq:integral laplacian gamma tilde}, \eqref{eq:limit} 
    and the expression of $C_d$ from \cref{thm:lam0}.
\end{proof}

We show in \cref{thm:fk} that the solution of \eqref{eq:fkepd distributional} can be decomposed as~$\widetilde \Lambda_k + S_k$,
with~$\widetilde \Lambda_k$ the most singular term and $S_k$ a function gathering subsingular terms.
To this end, we study the subsingular terms in the next subsection.

\subsection{\texorpdfstring{The subsingular terms $S_k$}{The subsingular terms S k}}
\label{sec:subsingular_terms}
Consider the following PDE for~$S_k$:
\begin{equation}\label{eq:subsingular_terms}
    \left\{
    \begin{aligned}
        -\Delta S_k & = \Delta \widetilde{\Lambda}_k + C_d ~~  && \text{ in } \Omega, \\
        \partial_n S_k & = 0 ~~  && \text{ on } \partial \Omega, \\
    \end{aligned}
    \right.
\end{equation}
where~$C_d$ is the constant given by~\eqref{eq:def_Cd}, introduced in~\cref{thm:lam0}.
We emphasize that the constant~$C_d$
has been chosen so that the compatibility condition \eqref{eq:compat_w1p}
 for the elliptic problem~\eqref{eq:subsingular_terms} is satisfied.
Let us state the main result of this section rigorously.

\begin{prop}\label{thm:subsingular_terms}
    Consider the function $\widetilde{\Lambda}_k$ defined in \eqref{eq:fundamental solution everywhere}
    and the constant $C_d$ defined in \eqref{eq:def_Cd}.
    Then, 
    for any $p \in (1, \frac{d}{d-1})$, 
    there exists a unique mean-zero weak $W^{1, p}(\Omega)$ solution $S_k$ 
    of~\eqref{eq:subsingular_terms}.

    Furthermore, it holds that $S_k \in C^\infty(\overline{\Omega}\backslash \{x^{(k)}\})$,  
    and there exists a constant $K_d > 0$ such that
    \begin{align}\label{eq:scaling subsingular terms}
        \forall x \in \Omega, \qquad
        \left| S_k(x)\right| &\leq  \left\{
            \begin{aligned}
            & K_2 \, &&\text{for } d = 2,\\
            & K_3 \left(1 + \left| \log|x - x^{(k)}| \right|\right)\, &&\text{for } d = 3,\\
            & K_d\, |x - x^{(k)}|^{3-d}\, &&\text{for } d \geq 4.
            \end{aligned}
        \right. 
    \end{align}
\end{prop}

Before turning to the proof of \cref{thm:subsingular_terms},
we need a technical lemma. 

\begin{lem}\label{lem:integral estimate}
    Let $d \geq 3$. 
    Then there exists~$K'_{d} > 0$ such that 
    \begin{equation}
        \label{eq:integral_estimate}
        \forall x \in \Omega, \qquad
        \int_{\Omega} |x - z|^{2-d}|z - x^{(k)}|^{1-d}\, \d z 
        \leq K'_{d} \, \ell_k(x) \, ,
    \end{equation}
    where
    \[
        \ell_k(x) =
        \begin{cases}
            1 + \left| \log |x - x^{(k)}| \right|\, & \text{ for } d = 3,\\
            \left| x - x^{(k)} \right|^{3 - d}\, & \text{ for } d \geq 4.\\
        \end{cases}
    \]
    Moreover, for $d=2$, the function
    \[
        x \mapsto \int_\Omega \left| \log|x-z| \right| |z - x^{(k)}|^{-1}\, \d z
    \]
    is in $L^\infty(\overline \Omega)$.
\end{lem}
\begin{proof}
    Up to a translation, we can assume that $x^{(k)} = 0$.
    Fix $x \in \Omega$.
    The integral in~\eqref{eq:integral_estimate} is finite since~$x \neq 0$. 
    Since~$\Omega$ is a bounded domain,
    there exists~$\zeta \geq 1$ such that $\Omega \subset B(0, \zeta)$.
    We have
    \begin{equation*}
        \int_{\Omega} |x - z|^{2-d}|z|^{1-d}\, \d z 
         \leq \int_{B(0, \zeta)} |x - z|^{2-d}|z|^{1-d}\, \d z = I_a + I_b,
    \end{equation*}
    where
    \begin{equation*}
        I_a = \int_{B(0,2|x|)} |x - z|^{2-d} |z|^{1-d} \, \d z
        \quad \text{ and } \quad I_b = {\int_{B(0, \zeta) \setminus B(0,2|x|)} |x - z|^{2-d} |z|^{1-d} \, \d z}.
    \end{equation*}
    The change of variable $z = |x| t$ in $I_a$ leads to
    \[
        I_a = |x|^{3-d}\int_{B(0, 2)} \left|e_x- t\right|^{2-d}|t|^{1-d}\, \d t,
    \]
    with $e_x = x/|x|$. 
    The integral in this equation is finite and independent of $|x|$. 
    For $I_b$,
    since $|x - z| \geq |z| - |x| \geq |z|/2$ for any~$z$ in the domain of integration,
    we have, in the limit $x \to 0$,
    \begin{align*}
        I_b &\leq
          2^{d-2}\int_{B(0, \zeta) \setminus B(0, 2|x|)} |z|^{3-2d}\, \d z
            = 2^{d-2} \omega_d \int_{2|x|}^{\zeta} r^{2-d} \, \d r \\
            &\quad = \left\{
            \begin{aligned} 
                & \mathrm O (\left| \log |x| \right|) & \text{ for } d = 3,\\
                & \mathrm O (|x|^{3-d}) & \text{ for } d \geq 4
            \end{aligned}
        \right.
    \end{align*}
    where $\omega_d$ is the surface area of the unit sphere in $\real^d$, introduced in \cref{thm:lam0}.
    The claimed estimate follows by summing up the two estimates on the integrals $I_a$ and $I_b$. 

    For $d=2$, 
    the same approach based on splitting the integration domain can be used,
    which results in the existence of positive constants~${C_a, C_b}$  such that
    \[
        \int_\Omega \left| \log|x-z| \right| |z - x^{(k)}|^{-1}\, \d z 
        \leq C_a + C_b \bigl\lvert x-x^{(k)}\bigr\rvert \left( 1+ \left| \log |x-x^{(k)}| \right| \right),
    \]
    which is indeed uniformly bounded in $x \in \overline{\Omega}$.
\end{proof}

We are now in position to prove \cref{thm:subsingular_terms}.
\begin{proof}[Proof of \cref{thm:subsingular_terms}]
    The existence and uniqueness of the function $S_k$ solution to \eqref{eq:subsingular_terms} is a consequence of \cref{prop:sol_a_la_aramaki}. 
    Indeed, 
    the estimate \eqref{eq:scaling laplacian tilde gamma} ensures that $\Delta \widetilde{\Lambda}_k + C_d \in L^{p}(\Omega)$ for $1 < p < \frac{d}{d-1}$. 
    The compatibility equation \eqref{eq:compat_w1p} is satisfied by definition of $C_d$ and \cref{lem:compatibility constant} since
    \[
        \int_{\Omega} \Delta \widetilde{\Lambda}_k + \int_\Omega C_d = 0.
    \]

    We next prove the bound \eqref{eq:scaling subsingular terms}.
    The proof is performed by constructing an explicit solution to the equation~\eqref{eq:subsingular_terms}
    using the Neumann's Green function.

    Denote by $G \colon \Omega \times \Omega \to \real$ the Green function of the Laplacian in $\Omega$ 
    with homogeneous Neumann boundary condition, i.e., the function such that the classical 
    $C^2(\Omega) \cap C^0(\overline{\Omega})$ solution of 
    \begin{equation}\label{eq:green function smooth case}
        \left\{
        \begin{aligned}
            -\Delta u & = g ~~  && \text{ in } \Omega, \\
            \partial_n u & = 0 ~~  && \text{ on } \partial \Omega, 
        \end{aligned}
        \right.
    \end{equation}
    for any $g \in C^\infty(\overline{\Omega})$ with $\int_\Omega g = 0$, is given by
    \begin{equation}\label{eq:representation formula}
        u(x) - \frac{1}{|\Omega|}\int_{\Omega} u(z) \, \d z = \int_{\Omega} G(x, z)\, g(z) \, \d z.
    \end{equation}
    It is well-known that such a function $G$ exists and is unique (see, e.g., \cite[Section 6.7]{Gilbrag}
    or \cite[Section 3, Theorem 3.1]{GreenFunctionDavid}).
    Furthermore, \cite[Section 4, see in particular (4.14)]{GreenFunctionDavid} provides a pointwise estimate on~$G$ in the case $d \geq 3$:
    there exists a constant~$G_d > 0$ such that for all $x \neq z$ in $\Omega$,
    \begin{equation}\label{eq:green function estimate}
        \left| G(x, z) \right| \leq G_d |x - z|^{2-d}.
    \end{equation}
    In the dimension $d=2$ case, \cite[Theorem 5.14]{green2d} gives
    \begin{equation}\label{eq:green function estimate d=2}
        \left| G(x, z) \right| \leq G_2 \left( 1 + \left| \log |x - z| \right| \right).
    \end{equation}
    Consider now, for any $x \in \Omega$,
    \begin{equation}\label{eq:representation formula subsingular terms}
        \widehat{S}_k(x) = \int_{\Omega} G(x, z)\, g_k(z) \, \d z,
    \end{equation}
    with 
    \begin{equation}\label{eq:def_gk}
        g_k = \Delta \widetilde{\Lambda}_k + C_d.
    \end{equation}
    The goal of the proof is to establish that $\widehat{S}_k$ is well-defined, to derive pointwise bounds for $\widehat{S}_k$, 
    and finally show that $\widehat{S}_k = S_k$, 
    where $S_k$ is the unique solution to \eqref{eq:subsingular_terms} introduced at the beginning of the proof.

    \paragraph{Well-definedness and pointwise bounds.}
    Fix $x \in \Omega$. The bound $|g_k(z)| \leq C |z - x^{(k)}|^{1-d}$ for $z \in \Omega$ with some $C>0$ positive,from \eqref{eq:scaling laplacian tilde gamma},
    combined with \eqref{eq:green function estimate} and inserted into \eqref{eq:representation formula subsingular terms},
    gives, for a constant $K_d' > 0$,
    \begin{equation}\label{eq:bound s hat}
        \left| \widehat{S}_k(x) \right| \leq K_d' \int_{\Omega} \bigl\lvert x - z\bigr\rvert ^{2-d} \bigl\lvert z - x^{(k)} \bigr\rvert^{1-d} \, \d z, 
    \end{equation}
    in dimension $d \geq 3$.
    Similarly, in dimension $d=2$, one has
    \begin{equation}\label{eq:bound s hat 2}
        \left| \widehat{S}_k(x) \right| \leq K_d' \left(1 + \int_{\Omega} \Bigl\lvert  \log\bigl\lvert x - z\bigr\rvert \Bigr\rvert \bigl\lvert z - x^{(k)} \bigr\rvert^{1-d} \, \d z \right).
    \end{equation}
    As $x \neq x^{(k)}$, 
    and the integrands in \eqref{eq:bound s hat} and \eqref{eq:bound s hat 2} are both integrable locally around $x$ and $x^{(k)}$, 
    the function~$\widehat{S}_k$ is well defined. 
    Applying \cref{lem:integral estimate} to the right-hand side of \eqref{eq:bound s hat} and \eqref{eq:bound s hat 2}
    we obtain, for all $x \in \Omega$, the bound~\eqref{eq:scaling subsingular terms} for the function $\widehat{S}_k$.

    \paragraph{Proof that $\widehat{S}_k = S_k$.}
    \Cref{eq:representation formula} is known as the Green representation formula 
    of the Poisson problem \eqref{eq:green function smooth case}
    (see for instance \cite[Equation (3.4)]{GreenFunctionDavid} or \cite[Theorem 12, Chapter 2]{Evans10}).
    We provide a self-contained proof that this formula also provides the weak solution
    in the sense of \eqref{eq:weak formulation} of \eqref{eq:green function smooth case} with~$g_k \not \in C^\infty(\overline{\Omega})$, given in \eqref{eq:def_gk}, as the right-hand side.

    Define the function $\widehat{S}_{k, \delta}$ by,
    \begin{equation}\label{eq:def s hat n}
        \forall x \in \Omega, \qquad \widehat{S}_{k, \delta}(x) = \int_{\Omega} G(x, z) g_{k, \delta}(z) \, \d z,
    \end{equation}
    with 
    \begin{equation} 
        \forall z \in \Omega, \qquad g_{k, \delta}(z) = \rho_\delta(z) g_k(z) - \frac{1}{|\Omega|} \int_\Omega \rho_\delta(t)  g_k(t) \, \d t
    \end{equation}
    where $\rho_\delta\colon \overline{\Omega} \to [0, 1]$ is a smooth function such that  
    \begin{equation}\label{eq:def rho delta}
        \rho_\delta = 0 \text{ on } B\left(x^{(k)}, \frac{\delta}{2}\right) 
        \qquad \text{ and } \qquad 
        \rho_\delta = 1 \text{ on } \overline{\Omega} \setminus B\left(x^{(k)}, \delta\right).
    \end{equation}
    Notice that $g_{k, \delta}$ is a regularization of $g_k = \Delta \widetilde{\Lambda}_k + C_d$ and that $\int_\Omega g_{k, \delta} = 0$ for any $\delta > 0$.
    It is clear that, by construction, $g_{k, \delta}$
    converges to $g_k$ in $L^q$ and uniformly on compact sets for any $q \in [1, \frac{d}{d-1})$ as $\delta \to 0$.
    Furthermore,
    notice that 
    $g_{k, \delta}$ is smooth on $\overline{\Omega}$, 
    as $\Delta \widetilde{\Lambda}_k$ is smooth on $\overline{\Omega} \setminus \{x^{(k)}\}$ from \cref{cor:tilde gamma},
    and~$\rho_\delta$ is smooth and vanishes in a neighborhood of~$x^{(k)}$.
    Since we are now in the classical smooth setting (see \eqref{eq:green function smooth case} and \eqref{eq:representation formula}),
    the function $\widehat{S}_{k, \delta}$ is the zero-mean solution of 
    \[
        \left\{
            \begin{aligned}
            -\Delta \widehat{S}_{k, \delta} & = g_{k, \delta} ~~  && \text{ in } \Omega, \\
            \partial_n \widehat{S}_{k, \delta} & = 0 ~~  && \text{ on } \partial \Omega, \\
    \end{aligned}
        \right.
    \]
    in the classical sense and thus in the sense of \eqref{eq:weak formulation}. 
    The stability estimate from \cite[Proposition 2.3]{aramaki2018existence} yields a constant $C > 0$ such that, 
    for $p \in (1, \frac{d}{d-1})$,
    \begin{equation}
        \left\| S_k - \widehat{S}_{k, \delta} \right\|_{W^{1, p}(\Omega)} 
        \leq C \left\| g_k - g_{k, \delta} \right\|_{L^p(\Omega)} \xrightarrow[\delta \to 0]{} 0.
    \end{equation}
    Hence $\widehat{S}_{k, \delta}$ converges to $S_k$ in $W^{1, p}(\Omega)$.
    We now show that $\widehat{S}_{k, \delta}$ converges to $\widehat{S}_k$ pointwise in $\Omega$,
    so that~$\widehat{S}_k = S_k$ will follow.
    Fix $x \in \Omega$ and write 
    \begin{equation}\label{eq:integral difference}
        \left| \widehat{S}_{k, \delta}(x) - \widehat{S}_k(x) \right|
        \leq \int_{\Omega} |G(x, z)| \left| g_{k, \delta}(z) - g_k(z) \right| \, \d z.
    \end{equation}
    Fix a small neighborhood of $x$ of radius $\zeta > 0$ with $B(x, \zeta) \subset \Omega$.
    There exists $\delta_0 > 0$ such that for all~$\delta < \delta_0$,
    it holds $\rho_\delta g_k = g_k$ on~$B(x, \zeta)$.
    Hence,
    \[  
        \lim_{\delta \to 0} g_{k, \delta} (z) - g_k(z) = 
        \lim_{\delta \to 0} - \frac{1}{|\Omega|} \int_\Omega \rho_\delta g_k
        = 0 \quad \text{ uniformly in } z \in B(x, \zeta).
    \]
    Thus the integral over $B(x, \zeta)$ satisfies 
    \begin{equation}\label{eq:integral convergence in zeta}
        \int_{B(x, \zeta)} |G(x, z)| \left| g_{k, \delta}(z) - g_k(z) \right| \, \d z
        \leq \left\| G(x, \cdot )\right\|_{L^1(B(x, \zeta))} \left\| g_{k, \delta} - g_k \right\|_{L^\infty(B(x, \zeta))}.
    \end{equation}
    Since $G(x, \cdot)$ is bounded on $\Omega \setminus B(x, \zeta)$,
    the integral over $\Omega \setminus B(x, \zeta)$ satisfies
    \begin{equation}\label{eq:integral convergence out of zeta}
        \int_{\Omega \setminus B(x, \zeta)} |G(x, z)| \left| g_{k, \delta}(z) - g_k(z) \right| \, \d z
        \leq \left\| G(x, \cdot) \right\|_{L^\infty(\Omega \setminus B(x, \zeta))} \left\| g_{k, \delta} - g_k \right\|_{L^1(\Omega \setminus B(x, \zeta))}.
    \end{equation}
    Both the right-hand side of~\eqref{eq:integral convergence in zeta} and the right-hand side of~\eqref{eq:integral convergence out of zeta}
    tend to $0$ as $\delta \to 0$.
    Combined with~\eqref{eq:integral difference}, 
    this shows that $\widehat{S}_{k, \delta}$ converges pointwise in $\Omega$ to $\widehat{S}_k$.
    Together with the $W^{1,p}(\Omega)$ convergence of $\widehat{S}_{k, \delta}$ to $S_k$, we conclude that $\widehat{S}_k = S_k$ almost everywhere,
    and the estimate \eqref{eq:scaling subsingular terms} on $\widehat{S}_k$ completes the proof.
\end{proof}

\subsection{\texorpdfstring{The explicit solution to the point source problem \eqref{eq:fkepd distributional}}{The explicit solution to the point source problem}}
\label{sec:actual construction fk}
        We now gather the previous results to obtain an explicit expression for the solution of \eqref{eq:fkepd distributional}.
\begin{thm}\label{thm:fk}
    Consider the function $f_k$ introduced in \cref{def:fk}. Then,
    $f_k \in {C^{\infty}(\overline{\Omega}\setminus\{ x^{(k)}\})}$ 
    and 
    \begin{equation}\label{eq: def fk}
        f_k =  \widetilde{\Lambda}_k + S_k - \frac{1}{|\Omega|} \int_\Omega \widetilde{\Lambda}_k.
    \end{equation}
    As $x \to x^{(k)}$, the function $f_k$ satisfies:
    \begin{equation}
        \label{eq:estimate_fk}
        f_k(x)  = \Lambda(x - x^{(k)}) + \left \{ \begin{aligned}
            & \mathrm{O}(1)\, && \text{ for } d = 2,\\
            & \mathrm{O}\bigl(\log |x - x^{(k)}|\bigr)\, && \text{ for } d = 3,\\
            &  \mathrm{O}\bigl(|x - x^{(k)}|^{3-d}\bigr)\, && \text{ for } d \geq 4.
        \end{aligned} \right.
    \end{equation}
\end{thm}

\begin{proof}
    The proof consists in showing that $\widetilde{\Lambda}_k + S_k - \frac{1}{|\Omega|} \int_\Omega \widetilde{\Lambda}_k$ 
    is the unique solution to \eqref{eq:fkepd distributional} granted by~\cref{prop:sol_a_la_aramaki}.
    To this end, consider a function~$v \in C^1(\overline{\Omega})$.
    
    From \cref{lem:compatibility constant} and the definition of the duality pairing in \eqref{eq:normal_}, 
    one has that 
    \begin{equation}\label{eq:weak Lambda_k}
        \int_\Omega \nabla \widetilde{\Lambda}_k \cdot \nabla v
        + \int_\Omega (\Delta \widetilde{\Lambda}_k) \, v
        = - |\Omega| C_d v(x^{(k)}).
    \end{equation}
    Additionally, 
    since $S_k$ has been defined as the solution of \eqref{eq:subsingular_terms} in the sense of \eqref{eq:weak formulation}
    in \cref{thm:subsingular_terms}, 
    the following holds:
    \begin{equation}\label{eq:weak Sk}
        \int_\Omega \nabla S_k \cdot \nabla v
        = \int_{\Omega} \left( \Delta \widetilde{\Lambda}_k + C_d \right) v.
    \end{equation}
    By summing \eqref{eq:weak Lambda_k} and \eqref{eq:weak Sk}
    and subtracting the constant term $\frac{1}{|\Omega|} \int_\Omega \widetilde{\Lambda}_k$ in the gradient,
    one obtains
    \begin{equation}\label{eq:weak Lambda_k + Sk}
        \int_\Omega \nabla \left(\widetilde{\Lambda}_k + S_k - \frac{1}{|\Omega|} \int_\Omega \widetilde{\Lambda}_k\right) \cdot \nabla v
        = C_d \left( \int_\Omega v \right) - |\Omega| C_d v(x^{(k)}).
    \end{equation}
    This is exactly the weak formulation \eqref{eq:weak formulation} of \eqref{eq:fkepd distributional}.
    Since $S_k$ has a zero mean on $\Omega$ by \cref{thm:subsingular_terms},
    so does $\widetilde{\Lambda}_k + S_k - \frac{1}{|\Omega|} \int_\Omega \widetilde{\Lambda}_k$.
    This function also belongs to $W^{1, p}(\Omega)$ for $p \in (1, \frac{d}{d-1})$ by \cref{thm:subsingular_terms} and \cref{cor:tilde gamma}.
    Finally,
    by the uniqueness of the solution to \eqref{eq:fkepd distributional} provided by \cref{prop:sol_a_la_aramaki},
    we conclude that 
    \[
        f_k = \widetilde{\Lambda}_k + S_k - \frac{1}{|\Omega|} \int_\Omega \widetilde{\Lambda}_k.
    \]
    It remains to obtain the estimates \eqref{eq:estimate_fk}. These are a consequence of \cref{thm:subsingular_terms}, 
    the definition of $\widetilde{\Lambda}_k$ from \eqref{eq:fundamental solution everywhere}
    as well as the estimates on $\Psi_k$ from \cref{lem:inverse psi} and the definition of $\Lambda$ given in~\eqref{eq:fundamental_solution}.
    Indeed, 
    in view of \cref{lem:inverse psi}, for $d \geq 3$, 
    it holds as~$x \to x^{(k)}$ that
    \[
        \bigl| \widetilde \Lambda_k(x) - \Lambda\bigl(x - x^{(k)}\bigr) \bigr| 
        = \Bigl| - \left| \Psi_k(x)\right|^{2-d} + L + \bigl\lvert x - x^{(k)}\bigr\rvert^{2-d} \Bigr| 
        = \mathrm{O}\left( \bigl\lvert x - x^{(k)}\bigr\rvert^{3-d}\right),
    \]
    thanks to $\left| \Psi_k(x) - (x - x^{(k)})\right| = \mathrm{O}(|x - x^{(k)}|^2)$. 
    For $d=2$, the claimed estimate follows from a similar argument.
\end{proof}

In addition to the estimates \eqref{eq:estimate_fk} on $f_k$ from \cref{thm:fk},
we can also obtain $L^2(\Omega_{\varepsilon})$ estimates on~$f_k$, 
where $\Omega_\varepsilon$ is the domain with $N$ holes introduced in \eqref{eq:omega varepsilon}.
\begin{lem}\label{lem:L2 estimate fk}
    In the limit $\varepsilon \to 0$, the following estimates on the $L^2$ norm of $f_k$ hold:
    \begin{equation}\label{eq:estimate fk_L2_lemma}
        \|f_k\|_{L^2(\Omega_\varepsilon)}^2 = \left \{ \begin{aligned}
            & \mathrm{O}(1)\, && \text{ for } d \leq 3,\\
            & \mathrm{O}(|\log \re{(k)}|)\, && \text{ for } d = 4,\\
            & \mathrm{O}\left(\left(\re{(k)}\right)^{4-d}\right)\, && \text{ for } d \geq 5.
        \end{aligned} \right.
    \end{equation}
\end{lem}
\begin{proof}
    Consider a fixed $\zeta > 0$, 
    small enough so that the estimates \eqref{eq:estimate_fk} on $f_k$ from \cref{thm:fk} hold for all $x \in B(x^{(k)}, \zeta)$, 
    and $\varepsilon > 0$ such that $r_\varepsilon^{(k)} < \zeta$.
    We decompose the squared $L^2(\Omega_\varepsilon)$ norm of $f_k$ as 
    \begin{equation}\label{eq:estimate fk_L2}
        \|f_k\|_{L^2(\Omega_\varepsilon)}^2 = 
        \|f_k\|_{L^2(\Omega_\varepsilon \setminus B(x^{(k)}, \zeta))}^2
        + \|f_k\|_{L^2(\Omega_\varepsilon \cap B(x^{(k)}, \zeta))}^2.
    \end{equation}
    Notice first that 
    $\Omega_\varepsilon \setminus B(x^{(k)}, \zeta) \subset \Omega \setminus B(x^{(k)}, \zeta)$
    and $f_k$ is smooth on $\overline{\Omega} \setminus B(x^{(k)}, \zeta)$ from \cref{thm:fk}, 
    hence there exists a constant $C>0$, independent of $\varepsilon$, such that
    \begin{equation}\label{eq:estimate fk_L2 far contribution}
        \|f_k\|_{L^2(\Omega_\varepsilon \setminus B(x^{(k)}, \zeta))} 
        \leq \|f_k\|_{L^2(\Omega \setminus B(x^{(k)}, \zeta))} 
        \leq C.
    \end{equation}
    The second contribution requires using the pointwise estimates on $f_k$ from \cref{thm:fk}.
    By these estimates,
    there exists a constant $C' > 0$ independent of $x$ such that, 
    for any $\varepsilon > 0$,
    \begin{equation} \label{eq:estimate fk_L2 close contribution}\begin{aligned}
        \|f_k\|_{L^2(\Omega_\varepsilon \cap B(x^{(k)}, \zeta))}^2
        & \leq C' \int_{\Omega_\varepsilon \cap B(x^{(k)}, \zeta)} \left| \Lambda \bigl(x-x^{(k)} \bigr) \right|^2 \, \d x\\
         & \leq C' \int_{B(x^{(k)}, \zeta) \setminus B(x^{(k)}, r_\varepsilon^{(k)})} \left| \Lambda \bigl(x-x^{(k)} \bigr) \right|^2 \, \d x.
    \end{aligned}\end{equation}
    On the annulus $B(x^{(k)}, \zeta) \setminus B(x^{(k)}, r_\varepsilon^{(k)})$, 
    we can integrate $\Lambda$ in polar coordinates.
    In dimension $d \geq 3$:
    \begin{equation}\label{eq:estimate fk_L2 close contribution final}
        \int_{B(x^{(k)}, \zeta) \setminus B(x^{(k)}, r_\varepsilon^{(k)})} \left| \Lambda \bigl(x-x^{(k)} \bigr) \right|^2 \, \d x 
        = \omega_d \int_{r_\varepsilon^{(k)}}^\zeta \rho^{2(2-d)} \rho^{d-1}\, \d \rho  \\
        = \left \{ \begin{aligned}            
            & \mathrm{O}(1)\, && \text{ for } d = 3,\\
            & \mathrm{O}(|\log \re{(k)}|)\, && \text{ for } d = 4,\\
            & \mathrm{O}\left(\left(\re{(k)}\right)^{4-d}\right),\, && \text{ for } d \geq 5,
        \end{aligned}
        \right.
    \end{equation}
    where the constant $\omega_d$ is the surface area of the unit sphere in dimension $d$ already
    introduced in \cref{thm:lam0}.
    Besides, in dimension $d=2$, 
    \begin{equation}\label{eq:estimate fk_L2 close contribution final d=2}
        \int_{B(x^{(k)}, \zeta) \setminus B(x^{(k)}, r_\varepsilon^{(k)})} \left| \Lambda \bigl(x-x^{(k)} \bigr) \right|^2 \, \d x 
        = \omega_2 \int_{r_\varepsilon^{(k)}}^\zeta |\log \rho|^2 \rho\, \d \rho  = \mathrm{O}(1).
    \end{equation}
    Plugging the estimates \eqref{eq:estimate fk_L2 far contribution},
    \eqref{eq:estimate fk_L2 close contribution final}
    and \eqref{eq:estimate fk_L2 close contribution final d=2} into 
    \eqref{eq:estimate fk_L2} leads to the result. 
\end{proof}

\section{Quasimode analysis and proof of the main results}
\label{sec:quasi-mode}
Let us now consider the quasi-mode $ \varphi_\varepsilon$ defined in \eqref{eq:def phi},
with parameters~$\Kek$ from \eqref{eq:def K} and the functions~$f_k$ solving~\eqref{eq:fkepd distributional}.
The key to prove our main results
is to compare the quasi-mode $\varphi_\varepsilon$ 
to the quasi-stationary distribution $\nu_\varepsilon$ (see \cref{thm:QSD}).
Before doing this, 
we state in \cref{sec:prop quasimode} all the properties on $\varphi_\varepsilon$ needed to this end.
We next move to the proof of the main results, 
starting with~\cref{thm:lam0} in~\cref{sec:mean exit} and then \cref{thm:exit hole distribution} in \cref{sec:exit hole}.

\subsection{\texorpdfstring{Results on the quasi-mode $\varphi_\varepsilon$}{Results on the quasi-mode phi epsilon}}
\label{sec:prop quasimode}
Let us start by estimating the $L^\infty$ norm of the quasi-mode $\varphi_\varepsilon$ on the Dirichlet boundary  $\Gamma_{\mathcal{D}}^{\varepsilon}$.

\begin{lem}\label{lem:quasi-mode}
    Let the parameters $\Kek$ be given by \eqref{eq:def K} for $k \in \{1, \ldots, N \}$.
    Then, in the limit $\varepsilon \to 0$, 
    the quasi-mode $\varphi_\varepsilon$ defined in \eqref{eq:def phi} satisfies:
    \[
        \lVert \varphi_\varepsilon \rVert_{L^{\infty}(\Gamma_{\mathcal{D}}^{\varepsilon})}
        = \mathrm{O}\left( \mathcal{E}_d(\Keo)\right) 
    \]
    where we recall that $\Keo = \sum_{k=1}^{N} \Kek$ and $\mathcal{E}_d$ is defined in \eqref{eq:def_Ed}.
\end{lem}

\begin{proof}
    For $x \in \Gamma_\mathcal{D}^{\varepsilon}$, 
    there exists $k \in \{ 1, \ldots, N \}$ such that $x \in \Gamma^{\varepsilon}_k$.
    The quasi-mode~$\varphi_\varepsilon$ can then be expressed as:
    \begin{equation}\label{eq:quasi-mode on the boundary}
         \varphi_\varepsilon(x) = 1 + \Kek f_k(x) +
        \sum_{j=1,\, j \neq k}^{N}\Ke^{(j)} f_j(x).
    \end{equation}
    From the disjointness of the holes
    and the regularity of $f_j$ given by~\cref{thm:fk} (namely $f_j \in C^\infty(\overline{\Omega} \setminus \{x^{(j)}\})$),
    the function $f_j$ is bounded uniformly in $\varepsilon$ in the limit $\varepsilon \to 0$ on $\Gamma_k^\varepsilon$ for~$j \neq k$.
    Hence, in the limit $\varepsilon \to 0$, 
    \begin{equation}\label{eq:estimate_other_fk}
        \left\| \sum_{j=1,\, j \neq k}^{N}\Ke^{(j)} f_j \right\|_{L^\infty(\Gamma_k^\varepsilon)} = \mathrm{O}(\Keo).
    \end{equation}
    Now, notice that the expression of $\Kek$ in \eqref{eq:def K}
    is simply $\Kek = - \left( \Lambda(x - \xk)\right)^{-1}$,
    which is indeed independent of $x \in \Gamma_k^{\varepsilon}$ as $\Lambda$ is a radial function.
    We can then use the estimate \eqref{eq:estimate_fk} on $f_k$ from~{\cref{thm:fk}} to obtain
    uniformly in $x \in \Gamma_k^\varepsilon$, when $\varepsilon \to 0$:
    \begin{equation}\label{eq:Kek fk estimate}       
        \begin{aligned}
        \Kek f_k(x)
        & = - 1
        + \left\{
            \begin{aligned}
                & \mathrm{O}\left(\left|\log \re{(k)}\right|^{-1}\right),\, && \text{ for } d = 2,\\
                & \mathrm{O}\left(\re{(k)}\,\left|\log \re{(k)}\right|\right),\, && \text{ for } d = 3,\\
                &  \mathrm{O}\left(\re{(k)}\right),\, && \text{ for } d \geq 4.
            \end{aligned}
        \right.\\
        \end{aligned}
    \end{equation}
    Combining the two bounds \eqref{eq:estimate_other_fk} and \eqref{eq:Kek fk estimate} in \eqref{eq:quasi-mode on the boundary},
    we find that, in the limit $\varepsilon \to 0$,
    \begin{equation}\label{eq:quasi-mode on the boundary estimate unsumed}
        \lVert \varphi_\varepsilon \rVert_{L^{\infty}(\Gamma_{k}^{\varepsilon})}
        = \mathrm{O}\left( \Keo \right) + \left\{
            \begin{aligned}
                & \mathrm{O}\bigl(\Kek\bigr)\, && \text{ for } d = 2,\\
                & \mathrm{O}\Bigl(\Kek \left| \log\bigl(\Kek\bigr)\right|\Bigr)\, && \text{ for } d = 3,\\
                & \mathrm{O}\Bigl(\left(\Kek\right)^{\frac{1}{d-2}}\Bigr)\, && \text{ for } d \geq 4.
            \end{aligned}
        \right.
    \end{equation}
    Repeating this for each $k \in \{1, \ldots, N\}$ shows that 
    $\lVert \varphi_\varepsilon \rVert_{L^{\infty}(\Gamma_{\mathcal{D}}^{\varepsilon})}$ can be bounded by the sum over $k$ 
    of the right-hand side of \eqref{eq:quasi-mode on the boundary estimate unsumed}. 
    Finally, Jensen's inequality implies
    \[
        \sum_{k=1}^N \left(\Kek\right)^{\frac{1}{d-2}}
        \leq N^{1-\frac{1}{d-2}}\left(\sum_{k=1}^N \Kek\right)^{\frac{1}{d-2}} = \mathrm{O} \left( \Keo^{\frac{1}{d-2}}\right),
    \]
    which leads to the desired estimate, in the limit $\varepsilon \to 0$ in dimension $d \geq 4$.
    In dimension $d=3$,
    the result follows from the bound
    \begin{equation}\label{eq:log fake jensen}
        \sum_{k=1}^N \Kek \left| \log\bigl(\Kek\bigr)\right|
        = \mathrm{O}\left( \Keo |\log \Keo|\right),
    \end{equation}
    where we used that the function $x \mapsto x |\log x|$ is increasing on $(0, e^{-1})$, 
    and $\Kek \leq \Keo$.
\end{proof}

Let us now state the properties of the quasi-mode $\varphi_\varepsilon$ 
that we need in the proof of our main results.
\begin{prop}\label{prop:properties quasi-mode}
    The quasi-mode $\varphi_\varepsilon$ defined in \eqref{eq:def phi} satisfies:
    % \begin{enumerate}[labelwidth=6em, leftmargin=!, align=left]
    %     \renewcommand{\theenumi}{Regularity}
    %     \item[\textbf{(Regularity.)\enspace}] \label{prop:quasi-mode smooth} $\varphi_\varepsilon$ is smooth on $\overline{\Omega_\varepsilon}$, in particular $\varphi_\varepsilon \in H^1(\Omega_\varepsilon)$;
    %     \renewcommand{\theenumi}{Laplacian}
    %     \item[\textbf{(Laplacian.)\enspace}] \label{prop:laplacian quasi-mode} $-\Delta \varphi_\varepsilon = C_d \Keo$ where $C_d$ is given in \eqref{eq:def_Cd} and $\Keo$ by \eqref{eq:def Keo};
    %     \renewcommand{\theenumi}{Neumann}
    %     \item[\textbf{(Neumann.)\enspace}] \label{prop:neumann quasi-mode} $\partial_n \varphi_\varepsilon = 0$ on $\Gamma_\mathcal{N}^\varepsilon$;
    %     \renewcommand{\theenumi}{Dirichlet}
    %     \item[\textbf{(Dirichlet.)\enspace}] \label{prop:Linf quasi-mode} $\|\varphi_\varepsilon\|_{L^\infty(\Gamma_{\mathcal{D}}^{\varepsilon})} = \mathrm{O}(\mathcal{E}_d(\Keo))$, where $\mathcal{E}_d$ is defined in \eqref{eq:def_Ed};
    %     \renewcommand{\theenumi}{Normalization}
    %     \item[\textbf{(Normalization.)\enspace}] \label{prop:L2 quasi-mode} $\|\varphi_\varepsilon - 1\|_{L^2(\Omega_\varepsilon)} =  \mathrm{O}(\mathcal{E}_d(\Keo))$.
    % \end{enumerate}
    \begin{enumerate}[label=\textbf{(\alph*)}]
        \item (Regularity) \label{prop:quasi-mode smooth} $\varphi_\varepsilon$ is smooth on $\overline{\Omega_\varepsilon}$, in particular $\varphi_\varepsilon \in H^1(\Omega_\varepsilon)$;
        \item (Laplacian) \label{prop:laplacian quasi-mode} $-\Delta \varphi_\varepsilon = C_d \Keo$ where $C_d$ is given in \eqref{eq:def_Cd} and $\Keo$ by \eqref{eq:def Keo};
        \item (Neumann) \label{prop:neumann quasi-mode} $\partial_n \varphi_\varepsilon = 0$ on $\Gamma_\mathcal{N}^\varepsilon$;
        \item (Dirichlet) \label{prop:Linf quasi-mode} $\|\varphi_\varepsilon\|_{L^\infty(\Gamma_{\mathcal{D}}^{\varepsilon})} = \mathrm{O}(\mathcal{E}_d(\Keo))$, where $\mathcal{E}_d$ is defined in \eqref{eq:def_Ed};
        \item (Normalization) \label{prop:L2 quasi-mode} $\|\varphi_\varepsilon - 1\|_{L^2(\Omega_\varepsilon)} =  \mathrm{O}(\mathcal{E}_d(\Keo))$.
    \end{enumerate}
\end{prop}
Let us emphasize that the bound stated in \cref{prop:L2 quasi-mode} 
is actually not sharp,
but 
it is sufficient for our purpose.
\begin{proof}
    Let us prove each point in order. \\
    \cref{prop:quasi-mode smooth} is clear from the definition of $\varphi_\varepsilon$ 
    and the regularity of $f_k$ from \cref{thm:fk}. \\
    \cref{prop:laplacian quasi-mode} is a consequence of \cref{def:fk}, 
    as
    \[
        -\Delta \varphi_\varepsilon = \sum_{k=1}^{N} \Kek (-\Delta f_k) =  \sum_{k=1}^{N} \Kek C_d.
    \]
    \cref{prop:neumann quasi-mode} is also a direct consequence of \cref{def:fk}, 
    as $\xk \not \in \Gamma_\mathcal{N}^\varepsilon$ for all $k \in \{1, \ldots, N\}$.\\
    \cref{prop:Linf quasi-mode} is the result stated in \cref{lem:quasi-mode}.\\
    Finally, \cref{prop:L2 quasi-mode} is a consequence of \cref{lem:L2 estimate fk}.
    Indeed, 
    from \eqref{eq:estimate fk_L2_lemma},
    one has that 
    \begin{equation}\label{eq:L2 quasi-mode estimate proof}
        \|\varphi_\varepsilon - 1\|_{L^2(\Omega_\varepsilon)} 
        = \left\| \sum_{k=1}^{N} \Kek f_k \right\|_{L^2(\Omega_\varepsilon)} 
        \leq \sum_{k=1}^{N} \Kek \|f_k\|_{L^2(\Omega_\varepsilon)} 
        = \left\{ \begin{aligned}
            & \mathrm{O}(\Keo)\, && \text{ for } d \leq 3,\\
            & \mathrm{O}\left(\sum_{k=1}^N \Kek \sqrt{\left|\log \Kek\right|}\right)\, && \text{ for } d = 4,\\
            & \mathrm{O}\left(\sum_{k=1}^N \left(\Kek\right)^\frac{d}{2(d-2)}\right)\, && \text{ for } d \geq 5.
        \end{aligned}
        \right.
    \end{equation}
    To conclude, it remains to show that the bounds in \eqref{eq:L2 quasi-mode estimate proof} 
    are of order $\mathcal{E}_d(\Keo)$, 
    which is clear for $d \leq 3$. 
    For $d \geq 5$, one has that $\frac{d}{2(d-2)} \in (\frac{1}{2}, \frac{5}{6}]$, 
    thus, by Jensen's inequality,
    \begin{equation}\label{eq:jensen bound d5}
        \sum_{k=1}^N \left(\Kek\right)^\frac{d}{2(d-2)} 
        = \mathrm{O}\left( \Keo^{\frac{d}{2(d-2)}}\right) 
        = \mathrm{O}\left( \Keo^{\frac{1}{d-2}}\right).
    \end{equation}
    For the case $d = 4$, using a similar argument as in \eqref{eq:log fake jensen},
    \begin{equation}\label{eq:jensen bound d4}
        \sum_{k=1}^N \Kek \sqrt{\left|\log \Kek\right|} 
        = \mathrm{O}\left( \Keo \sqrt{\left|\log \Keo\right|}\right) 
        = \mathrm{O}\left( \Keo^{\frac{1}{2}}\right),
    \end{equation}
    which allows us to conclude the proof.
\end{proof}

Now that we have built the quasi-mode $\varphi_\varepsilon$,
the last step before proving our main results is
to compute two scalar products that will come into play in the computations.
For the remainder of the paper, 
we will use $(\cdot , \, \cdot )_\mathrm{D}$ to denote the $L^2(\mathrm{D})$ scalar product 
on a domain $\mathrm{D}$.
\begin{lem}\label{prop: phi scalar u}
    There exists a constant $C > 0$ such that, for $\varepsilon$ sufficiently small,
    the quasi-mode $\varphi_\varepsilon$ (see~\eqref{eq:def phi}) and the quasi-stationary distribution $\nu_\varepsilon$ 
    (introduced in \cref{thm:QSD}) satisfy
    \begin{equation}\label{eq:scalar product u phi}
        \Bigl| \bigl( \varphi_\varepsilon , \nu_\varepsilon \bigr)_{\Omega_\varepsilon} - 1 \Bigr| \leq 
        C \mathcal{E}_d(\Keo),
    \end{equation}
    and
    $$ \Bigl| \left\langle \partial_n \nu_\varepsilon , \varphi_\varepsilon\right\rangle_{H^{-\frac{1}{2}}(\partial \Omega_\varepsilon), H^{\frac{1}{2}}(\partial \Omega_\varepsilon)} \Bigr|
    \leq C \lambda_\varepsilon^0 \,\mathcal{E}_d(\Keo).$$
\end{lem}
\begin{proof}
    Let us first derive an upper bound for the $L^2$ norm of $\nu_\varepsilon$.
    Since $\Omega_\varepsilon$ is uniformly bounded and satisfies a uniform interior cone condition, 
    the Poincaré--Wirtinger inequality holds in $\Omega_\varepsilon$ 
    for some constant $C>0$ independent of $\varepsilon$ (see \cite[Theorem 1.2]{Ruiz}).
    Applying this inequality to $\nu_\varepsilon$ gives
    \begin{align}\label{pw}
        \left\|\nu_\varepsilon-\frac{1}{|\Omega_\varepsilon|}\left( \nu_\varepsilon  , 1\right)\right\|_{L^2(\Omega_\varepsilon)}^2
        \leq C \|\nabla \nu_\varepsilon\|_{L^2(\Omega_\varepsilon)}^2 = C \lambda_\varepsilon^0 \| \nu_\varepsilon\|_{L^2(\Omega_\varepsilon)}^2,
    \end{align}
    where the last equality is a consequence of the fact that $\nu_\varepsilon$ is an eigenfunction of the Laplacian with eigenvalue $\lambda_\varepsilon^0$.
    Since $\left\lVert \nu_\varepsilon \right\rVert_{L^1(\Omega_\varepsilon)} = 1$, the left-hand side satisfies
    \begin{align}\label{lhs}
        \left\|\nu_\varepsilon-\frac{1}{|\Omega_\varepsilon|}\left( \nu_\varepsilon  , 1\right)\right\|_{L^2(\Omega_\varepsilon)}^2 = \|\nu_\varepsilon\|_{L^2(\Omega_\varepsilon)}^2-\frac{1}{|\Omega_\varepsilon|} .
    \end{align}
    Combining \eqref{pw} and 
    \eqref{lhs},
    we obtain
    \begin{align}\label{mino}
        \|\nu_\varepsilon\|_{L^2(\Omega_\varepsilon)}^2 (1 - C \lambda_\varepsilon^0) 
        \leq  |\Omega_\varepsilon|^{-1}.
    \end{align}
    Since 
    \[
        |\Omega| = |\Omega_\varepsilon| + \mathrm{O}\Biggl( \sum_{k=1}^N \bigl(\re{(k)}\bigr)^d\Biggr),
    \]
    and since $1 - C \lambda_\varepsilon^0 \geq \frac{1}{2}$ for $\varepsilon$ sufficiently small given that
    $\lambda_\varepsilon^0 \to 0$ when $\varepsilon \to 0$ by~\cref{lem:eigenvalues},
    the inequality \eqref{mino} leads to 
    \begin{equation} \label{eq:bound u_0}
        \|\nu_\varepsilon\|_{L^2(\Omega_\varepsilon)} \leq C',
    \end{equation}
    for some constant $C'>0$, independent of $\varepsilon$.

    We can then prove \eqref{eq:scalar product u phi} using the Cauchy--Schwarz inequality, 
    together with the bound on $\nu_\varepsilon$ obtained in \eqref{eq:bound u_0}
    and \cref{prop:L2 quasi-mode} from \cref{prop:properties quasi-mode}:
    $$ 
        \left| \bigl( \varphi_\varepsilon , \nu_\varepsilon \bigr)_{\Omega_\varepsilon} - 1\right|
        = \left| \bigl( \varphi_\varepsilon - 1, \nu_\varepsilon \bigr)_{\Omega_\varepsilon}\right|
        \leq \left\| \varphi_\varepsilon  - 1\right\|_{L^2(\Omega_\varepsilon)}\left\| \nu_\varepsilon \right\|_{L^2(\Omega_\varepsilon)}
        \leq C \mathcal{E}_d(\Keo).
    $$

    For the second statement, 
    let us first notice that $\bigl\langle \partial_n \nu_\varepsilon , \varphi_\varepsilon\bigr\rangle_{H^{-\frac{1}{2}}(\partial \Omega_\varepsilon), H^{\frac{1}{2}}(\partial \Omega_\varepsilon)} $ is well-defined by \cref{prop:quasi-mode smooth} of \cref{prop:properties quasi-mode} and~\cref{thm:QSD}. 
    Furthermore, since $\partial_n \nu_\varepsilon$ is a Radon measure on $\partial \Omega_\varepsilon$ (see \cref{cor:QSD})
    \[
        \bigl\langle \partial_n \nu_\varepsilon , \varphi_\varepsilon\bigr\rangle_{H^{-\frac{1}{2}}(\partial \Omega_\varepsilon), H^{\frac{1}{2}}(\partial \Omega_\varepsilon)} 
        = \int_{\partial \Omega_\varepsilon} \varphi_\varepsilon \, \d \left( \partial_n \nu_\varepsilon\right)
        = \int_{\Gamma_\mathcal{D}^\varepsilon} \varphi_\varepsilon \, \d \left( \partial_n \nu_\varepsilon\right),
    \]
    where we used that $\partial_n \nu_\varepsilon$ is supported on $\overline{\Gamma_\mathcal{D}^\varepsilon}$.
     Then,
    $$
    \Bigl| \bigl\langle \partial_n \nu_\varepsilon , \varphi_\varepsilon\bigr\rangle_{H^{-\frac{1}{2}}(\partial \Omega_\varepsilon), H^{\frac{1}{2}}(\partial \Omega_\varepsilon)} \Bigr|
    \leq \|\varphi_\varepsilon\|_{L^\infty(\Gamma_\mathcal{D}^\varepsilon)} \left| \partial_n \nu_\varepsilon (\Gamma_\mathcal{D}^\varepsilon)\right|.
    $$
    The norm $\|\varphi_\varepsilon\|_{L^\infty(\Gamma_\mathcal{D}^\varepsilon)}$ 
    is controlled by \cref{prop:Linf quasi-mode} of \cref{prop:properties quasi-mode}
    and, from \cref{cor:QSD}, $\left| \partial_n \nu_\varepsilon (\Gamma_\mathcal{D}^\varepsilon)\right| = \lambda_\varepsilon^0$. 
    This concludes the proof.
\end{proof}

\subsection{Mean exit time}
\label{sec:mean exit}
    We can now proceed with the proofs of the main results. 
    Let us start with the estimate on the first eigenvalue stated in \cref{thm:lam0}

    \begin{proof}[Proof of \cref{thm:lam0}]
        Notice that the first point of \cref{prop: phi scalar u} ensures that $\bigl( \varphi_\varepsilon, \, \nu_\varepsilon \bigr)_{\Omega_\varepsilon} \neq 0$ 
        for $\varepsilon$ small enough.
        Thus, 
        \begin{align}\label{lam01}
            \lambda_\varepsilon^0
            = \frac{\lambda_\varepsilon^0 \bigl( \varphi_\varepsilon, \nu_\varepsilon \bigr)_{\Omega_\varepsilon}}{ \bigl( \varphi_\varepsilon, \nu_\varepsilon \bigr)_{\Omega_\varepsilon}}
            = \frac{ \bigl( \varphi_\varepsilon, -\Delta \nu_\varepsilon \bigr)_{\Omega_\varepsilon}}{ \bigl( \varphi_\varepsilon, \nu_\varepsilon \bigr)_{\Omega_\varepsilon}}.
        \end{align}
        Now, by the Green formula, allowed by \cref{prop:quasi-mode smooth} of \cref{prop:properties quasi-mode},
        the fact $\nu_\varepsilon \in \mathcal{D}(\mathcal{L}_\varepsilon)$ and \eqref{eq:normal_derivative_weak_h12}, we have
        \begin{align*}
            \bigl( \varphi_\varepsilon, -\Delta \nu_\varepsilon \bigr)_{\Omega_\varepsilon} 
            &= \bigl( -\Delta \varphi_\varepsilon, \nu_\varepsilon \bigr)_{\Omega_\varepsilon} - \left\langle \partial_n \nu_\varepsilon , \varphi_\varepsilon\right\rangle_{H^{-\frac{1}{2}}(\partial \Omega_\varepsilon), H^{\frac{1}{2}}(\partial \Omega_\varepsilon)}
             + \bigl( \partial_n \varphi_\varepsilon,  \nu_\varepsilon \bigr)_{\partial \Omega_\varepsilon}  \\
            &=C_d \Keo \bigl( 1, \nu_\varepsilon \bigr)_{\Omega_\varepsilon} 
            - \left\langle \partial_n \nu_\varepsilon , \varphi_\varepsilon\right\rangle_{H^{-\frac{1}{2}}(\partial \Omega_\varepsilon), H^{\frac{1}{2}}(\partial \Omega_\varepsilon)} + 0 \, ,
        \end{align*}
        where we used \cref{prop:laplacian quasi-mode} of \cref{prop:properties quasi-mode} in the second line,
        together with the fact that
        the last term is null since $\partial_n \varphi_\varepsilon$ is $0$ on $\Gamma_\mathcal{N}^\varepsilon$
        while~$\nu_\varepsilon$ is $0$ on $\Gamma_\mathcal{D}^\varepsilon$. 
        This shows that \eqref{lam01} can be rewritten as
        \begin{equation} \label{lam02}
        \lambda_\varepsilon^0
        = \frac{C_d \Keo}{\bigl( \varphi_\varepsilon, \nu_\varepsilon \bigr)_{\Omega_\varepsilon}} - \frac{\left\langle \partial_n \nu_\varepsilon , \varphi_\varepsilon\right\rangle_{H^{-\frac{1}{2}}(\partial \Omega_\varepsilon), H^{\frac{1}{2}}(\partial \Omega_\varepsilon)}}{\bigl( \varphi_\varepsilon, \nu_\varepsilon \bigr)_{\Omega_\varepsilon}}.
        \end{equation}
        Using both results of \cref{prop: phi scalar u},
        one gets that, for $\varepsilon$ sufficiently small,
        \begin{align*}
            \lambda_\varepsilon^0 
            & = \frac{C_d \Keo + \lambda_\varepsilon^0 \, \mathrm{O}(\mathcal{E}_d(\Keo))}{1 + \mathrm{O}(\mathcal{E}_d(\Keo))}.
        \end{align*}
        Finally, 
        isolating $\lambda_\varepsilon^0$ in the previous equation leads to
        \begin{equation*}
            \lambda_\varepsilon^0 = C_d \Keo \left( 1 + \mathrm{O}(\mathcal{E}_d(\Keo))\right)^{-1} = C_d \Keo + \mathrm{O}(\Keo \mathcal{E}_d(\Keo)),
        \end{equation*}
        which is the claimed result.
    \end{proof}

\subsection{Exit hole distribution}
\label{sec:exit hole}
Let us now move to the proof of \cref{thm:exit hole distribution}.
The general strategy is similar to the one used in the proof of \cref{thm:lam0}.
However,
we use a slightly different quasi-mode,
corresponding to the setting where one hole has been removed. 
By comparing this modified quasi-mode with the quasi-stationary distribution $\nu_\varepsilon$, 
one can extract the probability to exit through the removed hole.
Naturally, in this subsection, one assumes $N \geq 2$.

\begin{proof}[Proof of \cref{thm:exit hole distribution}]
From \cref{cor:QSD},
one has that, for any $k_0 \in \{ 1, \ldots, N \}$,
$$\mathbb{P}_{\nu_\varepsilon} \Big( X_{\tau_\varepsilon} \in \Gamma_{k_0}^\varepsilon \Big) 
= - \frac{ 1}{\lambda_\varepsilon^0} \int_{\Gamma_{k_0}^\varepsilon} \mathrm{d}\bigl( \partial_n \nu_\varepsilon \bigr).$$
From the expression of $\lambda_\varepsilon^0$ from \cref{thm:lam0},
one only needs to show that
\begin{align}\label{amq}
    \int_{\Gamma_{k_0}^\varepsilon} \mathrm{d}\bigl( \partial_n \nu_\varepsilon \bigr) = \int_{\Gamma_{\mathcal{D}}^\varepsilon} \mathbbm{1}_{\Gamma_{k_0}^\varepsilon} \mathrm{d}\bigl( \partial_n \nu_\varepsilon \bigr) = - C_d K_\varepsilon^{(k_0)} + \mathrm{O}\left( \Keo \mathcal{E}_d(\Keo)\right).
\end{align}
Consider now the quasi-mode $\hat{\varphi}_{-k_0}^\varepsilon$ defined as
\begin{align}\label{hatphi}
 \hat{\varphi}_{-k_0}^\varepsilon (x) = 1+ \sum_{k\in \llbracket 1,N \rrbracket \setminus \{k_0\}} \Kek
 f_k(x).
\end{align}
This is the quasi-mode associated with the narrow escape problem where the hole centered at $x^{(k_0)}$ has been removed.
Naturally, \cref{prop:properties quasi-mode} still holds true for this quasi-mode.
In particular, 
\begin{equation}\label{eq:hatphidirichlet1}
    \|\hat{\varphi}_{-k_0}^\varepsilon \|_{L^\infty(\Gamma_{\mathcal{D}}^\varepsilon \setminus \Gamma_{k_0}^\varepsilon)}= 
    \mathrm{O}{\left(\mathcal{E}_d\bigl(\overline{K}_\varepsilon^{(-k_0)}\bigr)\right)}, 
\end{equation}
where $\overline{K}_\varepsilon^{(-k_0)} = \sum_{k\in \llbracket 1,N \rrbracket \setminus \{k_0\}} \Kek$.
Moreover, 
the quasi-mode $\hat{\varphi}_{-k_0}^\varepsilon$ is close to $1$ on $\Gamma_{k_0}^\varepsilon$,
as, 
\begin{equation}\label{eq:hatphidirichlet2}
    \| \hat{\varphi}_{-k_0}^\varepsilon - 1\|_{L^\infty(\Gamma_{k_0}^\varepsilon)}
    \leq \sum_{k\in \llbracket 1,N \rrbracket \setminus \{k_0\}} \Kek \|f_k \|_{L^\infty(\Gamma_{k_0}^\varepsilon)}
    = \mathrm{O}{\left( \overline{K}_\varepsilon^{(-k_0)}\right)} = \mathrm{O}{\left( \mathcal{E}_d(\overline{K}_\varepsilon^{(-k_0)}) \right)},
\end{equation}
in the limit $\varepsilon \to 0$.
Here, we used that $f_k$ is smooth on $\overline{\Omega} \setminus \{x^{(k)}\}$.

The idea of the proof is to decompose the integral in \eqref{amq} into two parts,
\begin{equation}\label{eq:decomposition integral}
\begin{aligned}
    \int_{\Gamma_{\mathcal{D}}^\varepsilon} \mathbbm{1}_{\Gamma_{k_0}^\varepsilon} \mathrm{d}\bigl( \partial_n \nu_\varepsilon \bigr) 
    = \int_{\Gamma_{\mathcal{D}}^\varepsilon} \hat{\varphi}_{-k_0}^\varepsilon \mathrm{d}\bigl( \partial_n \nu_\varepsilon \bigr) 
    + \int_{\Gamma_{\mathcal{D}}^\varepsilon} \left(\mathbbm{1}_{\Gamma_{k_0}^\varepsilon} - \hat{\varphi}_{-k_0}^\varepsilon\right)
    \mathrm{d}\bigl( \partial_n \nu_\varepsilon \bigr).
\end{aligned}
\end{equation}
In view of \eqref{eq:hatphidirichlet1} and \eqref{eq:hatphidirichlet2},
the second integral in \eqref{eq:decomposition integral} can be bounded as follows:
\begin{equation}\label{phi+rho}
    \begin{aligned}
    \Bigl| \int_{\Gamma_{\mathcal{D}}^\varepsilon} 
    \left(\mathbbm{1}_{\Gamma_{k_0}^\varepsilon} - \hat{\varphi}_{-k_0}^\varepsilon\right)
    \mathrm{d}\bigl( \partial_n \nu_\varepsilon \bigr)  \Bigr|
    & \leq \left| \partial_n \nu_\varepsilon(\Gamma_{\mathcal{D}}^\varepsilon) \right|
    \big\| \mathbbm{1}_{\Gamma^{\varepsilon}_{k_0}} - \hat{\varphi}_{-k_0}^\varepsilon \big\|_{L^\infty(\Gamma_{\mathcal{D}}^\varepsilon)}\\ 
    &  = \left| \partial_n \nu_\varepsilon(\Gamma_{\mathcal{D}}^\varepsilon) \right|\, \mathrm{O}{\Bigl( \mathcal{E}_d\bigl(\overline{K}_\varepsilon^{(-k_0)}\bigr) \Bigr)}\\
    &  = \lambda_\varepsilon^0 \, \mathrm{O}{\left( \mathcal{E}_d\bigl(\overline{K}_\varepsilon^{(-k_0)}\bigr) \right)}.
    \end{aligned}
\end{equation}
We used in the last equality the value of $\partial_n \nu_\varepsilon(\Gamma_{\mathcal{D}}^\varepsilon)$ 
obtained from \cref{cor:QSD}.
It remains to estimate the duality product between $\partial_n \nu_\varepsilon$ and $\hat{\varphi}_{-k_0}^\varepsilon$.
Using similar techniques as in the proof of \cref{thm:lam0}:
 \begin{equation}\label{eq:computation duality product}
        \begin{aligned}
            \int_{\Gamma_{\mathcal{D}}^\varepsilon} \hat{\varphi}_{-k_0}^\varepsilon \mathrm{d}\bigl( \partial_n \nu_\varepsilon \bigr) 
            & = \bigl( \Delta \nu_\varepsilon , \hat{\varphi}_{-k_0}^\varepsilon \bigr)_{\Omega_\varepsilon}
            - \bigl(  \nu_\varepsilon ,  \Delta \hat{\varphi}_{-k_0}^\varepsilon \bigr)_{\Omega_\varepsilon}   \\
            &= - \lambda_\varepsilon^0 
            \bigl(  \nu_\varepsilon ,  \hat{\varphi}_{-k_0}^\varepsilon \bigr)_{\Omega_\varepsilon}  
            +C_d \overline{K}_\varepsilon^{(-k_0)} \bigl( \nu_\varepsilon, 1\bigr)_{\Omega_\varepsilon} \\
            &= - C_d K_\varepsilon^{(k_0)} - \lambda_\varepsilon^0 \bigl(  \nu_\varepsilon ,  \hat{\varphi}_{-k_0}^\varepsilon - 1 \bigr)_{\Omega_\varepsilon} + \left( C_d \Keo - \lambda_\varepsilon^0 \right). 
            \\
        \end{aligned}
    \end{equation}
    The first line relies on the fact that the support of $\partial_n \nu_\varepsilon$ is included in $\overline{\Gamma_\mathcal{D}^\varepsilon}$,
    while the second line is obtained by using the eigenvalue equation \eqref{eq:eigen} for $\nu_\varepsilon$ and for $\hat{\varphi}_{-k_0}^\varepsilon$ (\cref{prop:laplacian quasi-mode} of \cref{prop:properties quasi-mode}). 
    In a similar fashion as in the proof of \cref{prop: phi scalar u}, 
    the scalar product $\bigl(  \nu_\varepsilon ,  \hat{\varphi}_{-k_0}^\varepsilon \bigr)_{\Omega_\varepsilon}$ can be estimated by using the Cauchy--Schwarz inequality.
    Indeed, there exists a constant $C>0$ such that
    \[
        \left|\bigl(  \nu_\varepsilon ,  \hat{\varphi}_{-k_0}^\varepsilon - 1\bigr)_{\Omega_\varepsilon}  \right|
        \leq \left\| \nu_\varepsilon \right\|_{L^2(\Omega_\varepsilon)} \left\| \hat{\varphi}_{-k_0}^\varepsilon - 1 \right\|_{L^2(\Omega_\varepsilon)}
        \leq C \mathcal{E}_d(\overline{K}_\varepsilon^{(-k_0)}),
    \]
    where we used the bound on $\nu_\varepsilon$ from \eqref{eq:bound u_0} and the $L^2$ bound on $\hat{\varphi}_{-k_0}^\varepsilon - 1$ from \cref{prop:L2 quasi-mode} of \cref{prop:properties quasi-mode}.
    Plugging this estimate and the estimate of $\lambda_\varepsilon^0$ from \cref{thm:lam0} into \eqref{eq:computation duality product},
    in addition to \eqref{phi+rho} and \eqref{eq:decomposition integral},
    leads to the desired result~\eqref{amq}.
\end{proof}

\section{Numerical results}
\label{sec:numerics}
The goal of this section is to illustrate the main results of this work, 
namely \cref{thm:lam0,thm:exit hole distribution}.
We first start by discussing the numerical methods we use in \cref{sec:numerical methods}.
Then, we present the numerical results on the influence of the initial distribution in \cref{sec:initial distribution}, on the influence of the dimension on the mean exit time in \cref{sec:mean exit time},
on the influence of the shape of the domain in \cref{sec:various domains} and finally an illustration of \cref{rq:exit hole distribution} 
on the exit hole distribution in \cref{sec:exit hole dist}.

\subsection{Numerical methods} 
\label{sec:numerical methods}
Traditionally, 
eigenvalue problems have been investigated numerically using the Finite Element Method (FEM)~\cite{Tony2024}.
In \cref{sec:various domains}, 
we use the implementation of the FEM provided by the Gridap library~\cite{Badia2020} 
to compute the two smallest eigenvalues $\lambda_\varepsilon^0$ and $\lambda_\varepsilon^1$ 
and the associated eigenfunctions~$u_\varepsilon^0$ and~$u_\varepsilon^1$.
The domain is meshed adaptively near the holes using the Gmsh library~\cite{Geuzaine}.

\paragraph{Monte Carlo methods.}
FEM implementations are typically restricted to dimensions $d \leq 3$,
whereas our results hold for larger dimensions.
Monte Carlo methods are a natural alternative to FEM in higher dimensions, 
and have already been applied to the narrow escape problem in \cite{Caginalp} (see \cite{Burdzy} for a more general reference on reflected Brownian motion).
However, a naive implementation is not efficient enough to reach the asymptotic regime of the narrow escape problem.
To illustrate this, consider the Euler--Maruyama scheme with a time step $h$ and a number of steps $N_\mathrm{step}$.
The reflected Brownian motion is simulated by rejecting all proposed steps that would lead the particle outside of the domain $\Omega_\varepsilon$ by $\Gamma_{\mathcal{N}}^\varepsilon$.
A rough estimate for the average number of steps required to simulate one trajectory until exit is, for~$d \geq 3$,
\[
    N_\mathrm{step} = \frac{\left(\lambda_\varepsilon^0\right)^{-1}}{h} \propto r_\varepsilon^{-d},
\]
using $h \propto r_\varepsilon^2$ as a reasonable upper bound for the time step, 
so the sizes of the holes are comparable to the average step size.
Experimentally, we found that $r_\varepsilon = 10^{-2}$ was the maximum hole size we could consider
to have a chance to observe the asymptotic regime for both the exit time and the exit hole distribution in dimension greater than $3$.
But already in dimension $d=3$, this leads to Monte Carlo simulations requiring billions of steps,
rendering them too slow to compute statistics.

In \cref{rema:Yaglom}, 
we discussed that the process $(X_t)_{t \geq 0}$ conditioned on the event $\{\tau_\varepsilon > t\}$ converges in law to the quasi-stationary distribution $\nu_\varepsilon$ as $t \to \infty$.
In practice,
after a short mixing time,
simulated particles will be distributed according to $\nu_\varepsilon$, which is almost uniform in the domain $\Omega_\varepsilon$.
Hence,
most of the computation time is spent far away from the holes, in the bulk of the domain,
where a very small time step is not required.
The algorithm we propose relies on the previous observations and on the Walk on Spheres (WoS) method \cite[Definition 2.2]{Muller1956}.
A detailed description in pseudocode of the algorithm we used can be found in \cref{alg:wos}.
We next discuss more precisely the various parts of the numerical method. 

\begin{algorithm}
    \caption{Sampling mean exit time and exit hole with Walk on Spheres (WoS) and Euler--Maruyama}
    \label{alg:wos}
    \begin{algorithmic}[1]
    \Require Domain $\Omega_\varepsilon$, time step $h$, initial position $x_0$
    
    \State Initialize $x \gets x_0$, $t \gets 0$
    \While{$x \in \Omega_\varepsilon$}
        \State Compute distance $\mathrm{d}_\text{boundary}$ to the boundary of $\Omega_\varepsilon$
        \If{$\mathrm{d}_\text{boundary} > 20\sqrt{h}$ } \Comment{WoS step}
            \State Sample a random direction $u \sim \partial B(0,1)$
            \State $x \gets x + \mathrm{d}_\text{boundary} \, u$
            \State $\Delta t \gets \mathrm{d}_\text{boundary}^{\, 2}/ (2d)$
        \Else \Comment{Euler--Maruyama step}
            \State Sample a random vector $Z \sim \mathcal{N}(0,\I_d)$
            \If {$x + \sqrt{2h} \, Z$ is inside $\overline{\Omega}$}
                \State $x \gets x + \sqrt{2h} \, Z$
            \EndIf
            \State $\Delta t \gets h$
        \EndIf
        \State $t \gets t + \Delta t$
    \EndWhile
    \State \Return $(t, \, \underset{k \in \{1, \ldots, \, N \}}{\mathrm{argmin}}|x - x^{(k)}|)$
    \end{algorithmic}
\end{algorithm}

\paragraph{The Walk on Spheres method.} 
\cref{alg:wos} uses the WoS method when the particle is 
far from the holes and the Neumann reflective parts of the boundary,
in order to accelerate the simulation.
When the particle is close to $\partial \Omega_\varepsilon$, 
we switch to the Euler--Maruyama method with a small time step.

To be more precise, 
at each step, 
we compute the distance $\mathrm{d}_\text{boundary}$ to the closest boundary (Neumann or Dirichlet).
If this distance is larger than $20 \sqrt{h}$, a large distance compared to the typical Euler--Maruyama step size,
we perform a \textit{large} WoS step of size~$\mathrm{d}_\text{boundary}$. 
The idea for WoS is the following:
we are sampling the exit event of a Brownian motion from a ball of radius $\mathrm{d}_\text{boundary}$,
starting from its center. 
Both the exit point and the exit time are random variables,
but they are independent and their distribution is known explicitly (see \cite[Section 3]{Muller1956}).
Hence, 
starting from a point $x_n$ at time $t_n$, 
we update the position and time as follows:

\begin{equation}\label{eq:wos step}\begin{aligned}
    x_{n+1} & =  x_n + \mathrm{d}_\text{boundary} \, u \quad \text{with} \quad u \sim \partial B(0,1), \\
    t_{n+1} & =  t_n + \frac{\mathrm{d}_\text{boundary}^2}{2d}.
\end{aligned}\end{equation}
Notice that in the second equation in \eqref{eq:wos step}, 
we increment the elapsed time by the mean exit time of the Brownian motion from the ball of radius $\mathrm{d}_\text{boundary}$.
We could also sample it from its distribution. 
However, by the linearity of the expectation, 
it is sufficient to add the mean exit times to obtain 
an estimate of the mean exit time $\mathbb{E}[\tau_\varepsilon]$.

Otherwise, 
if the process is too close to the boundary or a hole,
we perform a standard Euler--Maruyama step with time step $h$.
If an Euler--Maruyama step moves the particle outside the domain~$\Omega$, the step is rejected and $h$ is added to the time.

This method provides a way to implement an adaptive time step method for reflected Brownian motion.
For alternative options and error estimates see \cite{Hoel2023}.
Experimentally, we found the algorithm allows us to greatly reduce the number of steps required to simulate exit events compared to a naive Euler--Maruyama method,
by a factor of order $1000$ (see \cref{fig:wos vs em} for an illustration).

We used this approach to produce all the non-FEM numerical results of this section. 
Unless stated otherwise, all the results produced with \cref{alg:wos} have been obtained
with $h = 10^{-11}$ for the time step, averages over at least $10^3$ realizations and
a Dirac mass at the center of the domain as initial condition. 
The implementation of the algorithm has been done in Rust \cite{rust}, 
using the \texttt{nalgebra} crate \cite{nalgebra} and the parallelism provided by the \texttt{Rayon} crate \cite{rayon}.

\begin{figure}
    \centering
    \includegraphics[width=0.8\textwidth]{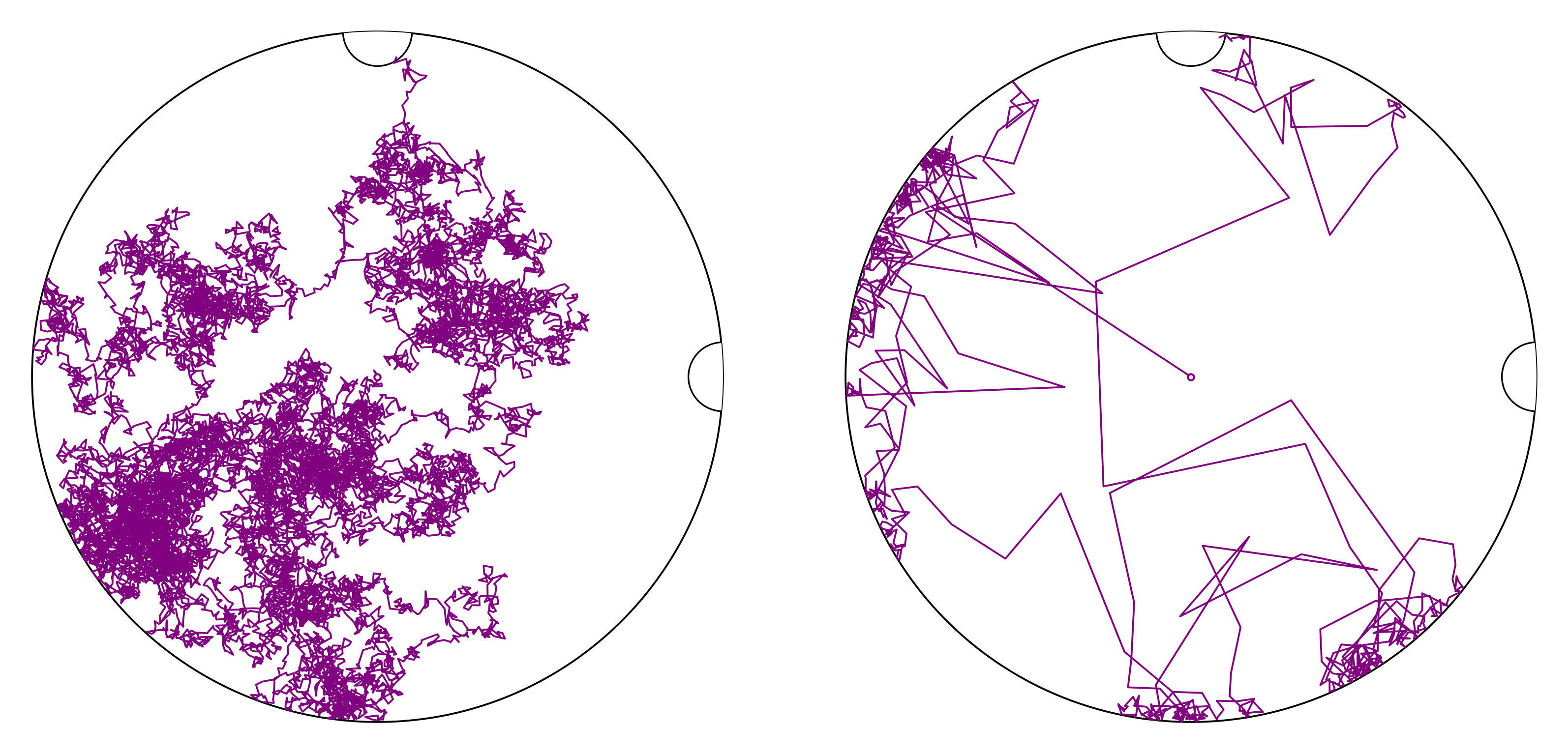}
    \caption{Trajectories from a naive Euler--Maruyama algorithm (left) and \cref{alg:wos} (right),
     for a two-dimensional disk with two holes of radius $r_\varepsilon = 10^{-1}$. 
     As an illustration, the time step $h$ is set to $10^{-4}$.}
    \label{fig:wos vs em}
\end{figure}

\subsection{Influence of the initial distribution}
\label{sec:initial distribution}
In this section, we illustrate the influence of the initial distribution on the exit time distribution. 
Indeed, both \cref{thm:lam0,thm:exit hole distribution} are stated for particles starting from the quasi-stationary distribution $\nu_\varepsilon$,
whereas, in practice, the initial distribution is a Dirac mass at a given point $x_0 \in \Omega_\varepsilon$ in \cref{alg:wos}.
As discussed in \cref{rema:Yaglom},
one nevertheless expects the law of the process to quickly relax to the quasi-stationary distribution $\nu_\varepsilon$ before exiting.

To illustrate this, we consider the unit disk in dimension $2$ with two holes of identical radius $r_\varepsilon$. 
We compare $\lambda_\varepsilon^0$, first computed through finite element methods, with the results obtained with \cref{alg:wos} for three different initial positions $x_0$:
the center of the domain, one position closer to the holes, and one farther.
\Cref{fig:initial distribution} shows that, when the particles are started from the center of the domain,
the results obtained with \cref{alg:wos} are in very good agreement with the eigenvalue computed with FEM,
despite the difference between the initial distributions.
For the two other initial positions, the curves are shifted upwards (start close to the holes)
or downwards (start far from the holes), 
but the behavior predicted by \cref{thm:lam0} is still recovered.
This is consistent with~\cite[Theorem 3.1]{AMMARI201266}, 
which states that the influence of the initial distribution on the mean exit time is an additive constant independent of $\varepsilon$.
Since $\lambda_\varepsilon^0$ is the reciprocal of the mean exit time, 
this additive constant only contributes to $\lambda_\varepsilon^0$ at a subleading order and vanishes in the limit~$\varepsilon \to 0$.
In the following, 
we always start the particles from the center of the domain.

\begin{figure}
    \centering
    \includegraphics[width=0.7\textwidth]{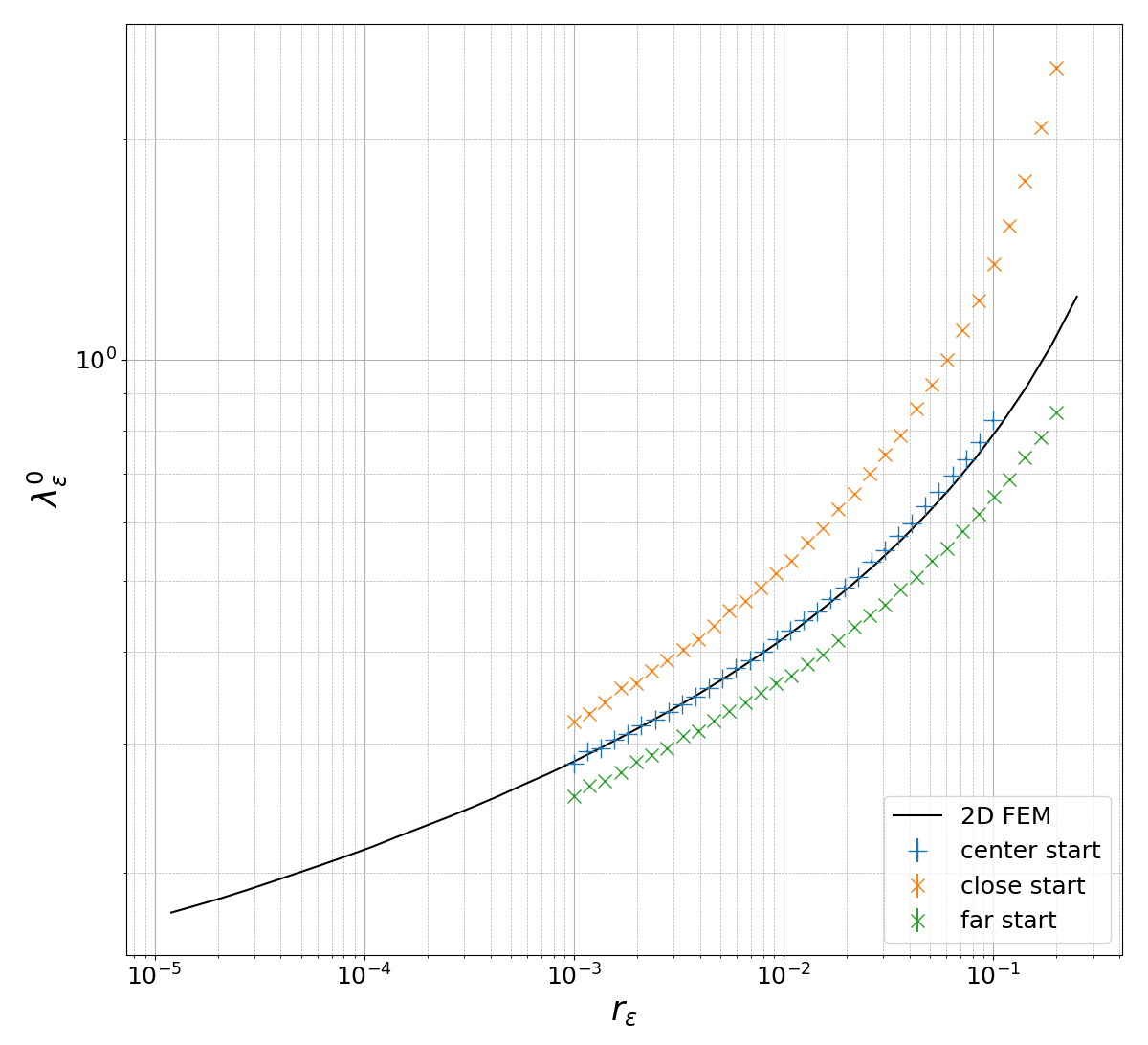}
    \caption{First eigenvalue $\lambda_\varepsilon^0$ as a function of $r_\varepsilon$ for the unit disk with two holes of identical radius $r_\varepsilon$. 
    The black line is the result of FEM simulations, which fits exactly the estimate given by \cref{thm:lam0}.
    The results obtained with \cref{alg:wos} are shown for three initial positions $x_0$: 
    the center of the domain, a position closer to the holes and a position farther from the holes.
    Starting from the center reproduces the FEM eigenvalue accurately.
    The two other initial conditions are shifted, but only through an additive constant on the mean exit time
    $1/\lambda_\varepsilon^0$, so that the relative deviation vanishes as $r_\varepsilon \to 0$.}
    \label{fig:initial distribution}
\end{figure}

\subsection{Mean exit time}
\label{sec:mean exit time}
We illustrate \cref{thm:lam0} in \cref{fig:mean exit time} in dimensions ranging from $d=2$ to $d=6$ using \cref{alg:wos}.
The chosen domain is the unit hyperball, 
with $N=d$ holes of equal radius $r_\varepsilon=\varepsilon$ placed in each of the cardinal directions $e_i = (0,\ldots,1,\ldots,0)^\top$.
The expected values for $C_d$ are respectively given by~$1, \,1.5, \,4, \,7.5, \,12$, from the smallest dimension $d=2$ to the largest dimension $d=6$.

\begin{figure}
    \centering
    \includegraphics[width=0.7\textwidth]{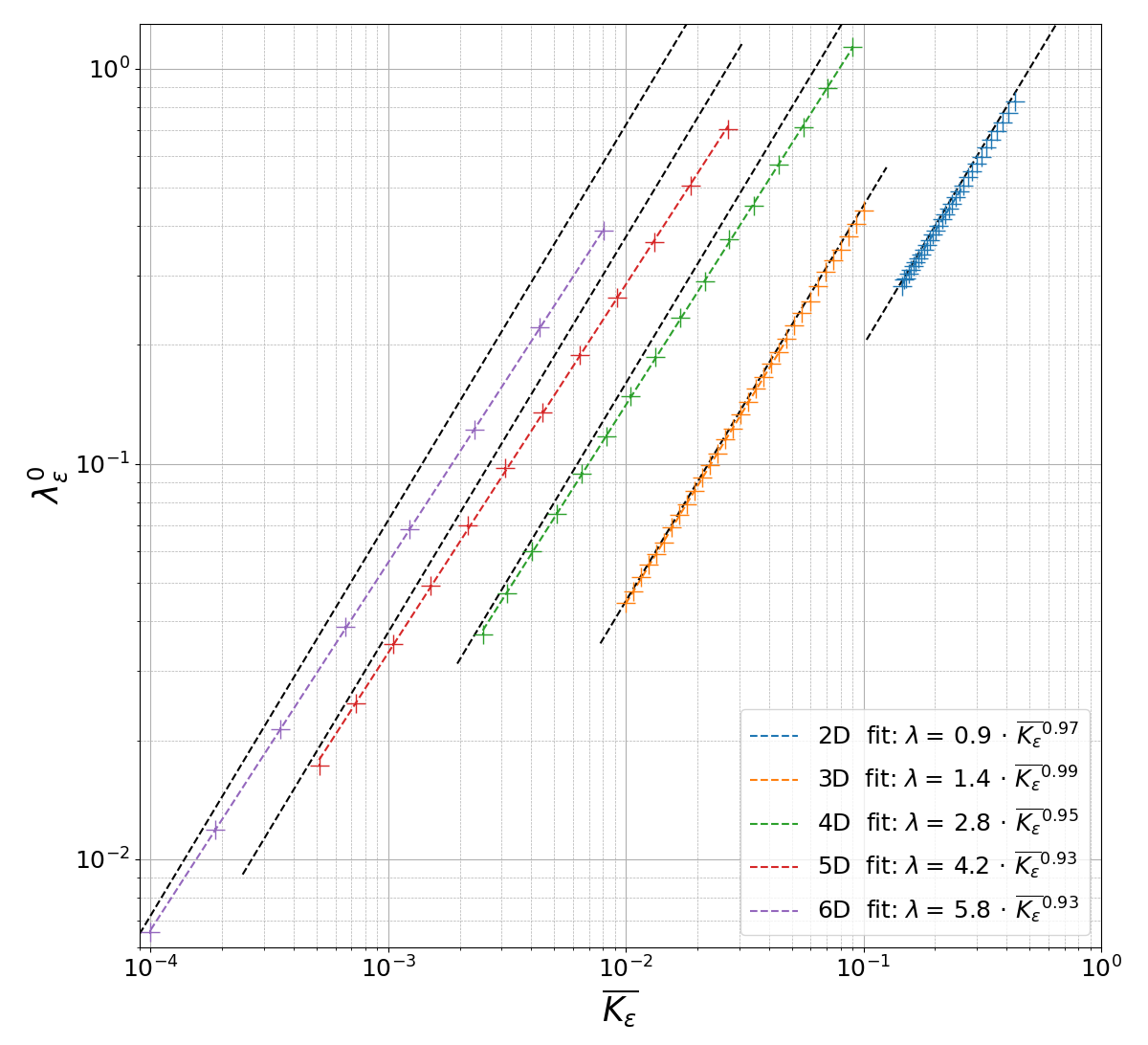}
    \caption{Estimation of the eigenvalue $\lambda_\varepsilon^0 = \left(\mathbb{E} [\tau_\varepsilon]\right)^{-1}$ as a function of $K_\varepsilon$ in dimensions $2$ to $6$.
        The black dashed lines are the theoretical estimates from \cref{thm:lam0}, 
    which predict $\lambda_\varepsilon^0$ to be asymptotically linear in $\Keo$
    with a known prefactor $C_d$.}
    \label{fig:mean exit time}
\end{figure}

In \cref{fig:mean exit time},
a fit $\lambda_\varepsilon^0 \sim C_d^\text{fit} \Keo^{\alpha^\text{fit}}$ is performed for each dimension.
In every dimension the fitted exponent~$\alpha^\text{fit}$ is very close to $1$,
which is in good agreement with the theoretical prediction of \cref{thm:lam0}.
For the constant $C_d^\text{fit}$, while the results for dimensions $2$ and $3$ are in good agreement with the theory, 
this is not the case in higher dimensions. 
This can be understood from the fact that the relative error term scales as $r_\varepsilon$ and not $\Keo$,
as explained in \cref{rema:mean exit time}, 
hence the asymptotic regime is not yet observed. 
For instance, the value $K_\varepsilon = 10^{-4}$ for the last point in dimension $6$ 
corresponds to only $r_\varepsilon = 10^{-1}$, 
which is presumably not small enough to observe the asymptotic regime. 

\subsection{Various domains in dimension 2}
\label{sec:various domains}
The results of \cref{thm:lam0} hold for any smooth domain.
We illustrate this in dimension $2$ for the domains shown in \cref{fig:different shapes draw}.
Notice that four of them are not smooth.
\begin{figure}
    \centering
    \includegraphics[width=0.7\textwidth]{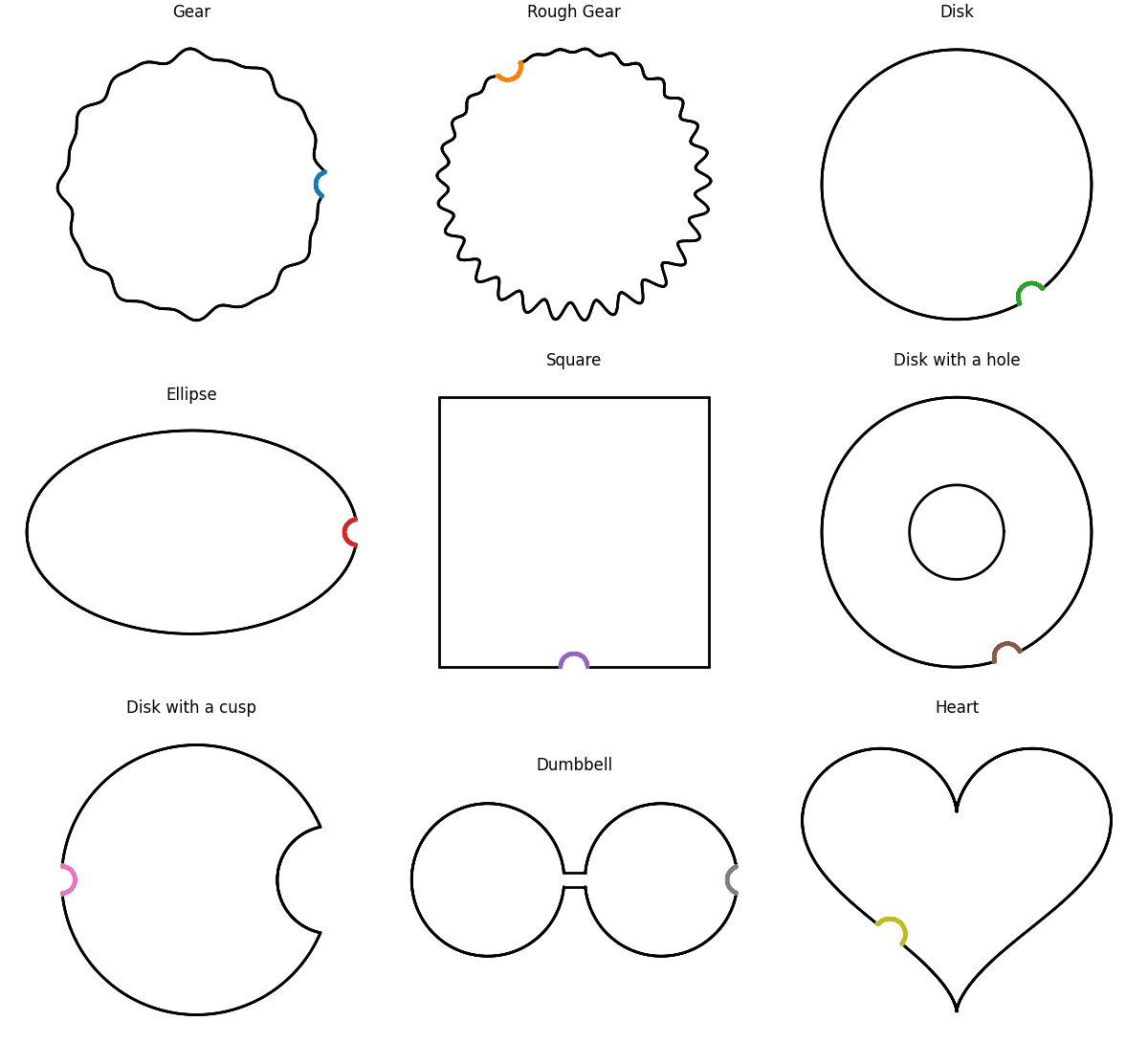}
    \caption{Various domains used to illustrate \cref{thm:lam0} in dimension $2$. 
    The red area represents the holes, of radius $0.1$ here.}
    \label{fig:different shapes draw}
\end{figure}

\begin{figure}
    \centering
    \includegraphics[width=0.9\textwidth]{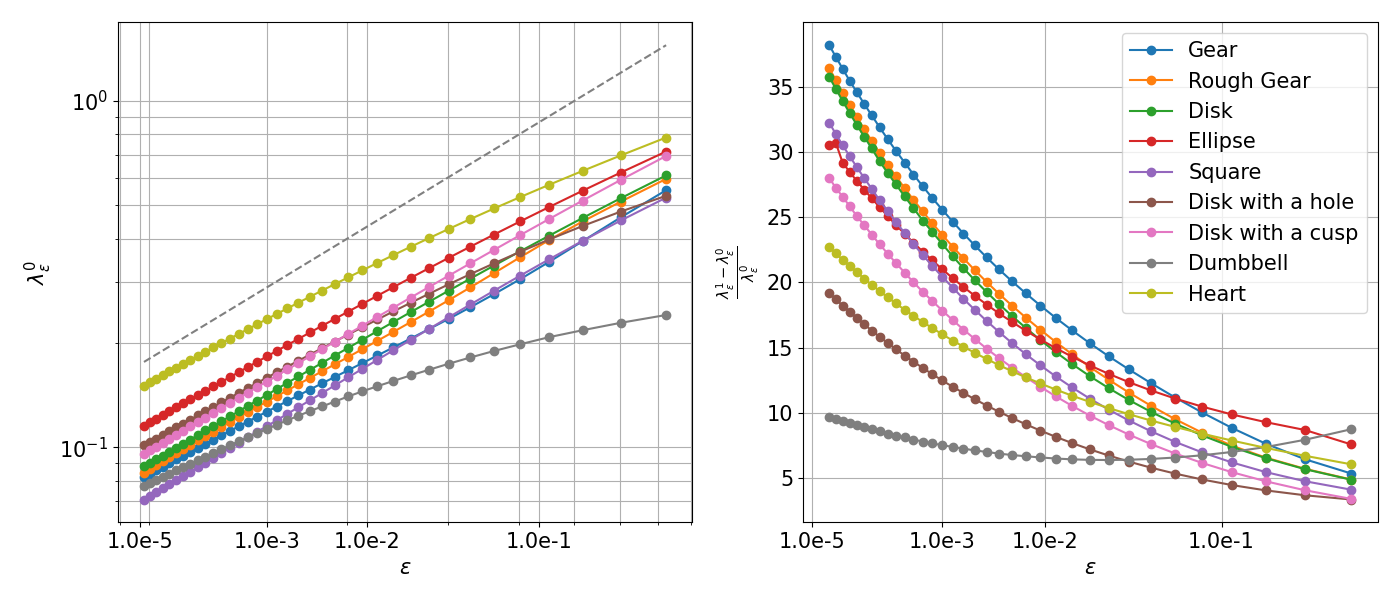}
    \caption{(Left) First eigenvalue $\lambda_\varepsilon^0$ in terms of $r_\varepsilon = \varepsilon$ for the various shapes. 
    The grey dashed line is the theoretical slope from \cref{thm:lam0}.
    (Right) First eigengap $\lambda_\varepsilon^1 - \lambda_\varepsilon^0$ in terms of $r_\varepsilon$.
    These results have been obtained with FEM in Gridap (see first paragraph of \cref{sec:numerical methods}).}
    \label{fig:different shapes}
\end{figure}
Visually, we see from \cref{fig:different shapes} that,
for all domains except the dumbbell,
the correct asymptotic scaling of the eigenvalue~$\lambda_{\varepsilon}^0$ is observed for small $r_\varepsilon^{(k)}$:
to see this,
compare the slope of the data with that of the grey dashed line in the left panel in~\cref{fig:different shapes}.
This does not hold for the dumbbell, likely because of the presence of a
narrow canal between the two disks, which induces metastability in the dynamics. 
The asymptotic regime is not observed but would likely be reached for smaller values of $r_\varepsilon$.
This feature is reflected in the fact that the ratio of the eigenvalues $\frac{\lambda_\varepsilon^1 - \lambda_\varepsilon^0}{\lambda_\varepsilon^0}$ is small compared to the values obtained for the other domains. 
We therefore make the experimental conjecture that, 
in order to observe the asymptotic scaling provided by our results, we need
\begin{equation*}
    \frac{\lambda_\varepsilon^1 - \lambda_\varepsilon^0}{\lambda_\varepsilon^0} \gg 1, \qquad
    r_\varepsilon \ll 1.
\end{equation*}
The second condition is natural given the scaling of the remainder terms in the theorem \cref{thm:lam0}. 
The first condition quantifies that the time scale to reach the QSD should be much smaller than the mean exit time.
Indeed, 
$\left(\lambda_\varepsilon^0\right)^{-1}$ is the mean exit time starting from the QSD, 
and $\left(\lambda_\varepsilon^1 - \lambda_\varepsilon^0\right)^{-1}$ is a rough estimate of the time scale required to reach the QSD,
starting from another distribution 
(see \eqref{eq:lambda0lambda1} in the proof of \cref{thm:QSD}). 
When these two time scales are of similar magnitude, this indicates the presence of internal metastabilities within the domain.

\subsection{The exit hole distribution}
\label{sec:exit hole dist}
In this subsection, we illustrate \cref{thm:exit hole distribution} and discuss the influence of the dimension on the results.
We consider the unit disk ($d=2$), ball ($d=3$), and hyperball ($d=4$) with two holes, one of radius $r_\varepsilon^{(1)} = 2 \varepsilon $ and the other one of radius $r_\varepsilon^{(2)} = \varepsilon$.
From~\cref{thm:exit hole distribution}
and \cref{rq:exit hole distribution}, 
we know that the probability to exit through the larger hole is $\frac{1}{2}$ at first order in dimension~$2$, 
$\frac{2}{3}$ in dimension $3$, and $\frac{4}{5}$ in dimension $4$.
In \cref{fig:small hole}, we report estimates obtained with \cref{alg:wos}. 
We see that the result is in good agreement with the theory in dimension $3$, 
and also in dimension~$4$ for a sufficiently small~$\varepsilon$. 
However, in dimension $2$, it is hard to see the asymptotic regime, 
as reaching $r_\varepsilon$ sufficiently small is computationally too expensive.
Indeed, the last point corresponds to~$r_\varepsilon = 10^{-5}$ ($K_\varepsilon \approx 0.1$), 
which is already of the same order as $\sqrt{h}$, and so the bias induced by the time step is not negligible.
Nevertheless, 
by fitting the estimate of the probability to leave through the larger hole with a function of the form $f(\varepsilon) = a + b \Keo$,
we find that $a \approx 0.5$, which is consistent with the theoretical estimates as $\varepsilon \to 0$.

\begin{figure}
    \centering
    \includegraphics[width=0.6\textwidth]{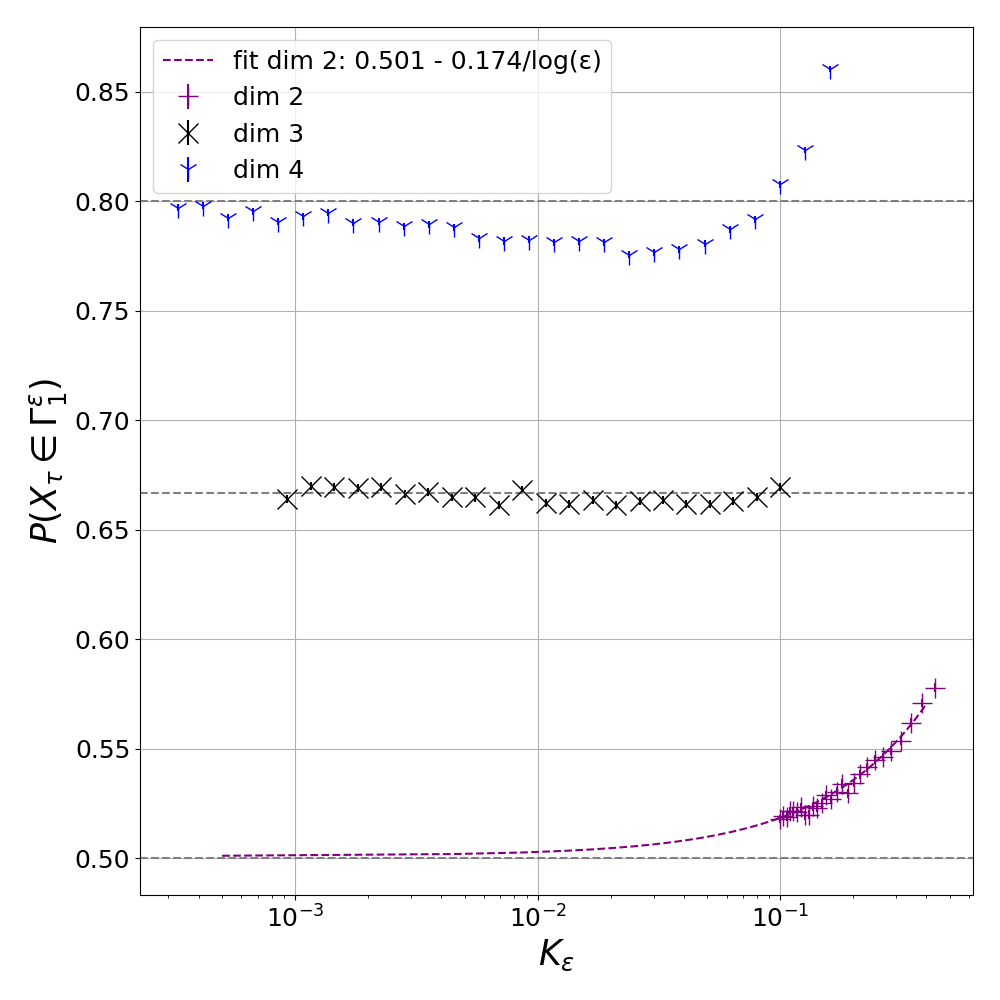}
    \caption{Probability to exit by the large hole in terms of $K_\varepsilon$. 
    The black dashed lines are the theoretical predictions of \cref{thm:exit hole distribution}.
    }
    \label{fig:small hole}
\end{figure}

\hop \hop 
\paragraph{Acknowledgements.}\\
We warmly thank Clément Guillot, Vincent Boulard, Lo\"is Delande and Boris Nectoux 
for the stimulating discussions. Part of this work is funded by the Agence
Nationale de la Recherche through the projects SINEQ (ANR-21-CE40-0006), IPSO (ANR-23-CE40-0027), and DySLoS (ANR-25-CE40-6875).
The authors have received funding from the European Research Council (ERC) 
under the European Union’s Horizon 2020 research and innovation programme (grant agreement No 810367), project EMC2.

\appendix

\section{A priori bounds on the smallest eigenvalues of the Laplacian with mixed boundary conditions}
\label{sec:app eigenvalues}
This section contains two lemmas presenting a priori bounds on the smallest eigenvalues of the Laplacian 
with mixed boundary conditions. 
These are both used in the main body of the paper, 
but also in~\cref{sec:app QSD}. 
It also includes a result in \cref{lem:sign eigenfunction} on the sign of their eigenfunctions. 
To this end, we denote 
\begin{equation}\label{eq:form domain}
    H^1_{0, \Gamma_\mathcal{D}^\varepsilon}(\Omega_\varepsilon) = \{ u \in H^1(\Omega_\varepsilon), \, u = 0 \text{ on } \Gamma_\mathcal{D}^\varepsilon\},
\end{equation}
where $u = 0$ on $\Gamma_\mathcal{D}^\varepsilon$ is understood in the sense of traces.
The associated quadratic form defined on~$H^1_{0, \Gamma_\mathcal{D}^\varepsilon}(\Omega_\varepsilon)$ is denoted by
\[
    q^\varepsilon(u) = \int_{\Omega_\varepsilon} \left|\nabla u \right|^2\, \mathrm{d} x.
\]
The operator $\mathcal{L}_\varepsilon$ introduced in \cref{sec:quasi_stationary} on the domain 
$\mathcal{D}(\mathcal{L}_\varepsilon)$ defined in \eqref{eq:domain L} is 
constructed rigorously as the Friedrichs extension of $q^\varepsilon$ 
(see \cite[Proposition 2.1]{Tony2024}).
\begin{lem}\label{lem:eigenvalues}
    Consider the operator $\mathcal{L}_\varepsilon$ and its smallest eigenvalue
    $\lambda_\varepsilon^0$.
    Then, there exists a constant~$C > 0$ independent of $\varepsilon$ such that,
    for~$\varepsilon$ sufficiently small,
    \[
        0 < \lambda_\varepsilon^0 < C \Keo,
    \]
    using the notation introduced in \eqref{eq:def Keo}.
\end{lem}
\begin{proof}
    From the min-max principle, 
    the smallest eigenvalue $\lambda_\varepsilon^0$ can be expressed as the infimum of the Rayleigh quotient
    \begin{equation}\label{eq:Rayleigh}
        \lambda_\varepsilon^0 = \inf_{u \in  H^1_{0, \Gamma_\mathcal{D}^\varepsilon}(\Omega_\varepsilon)  \setminus \{0\}} \frac{\displaystyle\int_{\Omega_\varepsilon} |\nabla u|^2}{\displaystyle\int_{\Omega_\varepsilon} |u|^2}.
    \end{equation}
    From Poincar\'e's inequality for functions in the form domain $H^1_{0, \Gamma_\mathcal{D}^\varepsilon}(\Omega_\varepsilon)$,
    one has $\lambda_\varepsilon^0 > 0$ for any $\varepsilon > 0$.
    To prove the upper bound on $\lambda_\varepsilon^0$, 
    let us exhibit a good test function in the form 
    domain~$H^1_{0, \Gamma_\mathcal{D}^\varepsilon}(\Omega_\varepsilon)$.
    As the holes are pairwise disjoint, there exists $R > 0$ such that $B(x^{(k_0)}, R) \cap B(x^{(k)}, R) = \emptyset$ for all $k \neq k_0$.
    Consider now, for $d \geq 3$, $\alpha > 0$ and any hole $k \in \{1, \ldots, N\}$,
    the function $\eta_\varepsilon^{\alpha, k} : \real_+ \to [0, \, 1]$ defined as follows:
    \begin{equation}
        \label{eq:def eta}
            \eta_\varepsilon^{\alpha, k}(r) = \left\{
                \begin{aligned}
                & 0 && \quad \text{ if } r \geq R,\\
                & 1 && \quad \text{ if } r \leq r_\varepsilon^{(k)},\\
                & \frac{r^{-\alpha} - R^{-\alpha}}{\left( r_\varepsilon^{(k)}\right)^{-\alpha} - R^{-\alpha}} && \quad \text{ if } r_\varepsilon^{(k)} < r < R.
        \end{aligned}
        \right.
    \end{equation}
    A test function in $H^1_{0, \Gamma_\mathcal{D}^\varepsilon}(\Omega_\varepsilon)$ 
    can be built from $\eta_\varepsilon^{\alpha, k}$ as follows:
    \begin{equation}
        \label{eq:quasi-mode form domain}
        \forall x \in \Omega_\varepsilon, \qquad \phi_\varepsilon(x) = 1 - \sum_{k=1}^N\eta_\varepsilon^{\alpha, k}(|x - x^{(k)}|).
    \end{equation}
    Let us now compute its Rayleigh quotient.
    For the $L^2$ norm, 
    notice first that, since $\Omega_\varepsilon \subset \Omega$,
    \[
        \int_{\Omega_\varepsilon} \left| \phi_\varepsilon \right|^2 \leq 
        \int_{\Omega} \left| \phi_\varepsilon \right|^2.
    \]
    It is clear that $\eta_\varepsilon^{\alpha, k} \to 0$ almost everywhere in $\real_+$ as $\varepsilon \to 0$.
    Furthermore, 
    $\eta_\varepsilon^{\alpha, k}$ is bounded by $1$ uniformly in $\varepsilon$.
    Thus, 
    by the dominated convergence theorem,
    \begin{equation}\label{eq:convergence eta}
        \| \eta_\varepsilon^{\alpha, k} \|_{L^2(\real_+)} \xrightarrow[\varepsilon \to 0]{} 0,
    \end{equation}
    hence $ \phi_\varepsilon \to 1$ in $L^2(\Omega)$ as $\varepsilon \to 0$. 

    Now to estimate the gradient of $\phi_\varepsilon$, 
    let us first compute, for any $\alpha > 0$,
    the derivative of $\eta_\varepsilon^{\alpha, k}$:
    \begin{equation}
        \forall r \in (r_\varepsilon^{(k)}, R) \qquad \left(\eta_\varepsilon^{\alpha, k}\right)'(r) = \frac{-\alpha r^{-\alpha-1}}{\left( r_\varepsilon^{(k)}\right)^{-\alpha} - R^{-\alpha}}.
    \end{equation}
    Notice that $\nabla \phi_\varepsilon$ is supported on 
    the union of the annuli~${A_k^\varepsilon = \{ x \in \Omega_\varepsilon \, |\, r_\varepsilon^{(k)} < |x - x^{(k)}| < R\}}$,
    and~$\| \nabla \phi_\varepsilon \|_{L^2(\Omega_\varepsilon)}$ can be computed as follows, for $\alpha$ such that $2 \alpha > d-2$:
    \begin{equation}\label{eq:gradient eta}
        \begin{aligned}
            \frac{1}{\omega_d}\| \nabla \phi_\varepsilon \|_{L^2(\Omega_\varepsilon)}^2
            \leq & \sum_{k=1}^N \int_{r_\varepsilon^{(k)}}^{R} \left| \left(\eta_\varepsilon^{\alpha, k}\right)'(r) \right|^2 r^{d-1} \, \mathrm{d} r\\
            = & \sum_{k=1}^N \frac{\alpha^2}{\left(\left( r_\varepsilon^{(k)}\right)^{-\alpha} - R^{-\alpha}\right)^2} \int_{r_\varepsilon^{(k)}}^{R} r^{-2\alpha - 2 + d - 1} \, \mathrm{d} r\\
            = & \sum_{k=1}^N \frac{\alpha^2}{\left(\left( r_\varepsilon^{(k)}\right)^{-\alpha} - R^{-\alpha}\right)^2} 
            \frac{R^{d -2 \alpha -2}-\left( r_\varepsilon^{(k)}\right)^{d -2 \alpha -2}}{d-2 \alpha -2} .
    \end{aligned}
    \end{equation}
    Thus at the leading order as $\varepsilon \to 0$,
    one has 
    \begin{equation}\label{eq:gradient eta 2}
        \| \nabla \phi_\varepsilon \|_{L^2(\Omega_\varepsilon)}^2
        \sim \omega_d \frac{\alpha^2}{2 \alpha - (d-2)} \, \sum_{k=1}^N\left( r_\varepsilon^{(k)}\right)^{d- 2}.
    \end{equation}
    To conclude, 
    gathering \eqref{eq:Rayleigh}, \eqref{eq:gradient eta} and \eqref{eq:gradient eta 2},
    one has that 
    \begin{equation} \label{eq:upper bound lambda equivalent}
        \lambda_\varepsilon^0 
        \leq \frac{\displaystyle\int_{\Omega_\varepsilon} |\nabla \phi_\varepsilon|^2\,\mathrm{d} x}{\displaystyle\int_{\Omega_\varepsilon} |\phi_\varepsilon|^2\,\mathrm{d} x} 
        = \frac{\displaystyle \frac{\alpha^2}{2 \alpha - (d-2)} \omega_d \sum_{k=1}^N \left( r_\varepsilon^{(k)}\right)^{d- 2} (1 + \mathrm{o}(1))}{|\Omega_\varepsilon| + \mathrm{o}(1)}
        \sim \frac{\alpha^2}{2 \alpha - (d-2)} \frac{\omega_d}{|\Omega|} \sum_{k=1}^N \left( r_\varepsilon^{(k)}\right)^{d- 2},
    \end{equation}
    with $\omega_d$ the surface area of the unit sphere in $\mathbb{R}^d$ introduced in \cref{thm:lam0}.
    This, in addition to \eqref{eq:def K}, concludes the proof in dimension $d \geq 3$.
    For the $2$ dimensional case, 
    a similar computation can be done with 
    \begin{equation*}
            \eta_\varepsilon^{k}(r) = \left\{
                \begin{aligned}
                & 0 && \quad \text{ if } r \geq R,\\
                & 1 && \quad \text{ if } r \leq r_\varepsilon^{(k)},\\
                & \frac{\log(r/R)}{\log(r_\varepsilon^{(k)}/R)} && \quad \text{ if } r_\varepsilon^{(k)} < r < R.
        \end{aligned}
        \right.
        \qedhere
    \end{equation*}
\end{proof}
\begin{rema}
It is natural to optimize the upper bound obtained in \eqref{eq:upper bound lambda equivalent} with respect to $\alpha$.
The optimal choice in dimension $d \geq 3$ is $\alpha = d-2$, which gives the following constant:
\[
    \lambda_\varepsilon^0 \leq (d-2) \frac{\omega_d}{|\Omega|} \sum_{k=1}^N \left( r_\varepsilon^{(k)}\right)^{d-2} (1 + o(1)).
\]
This upper bound is identical to the result obtained in \cref{thm:lam0} up to a factor $2$.
This factor can be removed by using a more refined computation of the Rayleigh quotient, 
using half of the annulus $A_k^\varepsilon$, 
in the same manner as in the proof of \cref{lem:compatibility constant}.
Yet notice that $\phi_\varepsilon$ is not a good candidate to be a quasi-mode as $\Delta \phi_\varepsilon \not \in L^2(\Omega_\varepsilon)$, 
and in fact is null except on $\Omega \cap \left( \bigcup_{k=1}^N \partial B(\xk, R) \right)$.
\end{rema}

The last two lemmas of this section state that the smallest eigenvalue $\lambda_\varepsilon^0$
is simple, with an associated eigenfunction of constant sign,
and that the second eigenvalue of $\mathcal{L}_\varepsilon$ is uniformly bounded from below
by a positive constant independent of $\varepsilon$.
The proof of the latter relies on the fact that the domain satisfies
the uniform interior cone condition,
hence the result is valid in the context of the main paper,
but also for other hole shapes, as in~\cref{sec:app QSD}.

\begin{lem}\label{lem:sign eigenfunction}
    Consider the operator $\mathcal{L}_\varepsilon$ introduced in \cref{sec:quasi_stationary}, and its smallest eigenvalue $\lambda_\varepsilon^0$.
    Then, $\lambda_\varepsilon^0$ is simple, and any eigenfunction $u_\varepsilon^0$ associated with $\lambda_\varepsilon^0$ has a constant sign in $\Omega_\varepsilon$.
\end{lem}
\begin{proof}
    This result is classical for elliptic operators, see for instance \cite[Theorem 8.38]{Gilbrag}.
    Let us recall the arguments. Consider an eigenfunction $u_\varepsilon^0 \in \mathcal{D}(\mathcal{L}_\varepsilon)$ 
    associated with $\lambda_\varepsilon^0$.
    Then, by the definition of the Rayleigh quotient, 
    \begin{equation}\label{eq:Raylight_equality}
        \lambda_\varepsilon^0 = \frac{\displaystyle\int_{\Omega_\varepsilon} |\nabla u_\varepsilon^0|^2}{\displaystyle\int_{\Omega_\varepsilon} |u_\varepsilon^0|^2}.
    \end{equation}
    However,
    notice that the function $|u_\varepsilon^0|$ also belongs to the form domain $H_{0, \Gamma_\mathcal{D}^\varepsilon}^1(\Omega_\varepsilon)$ and satisfies
    \[
        \frac{\displaystyle\int_{\Omega_\varepsilon} |\nabla |u_\varepsilon^0||^2}{\displaystyle\int_{\Omega_\varepsilon} |u_\varepsilon^0|^2} 
        = \frac{\displaystyle\int_{\Omega_\varepsilon} |\nabla u_\varepsilon^0|^2}{\displaystyle\int_{\Omega_\varepsilon} |u_\varepsilon^0|^2} = \lambda_\varepsilon^0.
    \]
    Hence it is also an eigenfunction associated with $\lambda_\varepsilon^0$.
    From the interior Harnack inequality applied to $|u_\varepsilon^0|$ (see \cite[Theorem 8.20]{Gilbrag}), 
    one has that $|u_\varepsilon^0|$ is positive in $\Omega_\varepsilon$. 
    This shows that $u_\varepsilon^0$ does not change sign in~$\Omega_\varepsilon$.
    Finally, 
    since eigenfunctions associated with $\lambda_\varepsilon^0$ do not change sign,
    there cannot exist two of them that are orthogonal in $L^2(\Omega_\varepsilon)$,
    hence $\lambda_\varepsilon^0$ is simple.
\end{proof}

\begin{lem}\label{lem:spectral gap}
    Consider the operator $\mathcal{L}_\varepsilon$ and its second eigenvalue $\lambda_\varepsilon^1$.
    Then there exists a constant $C > 0$ such that, for all $\varepsilon > 0$, 
    \[
        C < \lambda_\varepsilon^1.
    \]
\end{lem}
\begin{proof}
    We follow a similar strategy as in the proof of Proposition 2.3 and Lemma 3.1 of \cite{Tony2024}. 
    From the inclusion 
    of the domain $H^1_{0, \Gamma_\mathcal{D}^\varepsilon}(\Omega_\varepsilon)$ 
    of the form $q^\varepsilon$ into the form domain of the Laplacian with full Neumann boundary conditions, 
    namely $H^1(\Omega_\varepsilon)$,
    one has that 
    \begin{equation}\label{eq:spectral gap 1}
        \lambda_\varepsilon^1 \geq \lambda_1^\mathrm{Neumann}(\Omega_\varepsilon).
    \end{equation}
    To continue, 
    notice that $\Omega_\varepsilon$ defined as \eqref{eq:omega varepsilon} satisfies the uniform interior cone condition,
    thus, the Poincaré-Wirtinger inequality holds in $\Omega_\varepsilon$ with a constant independent of $\varepsilon$ (see \cite[Theorem 1.2]{Ruiz}).
    This directly shows that there exists $C>0$ independent of $\varepsilon$ such that 
    \begin{equation}\label{eq:spectral gap 2}
        \lambda_1^\mathrm{Neumann}(\Omega_\varepsilon) \geq C > 0.
    \end{equation}
    Combining \eqref{eq:spectral gap 1} and \eqref{eq:spectral gap 2} yields the claimed result.
\end{proof}

\section{\texorpdfstring{Proofs of \cref{thm:QSD} and \cref{cor:QSD}}{Proofs of the quasi-stationary results}}
\label{sec:app QSD}
This section aims at proving \cref{thm:QSD} and \cref{cor:QSD} for $\varepsilon > 0$ fixed.
Hence, we drop the dependence on $\varepsilon$ in the notation: in this section
 we simply denote $\Omega = \Omega_\varepsilon$,
$\Gamma_\mathcal{D} = \Gamma_\mathcal{D}^\varepsilon$ and
$\Gamma_\mathcal{N} = \Gamma_\mathcal{N}^\varepsilon$.
The results and the proofs of this section do not require any particular
geometry of the holes as in~\eqref{eq:omega varepsilon}.
The main assumption is that $\partial \Omega$ is at least Lipschitz and
the only non-smooth part of
$\partial \Omega$ is at $\overline{\Gamma_\mathcal{D}} \cap \overline{\Gamma_\mathcal{N}}$.
More precisely, both $\Gamma_\mathcal{D}$ and $\Gamma_\mathcal{N}$ are smooth
and at each point of their boundary, they meet at an angle smaller than or equal to $\pi$
(see \cite[Definition 31]{TonyBoris2022} for a precise definition of the angle between two manifolds).
For instance the results of this section are valid in the context of the holes considered in~\cite{narrow2d},
where the holes are not spherical.
In the proofs below, we distinguish between two cases:
either the whole boundary~$\partial \Omega$ is smooth,
or the holes meet the outer boundary of the domain with an angle smaller than~$\pi$
(these are so-called \textit{creased} domains, see \cite[Section 2]{MitreaMitrea}).

In order to prove the existence and the uniqueness of the quasi-stationary distribution,
we first prove a Feynman--Kac formula for the following process:
\begin{equation}\label{eq:def process reflexion}
    \d X_t = \sqrt{2} \d W_t - \mathbbm{1}_{X_t \in \partial \Omega} n(X_t) \d L_t,
\end{equation}
where $W_t$ is the standard Brownian motion and $L$ is the local time on the boundary $\partial \Omega$.
We denote by~$\tau_A$ the following first entrance time in the set $A \subset \overline{\Omega}$:
\begin{equation} \label{eq:def stopping time}
    \tau_A = \inf \{ t \geq 0\, |\, X_t \in A\}.
\end{equation}
Let us first show that the law of the killed process
has a bounded density.
\begin{lem}\label{lem:proba}
    Let $\mu$ be a probability measure on $\Omega$
    and, for $t > 0$,
    define the measure $Q_t^\mu$ on $\overline{\Omega}$ by
    \[
        Q^\mu_t(A) = \mathbb{P}_\mu(X_t \in A, \, t < \tau_{\overline{\Gamma_\mathcal{D}}}),
        \qquad A \subset \overline{\Omega} \text{ measurable}.
    \]
    Then $Q_t^\mu$ is absolutely continuous with respect to the Lebesgue measure on $\overline{\Omega}$
    and its density $q_t^\mu$ belongs to $L^2(\Omega)$ for any $t > 0$.
\end{lem}

\begin{proof}
    For any measurable set $A \subset \overline{\Omega}$ and any $x \in \overline{\Omega}$,
    it holds that 
    \[
        \mathbb{P}_x(X_t \in A, \, t < \tau_{\overline{\Gamma_\mathcal{D}}})
        \leq \mathbb{P}_x(X_t \in A)
        = \int_A p_t(x, y) \, \d y,
    \]
    where $p_t(x, y)$ is the transition density of the process~\eqref{eq:def process reflexion}, 
    (see \cite[Lemma 4.3, Theorem 4.4 ]{BassHsu} for its existence and regularity).
    Integrating this inequality with respect to $\mu(\mathrm{d}x)$ leads to
    \[
        Q_t^\mu(A) \leq \int_A \left( \int_{\overline{\Omega}} p_t(x, y) \, \mu(\mathrm{d}x) \right) \d y,
    \]
    so that $Q_t^\mu(A) = 0$ whenever $A$ is Lebesgue negligible.
    Hence $Q_t^\mu$ admits a density $q_t^\mu$ with respect to the Lebesgue measure,
    and this density satisfies, for almost every $y \in \Omega$,
    \begin{equation}\label{eq:bound density}
        q_t^\mu(y) \leq \int_{\overline{\Omega}} p_t(x, y) \, \mu(\mathrm{d}x)
        \leq \sup_{x', y' \in \overline{\Omega}} p_t(x', y'),
    \end{equation}
    where we used that $\mu$ is a probability measure.
    The right-hand side of \eqref{eq:bound density} is bounded by $C t^{-\frac{d}{2}}$ for~$t \in (0, T]$,
    as a consequence of the usual Gaussian upper bound on the transition density of the process~\eqref{eq:def process reflexion}:
    \[
        p_t(x,y) \leq C\, t^{-d/2} \exp\!\left( - \frac{|x-y|^2}{Ct} \right),
        \qquad x,y \in \overline{\Omega}, \, t \in (0, T],
    \]
    see \cite[Theorem 3.1]{BassHsu} for $d \geq 3$,
    the two dimensional case following from \cite[Remark 2.4]{Ramasubramanian}.
    Inserting this bound into \eqref{eq:bound density} shows that $q_t^\mu \in L^2(\Omega)$ for any $t > 0$.
\end{proof}

Consider next the negative Laplacian operator, denoted by~$\mathcal{L}$,
with domain $\mathcal{D}(\mathcal{L})$ introduced in \eqref{eq:domain L}.
Since $\Omega$ is bounded with a Lipschitz boundary,
the operator $\mathcal{L}$ is self-adjoint with compact resolvent,
so there exists an orthonormal basis $(u_k)_{k \geq 0}$ in $L^2(\Omega)$ consisting of eigenvectors of $\mathcal{L}$,
associated to an increasing sequence of eigenvalues $(\lambda_k)_{k \geq 0}$.
For $k \geq 0$,
the eigencouple $(\lambda_k, u_k)$ satisfies
\begin{equation}
    \label{eq:elliptic eigen}
    \left\{
        \begin{aligned}
    - \Delta u_k & = \lambda_k u_k && \text{ in } \Omega, \\
    \partial_n u_k & = 0 && \text{ on } \Gamma_\mathcal{N}, \\
    u_k & = 0 && \text{ on } \Gamma_\mathcal{D}.
        \end{aligned}
    \right.
\end{equation}

First, we state a technical result on the regularity of the eigenfunctions $(u_k)_{k \geq 0}$.
\begin{lem}\label{lem:continuity}
    For any $k \geq 0$, the eigenfunction $u_k$ is continuous on $\overline{\Omega}$, 
    and smooth on $\overline{\Omega} \setminus \left( \overline{\Gamma_\mathcal{D}} \cap \overline{\Gamma_\mathcal{N}} \right)$.
    Furthermore, 
    there exist $C > 0$ and $m>0$, independent of $k$, such that
    \begin{equation}\label{eq:c0 bound}
        \lVert u_k \rVert_{C^0(\overline{\Omega})}
        \leq C \lambda_k \left( 1 + \lambda_k \right)^{m} \lVert u_k \rVert _{L^{2}(\Omega)}.
    \end{equation}
\end{lem}

\begin{proof}
    Notice that the domain $\overline{\Omega} \setminus \left( \overline{\Gamma_\mathcal{D}} \cap \overline{\Gamma_\mathcal{N}} \right)$
    does not contain the intersection of the Dirichlet and Neumann parts of the boundary.
    Hence, 
    locally the problem \eqref{eq:elliptic eigen} is either a Dirichlet or a Neumann problem on a smooth domain.
    The smoothness of the eigenfunctions $u_k$ on $\overline{\Omega} \setminus \left( \overline{\Gamma_\mathcal{D}} \cap \overline{\Gamma_\mathcal{N}} \right)$
    follows then from standard elliptic regularity results (see \cite[Theorem 9.25, Theorem 9.26]{bezis}).

    For the continuity of $u_k$ on $\overline{\Omega}$, we split the proof into two parts.

    \textbf{Step 1: Elliptic regularity.}
    Let us start by recalling some results on elliptic regularity.
    Consider the Poisson equation with mixed boundary conditions
    \begin{equation}\label{eq:elliptic f}
        \left\{
            \begin{aligned}
        - \Delta u & = f \text{ in } \Omega, \\
        \partial_n u & = 0 \text{ on } \Gamma_\mathcal{N}, \\
        u & = 0\text{ on } \Gamma_\mathcal{D},
         \end{aligned}
        \right.
    \end{equation}
    where $f \in L^2(\Omega)$.
    This problem is known as the Zaremba problem
    and the existence of a unique solution in the weak $H^1$ sense is well established (see \cite{Savare}).
    Let $u$ denote the solution.
    Regarding the regularity of $u$,
    from~\cite[Theorem 1]{Savare},
    we know that
    \begin{equation}\label{eq:regularity mitrea}
        f \in L^2(\Omega) \Rightarrow u \in H^{\frac{3}{2} - \eta}(\Omega) \qquad \text{ for } \eta > 0.
    \end{equation}
    For $L^p$ regularity,
    we make the distinction between smooth and creased domains.
    For smooth domains~\cite[Equation (0.3)]{Savare} and \cite[Theorem 5.3]{shamir},
    one obtains, for any $p > 4$ and $s < \frac{1}{2} + \frac{2}{p}$,
    \begin{equation}
        \label{eq:regularity sharmir}
        f \in L^p(\Omega) \Rightarrow u \in W^{s, p}(\Omega).
    \end{equation}
    For creased domains
    (domains such that all the angles between $\Gamma_{\mathcal{D}}$ and $\Gamma_{\mathcal{N}}$ are smaller than $\pi$),
    \cite[Theorem of Section 1]{MitreaMitrea}
    states that, for $d \geq 3$, there exists $\delta \in (0, 1)$
    such that, for any $p \in (2, \infty)$,
    one has
    \begin{equation}
        \label{eq:regularity Mitrea2}
        f \in L^p(\Omega) \Rightarrow u \in W^{\delta, p}(\Omega).
    \end{equation}
    However,
    in these references,
    the authors do not always provide an inequality relating the norm of the source $f$
    to the Sobolev norm of the solution $u$.
    The first part of the proof is to prove these inequalities. \\
    Consider the Banach space $E$ defined, for $p > 4$ and $0 < s < \frac{1}{2} + \frac{2}{p}$, as
    \[
        E = \Bigl\{ u \in W^{s, p}(\Omega) \cap H^1(\Omega), \, \Delta u \in L^p(\Omega), \,
        u = 0 \text{ on } \Gamma_\mathcal{D}, \, \partial_n u = 0 \text{ on } \Gamma_\mathcal{N}\Bigr\},
    \]
    endowed with the norm
    \[
        \|u\|_E = \|u\|_{W^{s, p}(\Omega)} + \|u\|_{H^1(\Omega)} + \|\Delta u\|_{L^p(\Omega)}.
    \]
    Recall that the condition $u \in H^1(\Omega)$ ensures that
    the normal derivative $\partial_n u$ indeed has a weak sense as an element of $H^{-1/2}(\Gamma_{\mathcal{N}})$,
    as defined in \eqref{eq:normal_derivative_weak_h12}.
    Then the operator $\mathcal{T}$ defined by
    \begin{equation*}
        \mathcal{T} :
        \left\{
            \begin{aligned} E & \to L^p(\Omega)\\
        u & \mapsto \Delta u
            \end{aligned}
        \right.
    \end{equation*}
    is a continuous linear bijection.
    Continuity and linearity are clear from the definition of $\mathcal{T}$, while
    bijectivity follows from~\eqref{eq:regularity sharmir} proved in~\cite{shamir},
    together with~\eqref{eq:regularity mitrea} to obtain the $H^1$ regularity.
    Thus, by the bounded inverse theorem~\cite[Corollary 2.7]{bezis},
    the inverse $\mathcal{T}^{-1}$ is also continuous:
    for any $f \in L^p(\Omega)$,
    \begin{equation*}
        \left\| \mathcal{T}^{-1}f \right\|_{E} \leq C_{s, p} \left\|f \right\|_{L^p(\Omega)},
    \end{equation*}
    which implies that $u$ the solution of \eqref{eq:elliptic f} satisfies
    \begin{equation}\label{eq:constant C sp}
        \left\| u \right\|_{W^{s, p}}
        \leq \left\| u \right\|_{E} \leq C_{s, p} \left\|f \right\|_{L^p(\Omega)},
    \end{equation}
    with $C_{s, p} > 0$ a positive constant that depends on $s$, $p$ and possibly $\Omega$,
    but not on $f$.
    This reasoning can also be applied in the case of \eqref{eq:regularity mitrea},
    where $f \in L^2(\Omega)$
    implies $u \in  H^{\frac{3}{2} - \eta}(\Omega)$ for any $\eta > 0$.

    \textbf{Step 2: $L^{\infty}$ estimates on the eigenfunctions.}
    Fix $k \geq 0$.
    To determine the regularity of~$u_k$ the eigenfunction associated with the eigenvalue $\lambda_k$ introduced in \eqref{eq:elliptic eigen},
    we observe that the solution~$u_k$ is a $H^1(\Omega)$ solution to \eqref{eq:elliptic f} with right-hand side $f= \lambda_k u_k \in L^2(\Omega)$.
    Hence, thanks to \eqref{eq:regularity mitrea},
    the function $u_k$ is in $H^{\frac{3}{2} - \eta}(\Omega)$ for any~$\eta > 0$.
    In particular,
    the Sobolev embedding theorem \cite{Gilbrag} implies that, for the two dimensional case,
    $u_k$ is already in fact bounded and continuous on~$\overline{\Omega}$,
    and in this case, this step is complete.
    In higher dimensions,
    we use a bootstrap argument to show the boundedness of the eigenfunctions.
    This argument is based on the Moser iteration scheme,
    introduced by \cite{MoserOld}
    (see a similar application of the scheme in \cite[Proof of Theorem 5.3]{MoserTrudinger}).

    Assume, for a fixed $k \geq 0$,
    that
    $u_k \in L^p(\Omega)$ for some $p \geq 2$,
    and consider the function $u_{k, M} = \mathrm{sign}(u_k)\, \mathrm{min} (|u_k|, M)$ for some positive integer~$M \in \mathbb{N}^*$,
    and the test function~$\phi_{k, M} = \left|u_{k, M}\right|^{p-2} u_{k, M}$.
    Thanks to the cutoff,
    one easily checks that~$\phi_{k, M} \in
    \{ f \in H^1(\Omega), \, f = 0 \text{ on } \Gamma_\mathcal{D} \}$,
    so it is an admissible test function for the weak formulation of \eqref{eq:elliptic eigen}.
    Hence,
    we have that, as the boundary conditions are all homogeneous,
    \begin{equation}\label{eq:moser iteration step1}
        \int_{\Omega} \nabla u_k \cdot \nabla \phi_{k, M} = \lambda_k \int_{\Omega} u_k \phi_{k, M}.
    \end{equation}
    Since $\nabla u_{k, M} = \mathbbm{1}_{|u_k| \leq M} \nabla u_k$,
    the left-hand side can be rewritten as
    \begin{equation}\label{eq:moser iteration lhs}
        \int_{\Omega} \nabla u_k \cdot \nabla \phi_{k, M} =
        (p-1) \int_{\Omega} \left| \nabla u_{k, M} \right|^2 |u_{k, M}|^{p-2}
        = \frac{4(p-1)}{p^2}\int_{\Omega} \left| \nabla \left( |u_{k, M}|^{\frac{p}{2}} \right) \right|^2,
    \end{equation}
    where the second equality has been obtained through the chain rule.
    For the right-hand side of \eqref{eq:moser iteration step1},
    notice that~$|u_{k, M}| \leq |u_k|$,
    hence
    \begin{equation}\label{eq:moser iteration rhs}
        \lambda_k \int_{\Omega} u_k \phi_{k, M} \leq \lambda_k \int_{\Omega} |u_k|^p.
    \end{equation}
    Combining \eqref{eq:moser iteration step1}, \eqref{eq:moser iteration lhs} and \eqref{eq:moser iteration rhs},
    we obtain that
    \begin{equation}\label{eq:moser uniform bound}
        \frac{4(p-1)}{p^2}\int_{\Omega} \left| \nabla \left(|u_{k, M}|^{\frac{p}{2}} \right) \right|^2 \leq \lambda_k \int_{\Omega} |u_k|^p.
    \end{equation}
    Now,
    $\left(\left| \nabla \left(|u_{k, M}|^{\frac{p}{2}} \right)\right|^2\right)_{M \in \mathbb{N}^*}$ is a non-decreasing, non-negative sequence of functions
    whose integrals are uniformly bounded thanks to \eqref{eq:moser uniform bound}.
    By the monotone convergence theorem,
    one has that, in the limit~$M \to +\infty$,
    \begin{equation}\label{eq:moser bound}
        \int_{\Omega} \left| \nabla \left(|u_k|^{\frac{p}{2}} \right) \right|^2
        \leq \frac{p^2}{4(p-1)}\lambda_k \int_{\Omega} |u_k|^p.
    \end{equation}
    The estimate \eqref{eq:moser bound} shows that $|u_k|^{\frac{p}{2}}$ is in $H^1(\Omega)$,
    and its $H^1$ norm is controlled by the $L^p$ norm of $u_k$.
    Hence,
    by a final application of the Sobolev embedding theorem,
    we find that $|u_k|^{\frac{p}{2}}$ is in $L^{\frac{2d}{d-2}}(\Omega)$,
    and therefore,
    there exists a constant $C^{\rm Sob} > 0$ such that
    \begin{equation}\label{eq:sobolev inject}
        \lVert u_k \rVert_{L^{\frac{dp}{d-2}}(\Omega)}^{\frac{p}{2}}
        = \left\lVert |u_k|^{\frac{p}{2}} \right\rVert_{L^{\frac{2d}{d-2}}(\Omega)}
            \leq C^{\rm Sob} \left\lVert |u_k|^{\frac{p}{2}} \right\rVert_{H^1(\Omega)}
        \end{equation}
        Finally, using \eqref{eq:moser bound} and \eqref{eq:sobolev inject},
    we find
    \begin{equation*}
        \lVert u_k \rVert_{L^{\frac{dp}{d-2}}(\Omega)}
        \leq \left(C^{\rm Sob}\right)^\frac{2}{p} \left( 1 +  \frac{p^2}{4(p-1)}\lambda_k \right)^{\frac{1}{p}} \lVert u_k \rVert_{L^p(\Omega)}.
    \end{equation*}
    Starting from $p=2$,
    we can iterate this reasoning $m$ times to show that $u_k$ is in $L^{p_m}(\Omega)$
    where~${p_m = 2 \left(\frac{d}{d-2}\right)^m}$,
    and that
    \begin{equation}\label{eq:moser iteration final}
        \lVert u_k \rVert_{L^{p_m}(\Omega)}
        \leq \left( \prod_{j=0}^{m-1} \left(C^{\rm Sob}\right)^\frac{2}{p_j} \left( 1 +  \frac{p_j^2}{4(p_j-1)}\lambda_k \right)^{\frac{1}{p_j}} \right) \lVert u_k \rVert_{L^2(\Omega)}.
    \end{equation}
    Since $p_j \geq 2$, one has that $p_j \leq 4 (p_j - 1)$,
    and thus, for any $j \geq 0$,
    \begin{equation*}
        \left( 1 +  \frac{p_j^2}{4(p_j-1)}\lambda_k \right)^{\frac{1}{p_j}}
         \leq 1 +  \frac{p_j^2}{4(p_j-1)}\lambda_k
         \leq p_j (1 + \lambda_k).
    \end{equation*}
    
    In the smooth domain case, choosing $m^*(d) > \frac{\log(d)}{\log(d/(d-2))}$ in \eqref{eq:moser iteration final},
    we find that there exists $q > 2d$,
    and a constant $c = c(m, \Omega, d) > 0$ that does not depend on $\lambda_k$
    such that $u_k \in L^{q}(\Omega)$ and
    \begin{equation}\label{eq:moser iteration final2}
        \lVert u_k \rVert_{L^{q}(\Omega)}
        \leq c \, \left(1 + \lambda_k\right)^{m^{*}(d)} \lVert u_k \rVert_{L^2(\Omega)}.
    \end{equation}
    Now,
    we use the elliptic regularity result we introduced in the first step of the proof.
    Since $u_k \in L^{q}(\Omega)$ with~$q > 2d \geq 4$,
    the eigenvector $u_k$ is in $W^{\frac{1}{2},q}(\Omega)$ by~\eqref{eq:regularity sharmir}.
    Additionally, by \eqref{eq:constant C sp}, we have
    \begin{equation}\label{eq:constant C sp2}
        \lVert u_k \rVert_{W^{\frac{1}{2}, q}(\Omega)}
        \leq C_{\frac{1}{2}, q} \lambda_k \lVert u_k \rVert_{L^{q}(\Omega)}.
    \end{equation}
    A final application of the Sobolev embedding theorem gives that~$u_k \in C^0(\overline \Omega)$,
    and there exists a constant $C_{\rm Sob} > 0$ such that
    \begin{equation}\label{eq:c0 sobolev}
        \lVert u_k \rVert_{C^0(\overline{\Omega})}
        \leq C_{\rm Sob} \lVert u_k \rVert_{W^{\frac{1}{2}, q}(\Omega)}.
    \end{equation}
    Combining \eqref{eq:moser iteration final2},
    \eqref{eq:constant C sp2} and \eqref{eq:c0 sobolev}
    leads to the existence of a constant~$C > 0$, independent of $k$,
    such that
    \begin{equation}\label{eq:bound L infty norm of uk}
        \lVert u_k \rVert_{C^0(\overline{\Omega})}
        \leq C \lambda_k \left( 1 + \lambda_k \right)^{m^*(d)} \lVert u_k \rVert _{L^{2}(\Omega)}.
    \end{equation}
    This provides the desired result in the smooth domain case.
    
    The exact same technique can be employed in the creased domain case,
    taking $m^*(d) = \frac{\log(d/(2\delta))}{\log(d/(d-2))}$,
    where $\delta$ is the exponent in \eqref{eq:regularity Mitrea2},
    leading also to \eqref{eq:bound L infty norm of uk}.
\end{proof}

We now proceed to the proof a result
that will be key to prove \cref{thm:QSD} and \cref{cor:QSD}.
\begin{prop}[Feynman--Kac formula]\label{thm:feynman-kac}
    For any $k \geq 0$,
    consider the eigencouple $(\lambda_k, u_k)$ solution of~\eqref{eq:elliptic eigen}.
    Then,
    for all $t > 0$ and $x \in \Omega$,
    we have
    \[
        \mathbb{E} _x \left[ u_k(X_t)\, \mathbbm{1}_{t < \tau_{\overline{\Gamma_\mathcal{D}}}}\right]
        = \mathrm{e}^{- \lambda_k t}\, u_k(x) \, ,
    \]
    where $\mathbb{E}_x$ is the expectation with respect to the process defined in \eqref{eq:def process reflexion} starting at $x$.
\end{prop}

\begin{proof}
    Let $t > 0$ and $k \geq 0$.
    For any $0 \leq s \leq t$ and $x \in \overline{\Omega}$,
    let $v(s, x) = \mathrm{e}^{- \lambda_k (t - s)} u_k(x)$.
    This function satisfies the following backward parabolic problem
    \begin{equation}\label{eq:backward parabolic}
        \left\{
            \begin{aligned}
                \partial_s v & = - \Delta v &&\text{ in } [0, t] \times \Omega \, , \\
                v(t, \cdot) & = u_k &&\text{ in } \Omega \, , \\
                \partial_n v & = 0 &&\text{ on } [0, t] \times \Gamma_\mathcal{N} \, , \\
                v & = 0 &&\text{ on } [0, t] \times \overline{\Gamma_\mathcal{D}} \, .
         \end{aligned}
        \right.
    \end{equation}
    Let~$\Sigma =  \overline{\Gamma_\mathcal{N}} \cap \overline{\Gamma_\mathcal{D}}$ denote the junction of the Dirichlet and Neumann boundaries.
    The part of $\partial \Omega$ that is not necessarily smooth is contained in $\Sigma$.
    The time dependence of $v$ reduces to the scalar factor~$\mathrm{e}^{-\lambda_k (t-s)}$,
    which is analytic on $[0, t]$,
    so that the regularity of $v$ on $[0, t] \times \overline{\Omega}$
    is that of $u_k$ on $\overline{\Omega}$,
    up to and including the final time $s = t$, at which $v(t, \cdot) = u_k$.
    In particular, \cref{lem:continuity} gives that
    $v$ is bounded and continuous on $[0, t] \times \overline{\Omega}$,
    and that $v \in C^\infty\bigl([0, t] \times
    \left(\overline{\Omega} \setminus \Sigma \right)\bigr)$.
    Let~$\Omega_{\delta} \subset \Omega $ denote the subset given by $\Omega _\delta = \Omega \setminus \mathcal{V}_\delta$,
    where~$\mathcal{V}_\delta = \Gamma_\mathcal{D} + B(0, \, \delta)$
    for a small $\delta > 0$.
    Notice that since $\Sigma \subset \subset \mathcal{V}_\delta$,
    the function $v$ is smooth on $[0, t] \times \overline{\Omega_\delta}$.
    Applying It\^o's lemma to the process $X_s^x$ defined in \eqref{eq:def process reflexion}
    with the deterministic initial condition~$X_0^x = x$,
    we obtain
    \begin{equation}\label{eq:ito}
        \begin{aligned}
            v(t \wedge \tau_{\mathcal{V}_\delta}, X_{t \wedge \tau_{\mathcal{V}_\delta}}^x) - v(0, x)
            & = \int_0^{t \wedge \tau_{\mathcal{V}_\delta}} \partial_r v(r, X_r^x) \, \d r
            + \int_0^{t \wedge \tau_{\mathcal{V}_\delta}} \nabla v(r, X_r^x) \cdot \d X_r
            + \int_0^{t \wedge \tau_{\mathcal{V}_\delta}} \Delta v(r, X_r^x) \, \d r \\
            & = \sqrt{2}\int_0^{t \wedge \tau_{\mathcal{V}_\delta}} \nabla v(r, X_r^x) \cdot \d W_r
            - \int_0^{t \wedge \tau_{\mathcal{V}_\delta}} \partial_n v(r, X_r^x) \mathbbm{1}_{X_r^x \in \partial \Omega} \, \d L_r,
        \end{aligned}
    \end{equation}
    with $\tau_{\mathcal{V}_\delta}$ defined in \eqref{eq:def stopping time}.
    Since $\nabla v$ is bounded on $\Omega_\delta$ and $v$ is smooth on $\overline{\Omega_\delta}$,
    we have that
    $$ \mathbb{E}_x \Bigl( \int_0^{t \wedge \tau_{\mathcal{V}_\delta}} \nabla v(r, X_r^x) \cdot \d W_r \Bigr) = 0.$$
    Additionally,
    the second term is also zero since,
    from the Neumann boundary condition in \eqref{eq:backward parabolic},
    we have~$\partial_n v(r, X_r^x) \mathbbm{1}_{X_r^x \in \partial \Omega} = 0$ for all $r < \tau_{\mathcal{V}_\delta}$.
    Hence,
    by taking the expectation in~\eqref{eq:ito},
    we obtain
    \[
        \mathbb{E}_x \left( v(t \wedge \tau_{\mathcal{V}_\delta}, X_{t \wedge \tau_{\mathcal{V}_\delta}}^x) \right)
        = v(0, x) = \mathrm{e}^{- \lambda_k t} u_k(x).
    \]
    To conclude,
    we decompose the expectation into two parts,
    corresponding to whether the process hits~$\mathcal{V}_\delta$ after time $t$ or not.
    Since $v(t, \cdot) = u_k$, this yields
    \begin{equation}\label{eq:decomp expectation}
        \begin{aligned}
        \mathrm{e}^{- \lambda_k t} u_k(x)
        & = \mathbb{E}_x \left(u_k(X_{t}^x) \, \mathbbm{1}_{t < \tau_{\mathcal{V}_\delta}} \right)
        + \mathbb{E}_x \left( v(\tau_{\mathcal{V}_\delta}, X_{\tau_{\mathcal{V}_\delta}}^x) \, \mathbbm{1}_{t \geq \tau_{\mathcal{V}_\delta} } \right).
    \end{aligned}\end{equation}
    From the continuity of the paths $(X_s^x)_{s \geq 0}$,
    the stopping time $\tau_{\mathcal{V}_\delta}$ is a non-increasing function of $\delta$
    that converges to $\tau_{\overline{\Gamma_\mathcal{D}}}$ when $\delta \to 0$.
    Since $u_k$ is bounded on $\overline{\Omega}$ by \cref{lem:continuity},
    the dominated convergence theorem yields
    $$  \mathbb{E}_x \left(u_k(X_t^x) \, \mathbbm{1}_{t < \tau_{\mathcal{V}_\delta}} \right)
        \xrightarrow[\delta \to 0]{}  \mathbb{E}_x \left(u_k(X_t^x) \, \mathbbm{1}_{t < \tau_{\overline{\Gamma_\mathcal{D}}}} \right).$$
    For the second term in \eqref{eq:decomp expectation},
    notice that $\left|v(s, y)\right| \leq \left|u_k(y)\right|$ for any $s \in [0, t]$ and $y \in \overline{\Omega}$ since $\lambda_k \geq 0$,
    and that $X_{\tau_{\mathcal{V}_\delta}}^x \in \overline{\mathcal{V}_\delta} \cap \overline{\Omega}$,
    so that
    \begin{equation}\label{eq:lasteqb2}
        \left| \mathbb{E}_x \left(v(\tau_{\mathcal{V}_\delta}, X_{\tau_{\mathcal{V}_\delta}}^x) \, \mathbbm{1}_{t \geq \tau_{\mathcal{V}_\delta} } \right) \right|
        \leq \sup_{y \in \overline{\mathcal{V}_\delta} \cap \overline{\Omega}} \left|u_k(y)\right|
        \xrightarrow[\delta \to 0]{} 0,
    \end{equation}
    where the convergence follows from the uniform continuity of $u_k$ on the compact set $\overline{\Omega}$,
    together with $u_k = 0$ on $\overline{\Gamma_\mathcal{D}}$.
    Passing to the limit $\delta \to 0$ in \eqref{eq:decomp expectation} concludes the proof.
\end{proof}

Let us now use \cref{thm:feynman-kac} to prove \cref{thm:QSD}.
\begin{proof}[Proof of \cref{thm:QSD}]
    Let $u_0$ be the first element of the orthonormal basis $(u_k)_{k \geq 0}$,
    associated with the smallest eigenvalue $\lambda_0$ of $\mathcal{L}$.
    By \cref{lem:sign eigenfunction}, the eigenvalue $\lambda_0$ is simple
    and $u_0$ has a constant sign in $\Omega$.
    Define
    \[
        \nu = \frac{u_0}{\int_\Omega u_0},
    \]
    which is the non-negative, $L^1$-normalized solution of \eqref{eq:elliptic eigen}
    associated with $\lambda_0$.
    With a slight abuse of notation,
    we also denote by $\nu$ the associated probability measure on $\Omega$.

    \textbf{Step 1: Identification of the law of the killed process.}
    Let $t > 0$ and denote by $q_t = q_t^\nu$ the density provided by \cref{lem:proba} for $\mu = \nu$.
    For any $k \geq 0$,
    integrating the Feynman--Kac formula of~\cref{thm:feynman-kac} with respect to $\nu(\mathrm{d}x)$
    and using the orthonormality of the family $(u_k)_{k \geq 0}$,
    we obtain
    \begin{equation}\label{eq:fourier coefficients qt}
        \bigl( q_t, \, u_k \bigr)_\Omega
        = \mathbb{E}_\nu \left[ u_k(X_t) \, \mathbbm{1}_{t < \tau_{\overline{\Gamma_\mathcal{D}}}} \right]
        = \mathrm{e}^{-\lambda_k t} \bigl( u_k, \, \nu \bigr)_\Omega
        = \mathrm{e}^{-\lambda_k t} \frac{\bigl( u_k, \, u_0 \bigr)_\Omega}{\int_\Omega u_0}
        = \mathrm{e}^{-\lambda_k t} \frac{\delta_{k 0}}{\int_\Omega u_0}.
    \end{equation}
    Since $q_t \in L^2(\Omega)$ by \cref{lem:proba},
    the function $q_t$ coincides with its expansion in the complete orthonormal system $(u_k)_{k \geq 0}$,
    so that \eqref{eq:fourier coefficients qt} yields
    \begin{equation}\label{eq:identification qt}
        q_t = \sum_{k \geq 0} \bigl( q_t, \, u_k \bigr)_\Omega u_k
        = \frac{\mathrm{e}^{-\lambda_0 t}}{\int_\Omega u_0} \, u_0
        = \mathrm{e}^{-\lambda_0 t} \, \nu
        \qquad \text{in } L^2(\Omega).
    \end{equation}

    \textbf{Step 2: $\nu$ is a quasi-stationary distribution.}
    Integrating \eqref{eq:identification qt} over $\Omega$ and using the normalization $\int_\Omega \nu = 1$,
    we obtain
    \begin{equation}\label{eq:exponential law}
        \mathbb{P}_\nu \bigl( t < \tau_{\overline{\Gamma_\mathcal{D}}} \bigr)
        = \int_\Omega q_t
        = \mathrm{e}^{-\lambda_0 t},
        \qquad t \geq 0.
    \end{equation}
    Consequently, for any $t > 0$,
    the law of $X_t$ conditioned on the event $\{ t < \tau_{\overline{\Gamma_\mathcal{D}}} \}$
    admits the density
    \[
        \frac{q_t}{\mathbb{P}_\nu \bigl( t < \tau_{\overline{\Gamma_\mathcal{D}}} \bigr)}
        = \frac{\mathrm{e}^{-\lambda_0 t} \nu}{\mathrm{e}^{-\lambda_0 t}}
        = \nu
    \]
    with respect to the Lebesgue measure.
    Two measures with the same density coincide,
    hence $\mathbb{P}_\nu \bigl( X_t \in \, \cdot \, \big| \, t < \tau_{\overline{\Gamma_\mathcal{D}}} \bigr) = \nu$
    for any $t > 0$,
    which shows that $\nu$ is a quasi-stationary distribution.

    \textbf{Step 3: Yaglom limit and uniqueness.}
    We now prove that $\nu$ is the so-called Yaglom limit of the process:
    for any probability measure $\mu$ on $\Omega$,
    \begin{equation}\label{eq:yaglom limit}
        \lim_{t \to +\infty} \; \sup_{A \subset \overline{\Omega}}
        \left| \mathbb{P}_\mu \bigl( X_t \in A \, \big| \, t < \tau_{\overline{\Gamma_\mathcal{D}}} \bigr) - \nu(A) \right| = 0,
    \end{equation}
    where the supremum is taken over all measurable sets $A \subset \overline{\Omega}$.
    Since this limit does not depend on $\mu$,
    this also implies that $\nu$ is the unique quasi-stationary distribution.
Let $\mu$ be a probability measure on $\Omega$
    and denote by $q_t^\mu$ the density provided by \cref{lem:proba}.
    Integrating the Feynman--Kac formula of \cref{thm:feynman-kac} with respect to $\mu(\mathrm{d}x)$ gives,
    for any $k \geq 0$ and any $t > 0$,
    \begin{equation}\label{eq:fourier coefficients qt mu}
        \bigl( q_t^\mu, \, u_k \bigr)_\Omega
        = \mathrm{e}^{-\lambda_k t} \int_\Omega u_k \, \mathrm{d} \mu,
    \end{equation}
    each term being well defined since $u_k \in C^0(\overline{\Omega})$ by \cref{lem:continuity}.
    Since $q_t^\mu \in L^2(\Omega)$ by \cref{lem:proba},
    it coincides with its expansion in the complete orthonormal system $(u_k)_{k \geq 0}$,
    so that \eqref{eq:fourier coefficients qt mu} yields
    \begin{equation}\label{eq:qt mu in serie form}
        q_t^\mu = \sum_{k \geq 0} \mathrm{e}^{-\lambda_k t} \left( \int_\Omega u_k \, \mathrm{d} \mu \right) u_k .
    \end{equation}
    Let us show that the tail of this series is exponentially small in $L^\infty(\Omega)$.
    Since the family $(u_k)_{k \geq 0}$ is orthonormal in $L^2(\Omega)$
    and $\mu$ is a probability measure on $\Omega$,
    the estimate \eqref{eq:bound L infty norm of uk} of \cref{lem:continuity} gives
    \begin{equation}\label{eq:bound term serie}
        \left\Vert \mathrm{e}^{-\lambda_k t} \left( \int_\Omega u_k \, \mathrm{d} \mu \right) u_k \right\Vert_{L^\infty(\Omega)}
        \leq \mathrm{e}^{-\lambda_k t} \left\Vert u_k \right\Vert_{C^0(\overline{\Omega})}^2
        \leq C \mathrm{e}^{-\lambda_k t } \lambda_k^2 \left( 1 + \lambda_k \right)^{2 m},
    \end{equation}
    with $C > 0$ independent of $k$, and $m > 0$ that only depends on $d$.
    By Weyl's law,
    we know that the $k$-th eigenvalue for the negative Laplacian with either Dirichlet or Neumann boundary conditions
    scales as~$k^\frac{2}{d}$ as~$k \to \infty$ (see a historical introduction in \cite{Weyl} or \cite[Section XIII.15]{ReedSimon}).
    As in~\cite[Proposition~2.3]{Tony2024},
    we can bound the eigenvalues of $\mathcal{L}$ from below by those of the Laplacian with Neumann boundary condition,
    and from above by those of the Laplacian with Dirichlet boundary condition.
    Therefore,
    there exist~$0 < c_- \leq c_+$ such that, for $k$ large enough,
    \begin{equation}\label{eq:weyl mixed}
        c_- k^{\frac{2}{d}} \leq \lambda_k \leq c_+ k^{\frac{2}{d}} .
    \end{equation}
    Fix $t_0 > 0$.
    For any $k \geq 1$ and any $t \geq t_0$,
    one has $\mathrm{e}^{-\lambda_k t} \leq \mathrm{e}^{-\lambda_1 (t-t_0)} \mathrm{e}^{-\lambda_k t_0}$ since $\lambda_k \geq \lambda_1$.
    Summing~\eqref{eq:bound term serie} over $k \geq 1$ therefore leads to
    \begin{equation}\label{eq:expansion qt mu}
        \left\Vert q_t^\mu - \mathrm{e}^{-\lambda_0 t} \left( \int_\Omega u_0 \, \mathrm{d} \mu \right) u_0 \right\Vert_{L^\infty(\Omega)}
        \leq C \mathrm{e}^{-\lambda_1 (t - t_0)} \sum_{k \geq 1} \mathrm{e}^{-\lambda_k t_0} \left( 1 + \lambda_k \right)^{2 m}
        = \mathrm{O}\left( \mathrm{e}^{-\lambda_1 t} \right),
    \end{equation}
    the series being convergent thanks to \eqref{eq:weyl mixed}.
    In particular, the series \eqref{eq:qt mu in serie form} converges normally in $L^\infty(\Omega)$ for $t \geq t_0$.

    Let us now show \eqref{eq:yaglom limit}.
    Notice that $\int_\Omega u_0 \, \mathrm{d} \mu > 0$ since $u_0 > 0$ in $\Omega$,
    and that $\lambda_1 > \lambda_0$ since $\lambda_0$ is simple by \cref{lem:sign eigenfunction}
    (see also \cref{lem:spectral gap} for a lower bound on $\lambda_1$ which is uniform in $\varepsilon$).
    Dividing \eqref{eq:expansion qt mu} by
    $\int_\Omega q_t^\mu = \mathbb{P}_\mu(t < \tau_{\overline{\Gamma_\mathcal{D}}})$
    and factoring out $\mathrm{e}^{-\lambda_0 t}$,
    we obtain
    \begin{equation}\label{eq:lambda0lambda1}
        \frac{q_t^\mu}{\displaystyle \int_\Omega q_t^\mu}
        = \frac{\displaystyle \left( \int_\Omega u_0 \, \mathrm{d} \mu \right) u_0
        + \mathrm{O}\left( \mathrm{e}^{-(\lambda_1 - \lambda_0) t} \right)}
        {\displaystyle \left( \int_\Omega u_0 \, \mathrm{d} \mu \right) \int_\Omega u_0
        + \mathrm{O}\left( \mathrm{e}^{-(\lambda_1 - \lambda_0) t} \right)}
        \xrightarrow[t \to +\infty]{} \nu
        \qquad \text{in } L^\infty(\Omega).
    \end{equation}
    This convergence also holds in $L^1(\Omega)$,
    hence, 
    using the definition of $q_t^\mu$ from \cref{lem:proba}, we obtain~\eqref{eq:yaglom limit}.
    Consequently, if $\widetilde{\nu}$ is a quasi-stationary distribution,
    then by definition the law of $X_t$ under $\mathbb{P}_{\widetilde{\nu}}$
    conditioned on $\{ t < \tau_{\overline{\Gamma_\mathcal{D}}} \}$ is equal to $\widetilde{\nu}$ for any $t > 0$,
    so that taking $\mu = \widetilde{\nu}$ in \eqref{eq:yaglom limit}
    and letting $t \to +\infty$ gives $\widetilde{\nu} = \nu$.

    Finally,
    $u_0$ is in $\mathcal{D}(\mathcal{L}) \cap C^0(\overline{\Omega})$ from \cref{lem:continuity},
    hence $\nu = u_0 / \int_\Omega u_0$ is also in $\mathcal{D}(\mathcal{L}) \cap C^0(\overline{\Omega})$.
\end{proof}
We conclude this appendix with the proof of \cref{cor:QSD}.
\begin{proof}[Proof of \cref{cor:QSD}]
    We show the points of the statement one by one in order.

    \paragraph{Distribution of the exit time.}
    This is a direct consequence of Step $2$ of the proof of \cref{thm:QSD}.
    Indeed, \eqref{eq:exponential law} states that, for any $t \geq 0$,
    \[
        \mathbb{P}_\nu(t < \tau_{\overline{\Gamma_\mathcal{D}}}) = \mathrm{e}^{-\lambda_0 t},
    \]
    so the first exit time is indeed exponentially distributed
    with parameter $\lambda_0$, the smallest eigenvalue of $\mathcal{L}$.

    \paragraph{Independence of $\tau_{\overline{\Gamma_\mathcal{D}}}$ and $X_{\tau_{\overline{\Gamma_\mathcal{D}}}}$.}
    Notice that for any $t \geq 0$ and any measurable set $A \subset \Gamma_\mathcal{D}$,
    \begin{align*}
        \mathbb{P}_\nu(t < \tau_{\overline{\Gamma_\mathcal{D}}}, X_{\tau_{\overline{\Gamma_\mathcal{D}}}} \in A)
        & = \mathbb{P}_\nu(t < \tau_{\overline{\Gamma_\mathcal{D}}})
        \mathbb{P}_\nu(X_{\tau_{\overline{\Gamma_\mathcal{D}}}} \in A \, | \, t < \tau_{\overline{\Gamma_\mathcal{D}}})\\
        & = \mathbb{P}_\nu(t < \tau_{\overline{\Gamma_\mathcal{D}}})
        \mathbb{P}_\nu(X_{\tau_{\overline{\Gamma_\mathcal{D}}}} \in A),
    \end{align*}
    where we used that
    $\mathbb{P}_\nu(X_{\tau_{\overline{\Gamma_\mathcal{D}}}} \in A \, | \, t < \tau_{\overline{\Gamma_\mathcal{D}}}) = \mathbb{P}_\nu(X_{\tau_{\overline{\Gamma_\mathcal{D}}}} \in A)$,
    which is a direct consequence of the definition of $\nu$ as a quasi-stationary distribution
    for the process~$(X_t)_{t \geq 0}$.

    \paragraph{Exit point distribution.}
    Consider $\phi \in C^0(\overline{\Gamma_\mathcal{D}}) \cap H^\frac{1}{2}(\Gamma_\mathcal{D})$,
    and consider the solution $v \in H^1(\Omega)$ of the following elliptic problem:
    \begin{equation} \label{eq:exit hole proba}
        \left\{
            \begin{aligned}
                - \Delta v(x) & = 0 && \text{ in } \Omega, \\
                \partial_n v(x) & = 0 && \text{ on } \Gamma_\mathcal{N}, \\
                v(x) & = \phi(x)
            && \text{ on } \Gamma_\mathcal{D}.
         \end{aligned}
        \right.
    \end{equation}
    Classical results (see e.g. \cite{Savare}) ensure the existence,
    uniqueness and regularity of the solution of \eqref{eq:exit hole proba},
    in particular $v \in C^\infty(\Omega_\delta) \cap C^0(\overline{\Omega})$,
    with $\Omega_\delta = \Omega \setminus \mathcal{V}_\delta$,
    where $\mathcal{V}_\delta = \Gamma_\mathcal{D} + B(0, \, \delta)$ for a small $\delta > 0$
    as in the proof of~\cref{thm:feynman-kac} (for the continuity of $v$ on $\overline{\Omega}$,
    see \cite[Theorem of Section 1]{MitreaMitrea} for creased domains,
    and \cite[Theorem 5.3 and Corollary 5.4]{shamir}).

    We first show a Feynman--Kac formula for this problem, i.e., for any $x \in \Omega$,
    \begin{equation}\label{eq:exit hole proba feynman-kac}\begin{aligned}
    v(x) & = \mathbb{E}_x \left( \phi(X_{\tau_{\overline{\Gamma_\mathcal{D}}}}) \right).
    \end{aligned}\end{equation}
    The arguments are the same as in the proof of the Feynman--Kac formula of \cref{thm:feynman-kac}.
    By the regularity $C^\infty(\Omega_\delta) \cap C^0(\overline{\Omega})$ of $v$,
    we can apply It\^o's lemma to the process $X_s^x$ defined in \eqref{eq:def process reflexion}
    with the deterministic initial condition~$X_0^x = x \in \Omega_\delta$,
    and obtain for any $t > 0$,
    \begin{equation}\label{eq:ito exit hole}
        v(X_{t \wedge \tau_{\mathcal{V}_\delta}}^x) - v(x) =
        \sqrt{2} \int_0^{t \wedge \tau_{\mathcal{V}_\delta}} \nabla v(X_s^x) \cdot \d W_s +
        \int_0^{t \wedge \tau_{\mathcal{V}_\delta}} \Delta v(X_s^x) \d s
        - \int_{0}^{{t \wedge \tau_{\mathcal{V}_\delta}}} \partial_n v(X_s^x) \mathbbm{1}_{X_{s}^x \in \partial \Omega} \d L_s.
    \end{equation}
    The first term is a martingale,
    since $v$ is smooth on $\Omega_\delta$.
    The second and the third terms of \eqref{eq:ito exit hole} are zero,
    since $v$ is a solution of \eqref{eq:exit hole proba}.
    By taking the limit $\delta \to 0$,
    justified as in the proof of \cref{thm:feynman-kac},
    we have that
    \[
        v(x) =
        \mathbb{E}_x \left( v(X_{t \wedge \tau_{\overline{\Gamma_\mathcal{D}}}}^x) \right).
    \]
    Taking the limit $t \to +\infty$ and using the dominated convergence theorem
    thanks to the continuity of~$v$ on~$\overline{\Omega}$,
    \[
        v(x) = \mathbb{E}_x \left[ v(X_{\tau_{\overline{\Gamma_\mathcal{D}}}}) \right] =
        \mathbb{E}_x \left[ \phi(X_{\tau_{\overline{\Gamma_\mathcal{D}}}}) \right] .
    \]
    This is the Feynman--Kac formula \eqref{eq:exit hole proba feynman-kac} for the problem \eqref{eq:exit hole proba}.

    Now notice that, since $\nu$ is an eigenvector associated with the eigenvalue $\lambda_0$ of the operator $\mathcal{L}$,
    and~${\nu \in \mathcal{D}(\mathcal{L}) \cap C^0(\overline{\Omega})}$,
    we have that
    \begin{equation} \label{eq:integral exit hole}\begin{aligned}
        \mathbb{E}_\nu \left[ \phi(X_{\tau_{\overline{\Gamma_\mathcal{D}}}}) \right]
        & = \int_\Omega \mathbb{E}_x \left[ \phi(X_{\tau_{\overline{\Gamma_\mathcal{D}}}}) \right]
        \nu(\mathrm{d}x) = \int_\Omega v(x) \nu(x) \, \mathrm{d}x
        = -\frac{1}{\lambda_0} \int_\Omega v(x) \Delta \nu(x) \, \mathrm{d} x.
    \end{aligned}\end{equation}
    Now, by the definition of the duality \eqref{eq:normal_derivative_weak_h12},
    \begin{equation}\label{eq:weak nu}
        \left\langle \partial_n \nu, \, v\right\rangle_{H^{-\frac{1}{2}}(\partial \Omega), H^{\frac{1}{2}}(\partial \Omega)}
        = \int_\Omega \nabla \nu(x) \cdot \nabla v(x) \, \mathrm{d} x + \int_\Omega \Delta \nu(x) v(x) \,\mathrm{d} x.
    \end{equation}
    Similarly, since $\partial_n v \in H^{-\frac{1}{2}}(\partial \Omega)$ and $\nu \in H^{1}(\Omega)$,
    one gets
    \begin{equation}\label{eq:weak v}
         \left\langle \partial_n v, \, \nu\right\rangle_{H^{-\frac{1}{2}}(\partial \Omega), H^{\frac{1}{2}}(\partial \Omega)}
        = \int_\Omega \nabla \nu(x) \cdot \nabla v(x)  \mathrm{d} x + \int_\Omega \Delta v(x) \nu(x) \mathrm{d} x.
    \end{equation}
    Given that $\partial_n v = 0$ on $\Gamma_\mathcal{N}$, $\Delta v = 0$ from \eqref{eq:exit hole proba},
    and that $\nu = 0$ on $\overline{\Gamma_\mathcal{D}}$ from \eqref{eq:domain L} and the continuity of~$\nu$ from \cref{thm:QSD},
    it follows from \eqref{eq:weak v} that
    \[
        \int_\Omega \nabla \nu(x) \cdot \nabla v(x)  \mathrm{d} x = 0.
    \]
    Substituting this into \eqref{eq:weak nu} and rearranging
    leads to
    \begin{equation}\label{eq:weak nu 2}\begin{aligned}
        \int_\Omega \Delta \nu(x) v(x) \mathrm{d} x
        & = \left\langle \partial_n \nu, \, v\right\rangle_{H^{-\frac{1}{2}}(\partial \Omega), H^{\frac{1}{2}}(\partial \Omega)}\\
        & = \left\langle \partial_n \nu, \, \phi \right\rangle_{H^{-\frac{1}{2}}(\partial \Omega), H^{\frac{1}{2}}(\partial \Omega)}
        + \left\langle \partial_n \nu, \, v - \phi \right\rangle_{H^{-\frac{1}{2}}(\partial \Omega), H^{\frac{1}{2}}(\partial \Omega)},
    \end{aligned}\end{equation}
    where $\phi$ denotes any $C^0(\partial \Omega)$ extension of $\phi$ from $\overline{\Gamma_\mathcal{D}}$ to $\partial \Omega$ in $C^0(\partial \Omega)$.
    Notice that the duality pairing between $\partial_n \nu$ and $\phi$ does not depend 
    on the choice of the extension by the continuity of $\phi$ and \eqref{eq:meaning dn nu zero}.
    The last term in \eqref{eq:weak nu 2} is zero since $v - \phi = 0$ on $\overline{\Gamma_\mathcal{D}}$, 
    thus $v - \phi \in H^{\frac{1}{2}}_{00}(\Gamma_{\mathcal N})$.
    Hence,
    using \eqref{eq:weak nu 2} and \eqref{eq:integral exit hole}, we obtain
    \begin{equation}\label{eq:integral exit hole 2}
        \mathbb{E}_\nu \left[ \phi(X_{\tau_{\overline{\Gamma_\mathcal{D}}}}) \right]
        = -\frac{1}{\lambda_0} \left\langle \partial_n \nu, \, \phi\right\rangle_{H^{-\frac{1}{2}}(\partial \Omega
        ), H^{\frac{1}{2}}(\partial \Omega)}.
    \end{equation}

    It remains to show that $\partial_n \nu$ is a Radon measure and compute its mass.
    For this,
    we distinguish between the creased domain case, and the smooth domain case.

    \paragraph{The case of a creased domain.}
    In this situation,
    from \cite[Proposition 32]{TonyBoris2022}, used in \cite[Lemma 1.4]{Tony2024},
    and derived by \cite{jakab2009, Mitrea2},
    the normal derivative $\partial_n \nu$ is a function in $L^2(\partial \Omega)$.
    Hence, $\partial_n \nu$ is a Radon measure supported on $\overline{\Gamma_\mathcal{D}}$,
    absolutely continuous with respect to the surface measure $\sigma$ on $\partial \Omega$.

    \paragraph{The case of a smooth domain.}
    In this situation,
    the idea of the proof is to consider $-\frac{1}{\lambda_0}\partial_n \nu$
    as a distribution on $\partial \Omega$ since the boundary is smooth.
    Then,
    it can be shown that this distribution is positive (see \cite[Proposition A.1]{narrow2d} for the details),
    and hence it is a Radon measure (see this classical result in \cite[Theorem 3.18]{Duistermaat}).\\

    Hence, 
    in both cases, 
    we showed that for any $\phi \in C^0(\overline{\Gamma_\mathcal{D}}) \cap H^\frac{1}{2}(\Gamma_\mathcal{D})$,
    \begin{equation}\label{eq:integral exit hole 3}
        \mathbb{E}_\nu \left[ \phi(X_{\tau_{\overline{\Gamma_\mathcal{D}}}}) \right]
        = - \frac{1}{\lambda_0} \int_{\Gamma_\mathcal{D}} \phi \, \mathrm{d} \left(\partial_n \nu\right).
    \end{equation}
    From the continuity of the integral, the formula \eqref{eq:integral exit hole 3} can be extended to any $\phi \in C^0(\overline{\Gamma_\mathcal{D}})$
    using the density of $C^0(\overline{\Gamma_\mathcal{D}}) \cap H^\frac{1}{2}(\Gamma_\mathcal{D})$ in $C^0(\overline{\Gamma_\mathcal{D}})$.

    Since $\nu \in \mathcal{D}(\mathcal{L})$
    it is clear that the support of $\partial_n \nu$ is included in $\overline{\Gamma_\mathcal{D}}$.
    Thanks to the $L^1$ normalization condition, its total mass is given by
    \[
        - \frac{1}{\lambda_0} \partial_n \nu(\overline{\Gamma_\mathcal{D}})
        = - \frac{1}{\lambda_0} \left\langle \partial_n \nu, \, 1\right\rangle_{H^{-\frac{1}{2}}(\partial \Omega), H^{\frac{1}{2}}(\partial \Omega)}
        = - \frac{1}{\lambda_0} \int_\Omega \Delta \nu(x) \, \mathrm{d} x
        = - \frac{1}{\lambda_0} \int_\Omega -\lambda_0 \nu(x) \, \mathrm{d} x
        = 1.
    \]
    Finally, the claimed formula for the exit hole is obtained with \eqref{eq:integral exit hole 2},
    using $\phi = \mathbbm{1}_{\overline{\Gamma_k}}$,
    which is indeed continuous on $\overline{\Gamma_\mathcal{D}}$ as the holes are disjoint.
\end{proof}

\printbibliography
\end{document}